\documentclass[reqno]{amsart}

\usepackage{amssymb,latexsym,mathrsfs,amsmath}
\usepackage{amsthm}
\usepackage{graphicx}
\usepackage{diagbox}
\usepackage{mathrsfs}
\usepackage{hyperref}
\usepackage{tikz}
\usetikzlibrary{arrows}

\makeatletter
\@namedef{subjclassname@2020}{%
  \textup{2020} Mathematics Subject Classification}
\makeatother

\newtheorem{lemma}{Lemma}[section]
\newtheorem{proposition}[lemma]{Proposition}
\newtheorem{theorem}[lemma]{Theorem}
\newtheorem{corollary}[lemma]{Corollary}
\newtheorem{conjecture}[lemma]{Conjecture}

\theoremstyle{definition}
\newtheorem{remark}[lemma]{Remark}
\newtheorem{definition}[lemma]{Definition}
\newtheorem{example}[lemma]{Example}

\DeclareMathOperator{\id}{id}
\DeclareMathOperator{\ir}{ir}

\DeclareMathOperator{\im}{im}

\DeclareMathOperator{\Kh}{Kh}
\DeclareMathOperator{\CKh}{CKh}
\DeclareMathOperator{\rCKh}{\widetilde{CKh}}
\DeclareMathOperator{\rKh}{\widetilde{Kh}}
\DeclareMathOperator{\rCBLT}{\widetilde{C}_{BLT}}
\DeclareMathOperator{\BN}{BN}

\newcommand{\Z}{\mathbb{Z}}

\newcommand{\Q}{\mathbb{Q}}

\newcommand{\F}{\mathbb{F}}

\newcommand{\tomega}{\tilde{\omega}}
\newcommand{\talpha}{\tilde{\alpha}}
\newcommand{\tcal}[1]{\tilde{\mathcal{#1}}}
\newcommand{\tbar}[1]{\bar{\mathcal{#1}}}

\newcommand{\CBN}{C_{\BN}}

\newcommand{\Cob}{\mathcal{C}ob^{\Z[h]}_{\bullet/l}}
\newcommand{\cob}{\mathcal{C}ob_\bullet}

\newcommand{\plane}[2]{
\shade[color = gray!40, opacity = 0.3] (#1,#2-0.5) -- (#1+1,#2) -- (#1+1,#2+1.5) -- (#1,#2+1) -- (#1,#2-0.5);
\draw (#1,#2-0.5) -- (#1+1,#2) -- (#1+1,#2+1.5) -- (#1,#2+1) -- (#1,#2-0.5);
}

\newcommand{\sphere}[3]{
\shade[ball color = gray!40, opacity = 0.3] ({#1},{#2}) circle ({#3});
\draw (#1,#2) circle ({#3});
\draw (#1-#3,#2) arc (180:360:#3 and 0.3*#3);
\draw[dashed] (#1+#3,#2) arc (0:180:#3 and 0.3*#3);
}

\newcommand{\spheredot}[3]{
\sphere{#1}{#2}{#3}
\node at (#1-0.3*#3,#2+0.45*#3) [scale = 1.5*#3] {$\bullet$};
}

\newcommand{\bowl}[4]{
\shade[ball color = gray!40, opacity = 0.3] (#1+0.8*#3,#2) arc (0:180:0.4*#3 and -0.2*#3) -- (#1,#2-#4*#3) arc (180:360:0.4*#3 and 0.4*#3) -- (#1+0.8*#3,#2);
\draw (#1+0.8*#3,#2) arc (0:180:0.4*#3 and -0.2*#3) -- (#1,#2-#4*#3) arc (180:360:0.4*#3 and 0.4*#3) -- (#1+0.8*#3,#2);
\draw (#1,#2) arc (180:360:0.4*#3 and -0.2*#3);
\shade[ball color = gray!40, opacity = 0.15] (#1,#2) arc (180:360:0.4*#3 and -0.2*#3) arc (0:180:0.4*#3 and -0.2*#3);
}

\newcommand{\bowlud}[4]{
\shade[ball color = gray!40, opacity = 0.3] (#1,#2) arc (180:360:0.4*#3 and 0.2*#3) -- (#1+0.8*#3,#2+#4*#3) arc (0:180:0.4*#3 and 0.4*#3) -- (#1,#2);
\draw (#1,#2) arc (180:360:0.4*#3 and 0.2*#3) -- (#1+0.8*#3,#2+#4*#3) arc (0:180:0.4*#3 and 0.4*#3) -- (#1,#2);
\draw[dashed] (#1+0.8*#3,#2) arc (0:180:0.4*#3 and 0.2*#3);
}

\newcommand{\cylinder}[4]{
\shade[ball color = gray!40, opacity = 0.3] (#1,#2) arc (180:360:0.4*#3 and 0.2*#3) -- (#1+0.8*#3,#2+#4*#3) 
arc (0:180:0.4*#3 and -0.2*#3);
\shade[ball color = gray!40, opacity = 0.15] (#1,#2+#4*#3) arc (180:360:0.4*#3 and -0.2*#3) arc (0:180:0.4*#3 and -0.2*#3);
\draw (#1,#2) arc (180:360:0.4*#3 and 0.2*#3) -- (#1+0.8*#3,#2+#4*#3) 
arc (0:180:0.4*#3 and 0.2*#3) -- (#1,#2);
\draw[dashed] (#1+0.8*#3,#2) arc (0:180:0.4*#3 and 0.2*#3);
\draw (#1,#2+#4*#3) arc (180:360:0.4*#3 and 0.2*#3);
}

\newcommand{\drawblack}[4] {
\draw[->] (#1) -- node [above, scale = 0.7, #3] {$#4$} (#2);
}
\newcommand{\drawblackw}[4] {
\draw[-, line width = 6pt, color = white] (#1) -- (#2);
\draw[->] (#1) -- node [above, scale = 0.7, #3] {$#4$} (#2);
}
\newcommand{\drawgray}[2] {
\draw[->, color = lightgray] (#1) -- (#2);
}
\newcommand{\drawgrayw}[2] {
\draw[-, line width = 6pt, color = white] (#1) -- (#2);
\draw[->, color = lightgray] (#1) -- (#2);
}

\newcommand{\leftcl}[4]{
\draw[#4] (#1,#2) to [out = 270, in = 270] (#1 - 0.2 * #3, #2) -- ((#1 - 0.2 * #3, #2+#3) to [out = 90, in = 90] (#1, #2+#3);
}

\newcommand{\smomega}[5]{
\draw[#5] (#1,#2) -- (#1, #2+#4);
\draw[#5] (#1+0.5 * #3, #2) -- (#1+0.5 * #3, #2+#4);
\draw[#5] (#1+#3, #2) -- (#1+#3, #2+#4);
}

\newcommand{\smalpha}[5]{
\draw[#5] (#1,#2) -- (#1, #2 + 0.2 *#4) to [out = 90, in = 90] (#1+0.5*#3, #2 + 0.2 *#4) -- (#1+0.5*#3, #2);
\draw[#5] (#1, #2+#4) -- (#1, #2+ 0.8 * #4) to [out = 270, in = 270] (#1+0.5*#3, #2 + 0.8 *#4) -- (#1+0.5*#3, #2+#4);
\draw[#5] (#1+#3, #2) -- (#1+#3, #2+#4);
}

\newcommand{\smbeta}[5]{
\draw[#5] (#1,#2) -- (#1, #2+#4);
\draw[#5] (#1+0.5 * #3, #2) -- (#1+0.5 * #3, #2 + 0.2 *#4) to [out = 90, in = 90] (#1+#3, #2 + 0.2 *#4) -- (#1+#3, #2);
\draw[#5] (#1+0.5 * #3, #2+#4) -- (#1+0.5 * #3, #2+ 0.8 * #4) to [out = 270, in = 270] (#1+#3, #2 + 0.8 *#4) -- (#1+#3, #2+#4);
}

\newcommand{\smdelta}[5] {
\draw[#5] (#1, #2) -- (#1, #2+ 0.2 * #4) to [out = 90, in = 270] (#1+#3, #2+0.8 * #4) -- (#1+#3, #2+#4);
\draw[#5] (#1, #2+#4) -- (#1, #2+ 0.8 * #4) to [out = 270, in = 270] (#1+0.5*#3, #2 + 0.8 *#4) -- (#1+0.5*#3, #2+#4);
\draw[#5] (#1+0.5 * #3, #2) -- (#1+0.5 * #3, #2 + 0.2 *#4) to [out = 90, in = 90] (#1+#3, #2 + 0.2 *#4) -- (#1+#3, #2);
}

\newcommand{\smgamma}[5] {
\draw[#5] (#1, #2+#4) -- (#1, #2+0.8*#4) to [out = 270, in = 90] (#1+#3, #2+0.2 * #4) -- (#1+#3, #2);
\draw[#5] (#1,#2) -- (#1, #2 + 0.2 *#4) to [out = 90, in = 90] (#1+0.5*#3, #2 + 0.2 *#4) -- (#1+0.5*#3, #2);
\draw[#5] (#1+0.5 * #3, #2+#4) -- (#1+0.5 * #3, #2+ 0.8 * #4) to [out = 270, in = 270] (#1+#3, #2 + 0.8 *#4) -- (#1+#3, #2+#4);
}

\newcommand{\smeps}[5] {
\draw[#5] (#1+0.5*#3, #2+#4) -- (#1+0.5*#3, #2+0.8*#4) to [out = 270, in = 90] (#1+#3, #2+0.2 * #4) -- (#1+#3, #2);
\draw[#5] (#1,#2) -- (#1, #2 + 0.2 *#4) to [out = 90, in = 90] (#1+0.5*#3, #2 + 0.2 *#4) -- (#1+0.5*#3, #2);
}

\newcommand{\smepsd}[5] {
\draw[#5] (#1+0.5*#3, #2) -- (#1+0.5*#3, #2+ 0.2 * #4) to [out = 90, in = 270] (#1+#3, #2+0.8 * #4) -- (#1+#3, #2+#4);
\draw[#5] (#1, #2+#4) -- (#1, #2+ 0.8 * #4) to [out = 270, in = 270] (#1+0.5*#3, #2 + 0.8 *#4) -- (#1+0.5*#3, #2+#4);
}

\newcommand{\smzet}[5] {
\draw[#5] (#1, #2) -- (#1, #2+ 0.2 * #4) to [out = 90, in = 270] (#1+0.5*#3, #2+0.8 * #4) -- (#1+0.5*#3, #2+#4);
\draw[#5] (#1+0.5 * #3, #2) -- (#1+0.5 * #3, #2 + 0.2 *#4) to [out = 90, in = 90] (#1+#3, #2 + 0.2 *#4) -- (#1+#3, #2);
}

\newcommand{\smzetd}[5] {
\draw[#5] (#1, #2+#4) -- (#1, #2+0.8*#4) to [out = 270, in = 90] (#1+0.5*#3, #2+0.2 * #4) -- (#1+0.5*#3, #2);
\draw[#5] (#1+0.5 * #3, #2+#4) -- (#1+0.5 * #3, #2+ 0.8 * #4) to [out = 270, in = 270] (#1+#3, #2 + 0.8 *#4) -- (#1+#3, #2+#4);
}

\newcommand{\smcup}[4] {
\draw[#4] (#1, #2+0.7*#3) -- (#1, #2+0.5*#3) to [out = 270, in = 270] (#1+0.5*#3, #2+0.5*#3) -- (#1+0.5*#3, #2+0.7*#3);
}

\newcommand{\smcap}[4] {
\draw[#4] (#1,#2) -- (#1, #2 + 0.2 *#3) to [out = 90, in = 90] (#1+0.5*#3, #2 + 0.2 *#3) -- (#1+0.5*#3, #2);
}

\newcommand{\tmomega}[4]{
\draw[#4] (#1,#2) to [out = 90, in = 270] (#1+0.15*#3, #2+0.5*#3) to [out = 90, in = 270] (#1, #2+#3);
\draw[#4] (#1+#3, #2) to [out = 90, in = 270] (#1+0.85*#3, #2+0.5*#3) to [out = 90, in = 270] (#1+#3, #2+#3);
}

\newcommand{\tmalpha}[4]{
\draw[#4] (#1, #2+#3) -- (#1, #2+0.95*#3) to [out = 270, in = 270] (#1+#3, #2+0.95*#3) -- (#1+#3, #2+#3);
\draw[#4] (#1,#2) -- (#1, #2 + 0.05 *#3) to [out = 90, in = 90] (#1+#3, #2 + 0.05 *#3) -- (#1+#3, #2);
}

\newcommand{\dmomega}[4]{
\draw[#4] (#1,#2) to [out = 90, in = 270] (#1+0.15*#3, #2+0.5*#3) to [out = 90, in = 270] (#1, #2+#3);
}

\newcommand{\crossing}[4] {
\draw[#4](#1,#2) -- (#1+#3,#2+#3);
\draw[#4](#1+#3, #2) -- (#1+0.6 * #3, #2+0.4 * #3);
\draw[#4](#1,#2+#3) -- (#1+0.4 * #3, #2 + 0.6 * #3);
}

\newcommand{\mcrossing}[4] {
\draw[#4](#1+#3,#2) -- (#1,#2+#3);
\draw[#4](#1+#3, #2+#3) -- (#1+0.6 * #3, #2+0.6 * #3);
\draw[#4](#1,#2) -- (#1+0.4 * #3, #2 + 0.4 * #3);
}

\begin{document}
\parindent0em
\setlength\parskip{.1cm}
\thispagestyle{empty}
\title{On the Khovanov homology of 3-braids}
\author[Dirk Sch\"utz]{Dirk Sch\"utz}
\address{Department of Mathematical Sciences\\ Durham University\\ Durham DH1 3LE\\ United Kingdom}
\email{dirk.schuetz@durham.ac.uk}

\begin {abstract}
We prove the conjecture of Przytycki and Sazdanovi\'c that the Khovanov homology of the closure of a 3-stranded braid only contains torsion of order 2. This conjecture has been known for six out of seven classes in the Murasugi-classification of 3-braids and we show it for the remaining class. Our proof also works for the other classes and relies on Bar-Natan's version of Khovanov homology for tangles as well as his delooping and cancellation techniques, and the reduced integral Bar-Natan--Lee--Turner spectral sequence. We also show that the Knight-move conjecture holds for 3-braids.
\end{abstract}

\maketitle

\section{Introduction}
In the last two decades, Khovanov homology has become an indispensable tool in knot theory. Despite this, the occurrence of torsion is still very mysterious. For some classes of knots it is now understood that only torsion of order $2$ can appear, see \cite{MR4407084}, but other torsion orders occur, for example in torus knots, compare \cite{MR2320156}. Particularly for torus knots there seems to be a relation to the number of strands in the minimal braid representation.  

Based on computations, Przytycki and Sazdanovi\'c \cite{MR3205574} made several conjectures for torsion in the Khovanov homology of closures of braids. However, counterexamples have been found for most, see \cite{MR3894728, MR4477421}. The remaining conjecture involves $3$-braids and can be stated as follows.

\begin{conjecture}[Przytycki--Sazdanovi\'c \cite{MR3205574}]\label{con:przysazd}
The Khovanov homology of a closed $3$-braid can only have torsion of order $2$.
\end{conjecture}

Murasugi \cite{MR356023} listed seven sets with conjugacy classes for words in the braid group $B_3$. For the first four sets Conjecture \ref{con:przysazd} was shown to hold in \cite{MR4430925}, and for the next two sets it was shown to hold in \cite{kelo2024discrete}. This leaves only one set in the Murasugi classification, and we show that the conjecture also holds for those braids.

\begin{theorem}\label{thm:maintorsion}
{\em Conjecture \ref{con:przysazd}} is true.
\end{theorem} 

Our techniques are different from the ones used in \cite{MR4430925} and \cite{kelo2024discrete}, they are mostly based on the delooping and cancellation techniques developed by Bar-Natan for tangle diagrams, see \cite{MR2174270, MR2320156}. Indeed, we have to apply these techniques to all of Murasugi's sets in order to derive the conjecture for the last set. Furthermore, for the first six sets we get a nice decomposition result for the Khovanov cochain complexes which can be used to simply read off the Khovanov homology for these links. For the last set we do not quite get such a nice decomposition, and we have to resort to additional techniques involving the reduced integral Bar-Natan--Lee--Turner spectral sequence.

We can also say something about the free part in the Khovanov homology, namely, we can prove the {\em Knight Move Conjecture} for closures of $3$-strand braids.

\begin{conjecture}[Knight Move Conjecture \cite{MR1917056}]\label{con:knight_move}
Given a knot $K$, its Khovanov homology over $\Q$ is the direct sum of a single \em pawn move \em piece
\[
q^{s-1}\Q \oplus q^{s+1}\Q
\]
for an even number $s$, and several \em knight move \em pieces
\[
u^iq^j \Q \oplus u^{i+1}q^{j+4}\Q
\]
for $i,j\in \Z$, $j$ odd.
\end{conjecture}

The Knight Move Conjecture is known to be false \cite{MR4042864} in general, but it still holds for some classes of knots (such as quasi-alternating knots \cite{MR2509750} and knots with unknotting number less than $3$ \cite{MR3894208}), and most knots for which calculations have been done.

One can also state it over other fields, although counterexamples are then easier to come by. In characteristic $2$ one also allows pieces $u^iq^j\F_2\oplus u^{i+1}q^{j+2}\F_2$ to avoid triviality. It can be extended to links by allowing more pawn moves ($2^{c-1}$-many, if the link has $c$ components).

The Knight Move Conjecture holds whenever the Lee spectral sequence degenerates after the first page and this is a typical strategy to prove it. Indeed, for many knots, particularly with a small number of crossings, the second differential in the Lee spectral sequence $d_2$ of bidegree $(1, 8)$ is $0$ simply because at least one of the two groups $\Kh^{i,j}(K)$ and $\Kh^{i+1, j+8}(K)$ is always $0$. There exist $3$-braids for which both groups can be non-zero (in fact, one can get non-zero groups in bidegrees $(i,j)$ and $(i+1,j+4n)$ for arbitrarily large $n$), yet we show that the Lee spectral sequence collapses after the first page.

\begin{theorem}\label{thm:main_knight}
Let $\F$ be a field of characteristic different from $2$, and $L$ the closure of a $3$-braid.  Then the Lee differential $d_n$ for $L$ of bidegree $(1,4n)$ is $0$ for $n\geq 2$. In particular, the Knight Move Conjecture over $\F$ holds for $L$.
\end{theorem}

In characteristic $2$ we have the Bar-Natan--Turner differential $d_n$ of bidegree $(1,2n)$, and it is $0$ for $n\geq 3$. In particular, the appropriately adapted Knight Move Conjecture for $\F_2$ also holds for $3$-braids.

\section{Generalities on 3-braids}

A word $w$ in letters $\{a,a^{-1},b,b^{-1}\}$ gives rise to a tangle diagram $T_w$ by using the diagrams in Figure \ref{fig:basic_tangle} and stacking them on top of each other. We note that we picture the resulting braid to move from bottom to top, with $a$ and $b$ leading to positive crossings (and $a^{-1}$, $b^{-1}$ leading to negative crossings) when all strands are oriented in the same direction. This agrees with \cite{MR356023}, if we identify $a$ with $\sigma_1^{-1}$ and $b$ with $\sigma_2^{-1}$.

\begin{figure}[ht]
\begin{tikzpicture}
\crossing{-6}{0.6}{0.6}{very thick}
\draw[very thick] (-4.8,0.6) -- (-4.8,1.2);
\draw[very thick] (-3, 0.6) -- (-3, 1.2);
\crossing{-2.4}{0.6}{0.6}{very thick}
\crossing{0}{0}{0.6}{very thick}
\draw[very thick] (1.2,0) -- (1.2,0.6);
\crossing{0.6}{0.6}{0.6}{very thick}
\mcrossing{0.6}{1.2}{0.6}{very thick}
\draw[very thick] (0,0.6) -- (0, 1.8);
\end{tikzpicture}
\caption{\label{fig:basic_tangle}The tangles for $a$, $b$, and $abb^{-1}$.}
\end{figure}

Two such words $w,w'$ represent the same tangle if and only if they represent the same element in the braid group $B_3 = \langle\, a,b\mid aba = bab\,\rangle$. Furthermore, conjugate elements  represent the same link.

Murasugi \cite[Prop.2.1]{MR356023} gave normal forms for the conjugacy classes in $B_3$, in the form of seven sets. To describe these sets, let us first introduce the notion of an alternating word.

\begin{definition}
A word $w$ in $\{a^{-1}, b\}$ is called an {\em alternating word}. If it starts with $a^{-1}$ and ends with $b$, it is called a {\em proper alternating word}.
\end{definition}

It is easy to see that alternating words give rise to alternating links, although if we use at most one letter, we get a split link with at least one unknot component. We could also use words in the letters $a$ and $b^{-1}$ to get alternating links, but these are conjugate to alternating words. 

Murasugi's normal forms for conjugacy classes are now given by
\begin{align*}
\Omega_0 &= \{ (ab)^{3k}\mid k\in \Z\},\\
\Omega_1 &= \{ (ab)^{3k+1}\mid k\in \Z\},\\
\Omega_2 &= \{ (ab)^{3k+2}\mid k\in \Z\},\\
\Omega_3 &= \{ (ab)^{3k+1}a\mid k\in \Z\},\\
\Omega_4 &= \{ (ab)^{3k} a^{-l}\mid k,l\in \Z, l>0\},\\
\Omega_5 &= \{ (ab)^{3k} b^l \mid k,l\in \Z, l>0\},\\
\Omega_6 &= \{(ab)^{3k} w \mid k\in \Z, w \mbox{ a proper alternating word.}\}.
\end{align*}

Inverting these normal forms keep $\Omega_0$, $\Omega_3$, and $\Omega_6$ invariant, while flipping $\Omega_1$ with $\Omega_2$, and $\Omega_4$ with $\Omega_5$. Since Khovanov homology behaves well with respect to mirroring, we will only be considering the cases $k\geq 1$ and write $\Omega^+_i$ for the corresponding sets. The cases with $k=0$ are mostly trivial and will be treated separately.

\section{The Bar-Natan complex for links and tangles}

A (commutative) \em Frobenius system \em is a tuple $\mathcal{F}=(R,A,\varepsilon, \Delta)$ with $A$ a commutative ring and a subring $R$, $\varepsilon\colon A\to R$ an $R$-module map, $\Delta\colon A\to A\otimes_R A$ an $A$-bimodule map that is co-associative and co-commutative, such that $(\varepsilon\otimes \id)\circ \Delta = \id$.

Given a Frobenius system $\mathcal{F}= (R,A,\varepsilon,\Delta)$ such that $A$ is free of rank $2$ over $R$, Khovanov  \cite{MR2232858} showed that for a link diagram $D$ one can define a cochain complex $C(D;\mathcal{F})$ over $R$ whose homology is a link invariant.

Khovanov homology can now be defined using the Frobenius system $\mathcal{F}_{\Kh}$ where $R=\Z$, $A=\Z[X]/(X^2)$, $\varepsilon\colon \Z[X]/(X^2)\to \Z$ sends $1$ to $0$ and $X$ to $1$, and $\Delta(1) = 1\otimes X+X\otimes 1$. The ring $\Z[X]/(X^2)$ has a grading defined by $|1| = 0$ and $|X| = -2$. This grading can be used to give the \em Khovanov complex of $D$\em, $\CKh(D;\Z) = C(D;\mathcal{F}_{\Kh})$ a second grading, called the $q$-grading. We sometimes write $\CKh^{i,j}(D;\Z)$ with $i$ referring to the homological grading and $j$ to the $q$-grading.

Another Frobenius system $\mathcal{F}_{\mathrm{BN}}$, named after Bar-Natan who considered it in characteristic 2 \cite{MR2174270}, is given by $R=\Z[h]$, $A= \Z[X,h]/(X^2-Xh)$, $\varepsilon(1) = 0$, $\varepsilon(X) = 1$, and $\Delta(1) = X\otimes 1 + 1 \otimes X - h\otimes 1$. Again, $\Z[X,h]/(X^2-Xh)$ can be graded using $|1| = 0$, $|h| = -2 = |X|$, leading to a bigraded cochain complex $\CBN(D;\Z[h]) = C(D;\mathcal{F}_{\mathrm{BN}})$.

The Khovanov complex can be recovered from the Bar-Natan complex using the ring homomorphism $\eta\colon \Z[h]\to \Z$ which sends $h$ to $0$. That is,
\[
\CKh(D;\Z) = \CBN(D;\Z[h])\otimes_{\Z[h]}\Z,
\]
where $h$ acts as $0$ on $\Z$. We denote Khovanov homology by $\Kh^{i,j}(L;\Z)$ with $i$ the homological degree and $j$ the $q$-degree.

If the link diagram has a chosen basepoint, the Bar-Natan complex $\CBN(D;\Z[h])$ has an $A$-action corresponding to multiplication by $X$ on the copy of $A$ corresponding to the basepoint. We can think of $\CBN(D;\Z[h])$ as a cochain complex over the category $\mathfrak{Mod}^q_{\Z[X,h]/(X^2-Xh)}$ of finitely generated free graded modules over $\Z[X,h]/(X^2-Xh)$.

If $M$ is an object of $\mathfrak{Mod}^q_{\Z[X,h]/(X^2-Xh)}$, that is, a finitely generated free graded $\Z[X,h]/(X^2-Xh)$-module, and $j\in \Z$, let $q^j M$ be the object of $\mathfrak{Mod}^q_{\Z[X,h]/(X^2-Xh)}$ whose underlying module is the same as $M$, but with a grading shift that adds $j$ to the $q$-grading. It will be convenient for us to let $A= q\Z[X,h]/(X^2-Xh)$, that is, $|1| = 1$, and $|h| = -1 = |X|$, when viewed as elements of $A$.

For a link diagram $D$ with a basepoint we now define the reduced Khovanov complex as
\[
\rCKh(D;\Z) = q^{-1}\CBN(D;\Z[h])\otimes_{\Z[X,h]/(X^2-Xh)} \Z,
\]
where $X$ and $h$ act on $\Z$ as $0$. The resulting homology is again bigraded and a link invariant, denoted by $\rKh^{i,j}(L;\Z)$.

If we let $h$ act on $\Z$ as $1$ and $X$ as $0$, we get another complex, the {\em reduced Bar-Natan--Lee--Turner complex} $\rCBLT(D;\Z)$. This complex is no longer graded, but it has a filtration
\[
0 \subset \cdots \subset F_{2j} \subset F_{2j-2} \subset \cdots \subset \rCBLT(D;\Z),
\]
such that $F_{2j}/F_{2j+2} \cong \rCKh^{\ast, 2j}(D;\Z)$. In particular, there is a spectral sequence, the reduced BLT-spectral sequence $E^{i,j}_k$ with $E^{i,j}_1 = \rKh^{i,j}(L;\Z)$ which converges to the cohomology of $\rCBLT(L;\Z)$.

Our proof of Conjecture \ref{con:przysazd} and Conjecture \ref{con:knight_move} relies on the following proposition.

\begin{proposition}\label{prp:concludeconjectures}
Let $L$ be a link such that for the reduced BLT-spectral sequence we have that $E^{i,j}_k$ is free abelian for $k = 1,2$ and all $i,j\in \Z$. Then $\Kh^{i,j}(L;\Z)$ has only torsion of order $2$.  Furthermore, if $E^{i,j}_3 = E^{i,j}_\infty$ is also free abelian, the Lee spectral sequence collapses at page $2$ and $L$ satisfies the Knight Move Conjecture in every characteristic.
\end{proposition}

\begin{proof}
Let $d_1\colon E^{i,j}_1\to E^{i+1,j+2}_1$ be the boundary homomorphism that calculates $E^{i,j}_2$. Since all $E_1^{i,j}$ are free abelian, we can find a basis of them so that $d_1$ is represented by a matrix in Smith Normal form. Furthermore, since all of the $E^{i,j}_2$ are free abelian, the Smith Normal form only has entries $0$ and $1$.

Consider the long exact sequence relating reduced and unreduced Khovanov homology
\[
\cdots \longrightarrow \rKh^{i-1,j}(L;\Z) \stackrel{\beta}{\longrightarrow} \rKh^{i,j+2}(L;\Z) \longrightarrow \Kh^{i,j+1}(L;\Z) \longrightarrow \rKh^{i,j}(L;\Z) \stackrel{\beta}{\longrightarrow}\cdots
\]
Since the reduced Khovanov homology is free abelian, we get
\[
\Kh^{i,j+1}(L;\Z) \cong \rKh^{i,j+2}(L;\Z)/\im \beta \oplus \ker \beta.
\]
By \cite[Lm.5.6]{MR4873797} $\beta = \pm 2d_1$, so $\Kh^{i,j+1}(L;\Z)$ only has torsion of order $2$.

Since all the pages in the integral reduced BLT-spectral sequence are free abelian, we also get that the reduced BLT-spectral sequence over any field collapses at page $3$. In characteristic $2$ this implies the modified Knight Move Conjecture for this characteristic. In characteristic different from $2$ this implies that the Lee spectral sequence collapses at page 2 by \cite{schuetz2024extortion}, which in turn implies the Knight Move Conjecture for this characteristic.
\end{proof}

If $C$ is a cochain complex over $\mathfrak{Mod}^q_{\Z[X,h]/(X^2-Xh)}$, and $i,j\in \Z$, we write $u^iq^jC$ for the cochain complex over $\mathfrak{Mod}^q_{\Z[X,h]/(X^2-Xh)}$ such that $u^iq^j C^{u,v} = C^{u-i,v-j}$. 

\begin{definition}
Let $j\geq 1$. The cochain complex $A(j)$ over $\Z[X,h]/(X^2-Xh)$ is defined by
\[
A(j)^0 = A, \hspace{0.5cm}A(j)^1 = q^{2j}A, \hspace{0.5cm} \partial(a) = (2X-h)^j\cdot a.
\]
and $0$ in all other homological degrees.
\end{definition}

It is worth noting that for $j = 2i$ we have $(2X-h)^j = h^{2i}$. If we work over a field $\F$ instead of $\Z$, the Bar-Natan complex $\CBN(D;\F[h])$ decomposes into a direct sum of suitably shifted copies of $A(j)\otimes\F$ and copies of $A\otimes\F$, see \cite{schuetz2024extortion}. Over $\Z$ this is too much too expect in general, but we will see that this holds with $j\leq 2$ for all braid words in $\Omega_i$ with $i\leq 5$. It is an interesting question whether this also holds for $i=6$.

\begin{corollary}\label{cor:finalstep}
Let $L$ be a link with diagram $D$ such that $\CBN(D;\Z[h])$ decomposes into a direct sum of cochain complexes, each of which is either a shifted single copy of $A$, or a shifted copy of $A(j)$ with $j=1$ or $2$. Then $L$ only has torsion of order $2$, the Lee spectral sequence collapses at page $2$, and $L$ satisfies the Knight Move Conjecture in every characteristic.
\end{corollary} 

\begin{proof}
It is easy to see that the conditions of Proposition \ref{prp:concludeconjectures} are satisfied for the complexes $A$, $A(1)$ and $A(2)$.
\end{proof}

Now assume we have a tangle diagram $T$ in a planar disc $D$ with $2k$ marked endpoints on $\partial D$. In \cite{MR2174270}, Bar-Natan constructs a cochain complex $\Kh(T)$ over an appropriate additive category which is a tangle invariant up to chain homotopy. There are in fact several variations and we will describe the one suitable for our considerations.

Let $D_{2k}$ be a disc with $2k$ points on its boundary. We denote by $\cob(D_{2k})$ the category whose objects are pairs $(S,j)$, where $S$ is a smooth compact $1$-manifold embedded in $D$ such that $S$ intersects $\partial D$ transversally in $\partial S$, and this intersection agrees with the $2k$ points. Also, $j\in \Z$. Morphisms between $(S_0,j_0)$ and $(S_1,j_1)$ are oriented cobordisms $C$ embedded in $D_{2k}\times [0,1]$ which admit finitely many dots and which fix the $2k$ points on $\partial D_{2k}$, up to boundary preserving isotopy. We require $C\cap D_{2k}\times \{i\} = S_i$ for $i=0,1$, and 
\[
\chi(C) - \chi(S_0) = j_0 - j_1 + 2\# \{\mbox{dots on }C\}.
\]
We now denote by $\Cob(D_{2k})$ the additive category whose objects are finitely generated based free $\Z[h]$-modules where basis elements are objects $(S,j)$ from $\cob(D_{2k})$. Morphisms are given by matrices $(M_{nm})$ with each matrix entry $M_{nm}$ an element of the free $\Z[h]$-module generated by the morphism set between objects $(S_n,j_n)$ and $(S_m,j_m)$ in $\cob(D_{2k})$, modulo the following local relations:
\begin{equation}\label{eq:localrels}
\begin{tikzpicture}[baseline={([yshift=-.5ex]current bounding box.center)}]
\sphere{0}{0}{0.5}
\node at (1.1,0) {$= 0$,};
\spheredot{3}{0}{0.5}
\node at (4.1,0) {$= 1$,};
\plane{5}{-0.5}
\plane{7.1}{-0.5}
\node at (5.5,0.2) [scale = 0.75] {$\bullet$};
\node at (5.5,-0.2) [scale = 0.75] {$\bullet$};
\node at (7.6,0) [scale = 0.75] {$\bullet$};
\node at (6.6,0) {$= h$};
\end{tikzpicture}
\end{equation}
and
\begin{equation}\label{eq:localrelb}
\begin{tikzpicture}[baseline={([yshift=-.5ex]current bounding box.center)}]
\cylinder{0}{0}{1}{1.5}
\node at (1.25,0.75) {$=$};
\bowl{1.6}{1.5}{1}{0.1}
\bowlud{1.6}{0}{1}{0.1}
\node at (2,0.1) [scale = 0.75] {$\bullet$};
\node at (2.85,0.75) {$+$};
\bowl{3.2}{1.5}{1}{0.1}
\bowlud{3.2}{0}{1}{0.1}
\node at (3.6,1.2) [scale = 0.75] {$\bullet$};
\node at (4.45,0.75) {$-\,h$};
\bowl{4.8}{1.5}{1}{0.1}
\bowlud{4.8}{0}{1}{0.1}
\end{tikzpicture}
\end{equation}
Here multiplication by $h$ has the same effect on the $q$-grading as a dot.

Bar Natan \cite{MR2174270} showed that these categories have excellent gluing properties. Since we are only interested in tangles arising from braids, we do not need the full power of their gluing behaviour, and instead we settle for a simpler version.

Let $n,m$ be non-negative integers with $n+m$ even, and let $D^n_m$ be a rectangle with $n$ points on the top boundary and $m$ points on the bottom boundary. Bar-Natan \cite{MR2174270} describes the gluing operation as a functor 
\[
D\colon \Cob(D^n_m)\times \Cob(D^p_n)\to \Cob(D^p_m)
\]
which we are going to write as a tensor product. Furthermore, there are closing operations
\[
C_L, C_R \colon \Cob(D^n_m)\to \Cob(D^{n-1}_{m-1}),
\]
which can be described as tensor product with an object that is an interval between the two left-most points (in case of $C_L$) or the two right-most points (in case of $C_R$) on the top and bottom of $D^n_m$.

Recall the tangle diagrams $T_a$, $T_b$ for the letters $a$ and $b$ from Figure \ref{fig:basic_tangle}. The {\em Bar-Natan complexes} $\CBN(T_a;\Z[h])$ and $\CBN(T_b;\Z[h])$ are then given by the cochain complexes over $\Cob(D^3_3)$ given as
\[
\begin{tikzpicture}
\node at (0,1) {$\CBN(T_a;\Z[h])=q$};
\smomega{1.6}{0.7}{0.6}{0.6}{very thick}
\drawblack{2.4,1}{3,1}{}{S}
\node at (3.4,1) {$q^2$};
\smalpha{3.7}{0.7}{0.6}{0.6}{very thick}
\node at (5.9,1) {, $\CBN(T_b;\Z[h])=q$};
\smomega{7.6}{0.7}{0.6}{0.6}{very thick}
\drawblack{8.4,1}{9,1}{}{S}
\node at (9.4,1) {$q^2$};
\smbeta{9.7}{0.7}{0.6}{0.6}{very thick}
\end{tikzpicture}
\]
where we write $q^jS$ for the object $(S,j)$, and the morphism $S$ is the standard saddle cobordism. The objects are in homological degrees $0$ and $1$. The complexes for $T_{a^{-1}}$ and $T_{b^{-1}}$ are obtained by dualizing, which involves negating homological and $q$-degrees.

The general Bar-Natan complex $\CBN(T_w;\Z[h])$ for a braid word $w$ is then obtained from these four basic complexes using the tensor product operation of Bar-Natan.

Closed circles in a smoothing $S$ can be removed using the {\em Delooping} operation, see Bar-Natan \cite{MR2320156} and Naot \cite[Prop.5.1]{MR2263052} for the generality required here.

\begin{lemma}[Delooping]\label{lm:deloop}
Let $S$ be a compact $1$-manifold in $D_{2k}$ which has a circle component $C$, and let $S' = S - C$. Then $S$, viewed as an object in $\Cob(D_{2k})$ is isomorphic to $q S'\oplus q^{-1} S'$. \hfill\qedsymbol
\end{lemma}

Let us denote this isomorphism by $\Phi\colon S \to qS'\oplus q^{-1}S'$. This is a $2\times 1$-matrix with entries $\Phi_+$ and $\Phi_-$ given by
\[
\begin{tikzpicture}
\node at (0,0) {$\Phi_+ = $};
\bowlud{0.6}{-0.1}{0.5}{0.1}
\node at (0.8,0.05) {$\boldsymbol{\cdot}$};
\node at (1.4,0) {$-\,h$};
\bowlud{1.8}{-0.1}{0.5}{0.1}
\node at (4,0) {$\Phi_- = $};
\bowlud{4.6}{-0.1}{0.5}{0.1}
\end{tikzpicture}
\]
where the cap represents death of the circle as a cobordism. Similarly, the inverse isomorphism is a $1\times 2$-matrix with entries a birth and a dotted birth.

Because of delooping, every object in $\Cob(D^0_0)$ is isomorphic to a direct sum of objects with empty $1$-manifold. Furthermore, cobordisms between empty sets are closed surfaces, which can be simplified using the relations (\ref{eq:localrels}) and (\ref{eq:localrelb}) until the cobordism is empty as well. It follows that $\Cob(D^0_0)$ can be identified with the category $\mathfrak{Mod}^q_{\Z[h]}$ of finitely generated free graded modules over $\Z[h]$. Similarly, the category $\Cob(D^1_1)$ can be identified with the category $\mathfrak{Mod}^q_{\Z[X,h]/(X^2-Xh)}$ of finitely generated free graded modules over $\Z[X,h]/(X^2-Xh)$.

We also need Bar-Natan's Gaussian elimination in cochain complexes, see \cite[Lm.4.2]{MR2320156}.

\begin{lemma}[Gaussian Elimination]\label{lm:gausselim}
Let $\mathfrak{C}$ be an additive category and $(C,c)$ be a cochain complex over $\mathfrak{C}$ such that $C^n=A^n\oplus B^n$, $C^{n+1}=A^{n+1}\oplus B^{n+1}$, and the coboundary $c_k\colon C^k\to C^{k+1}$ for $k=n-1,n,n+1$ are represented by matrices
\[
c_{n-1} = \begin{pmatrix} \alpha \\ \beta \end{pmatrix}, \hspace{0.4cm}
c_n=\begin{pmatrix} \varphi & \delta \\ \gamma & \varepsilon \end{pmatrix},\hspace{0.4cm}
c_{n+1} = \begin{pmatrix} \mu & \nu \end{pmatrix},
\]
with $\varphi\colon A^n \to A^{n+1}$ an isomorphism. Then $C$ is chain homotopy equivalent to a cochain complex $(D,d)$ with $D^k=C^k$ for $k\not=n,n+1$, $D^k=B^k$ for $k=n,n+1$, $d_k=c_k$ for $k\not=n-1,n,n+1$, and
\[
\pushQED{\qed}
d_{n-1} = \beta, \hspace{0.4cm} d_n = \varepsilon - \gamma\varphi^{-1}\delta, \hspace{0.4cm} d_{n+1} = \nu. \qedhere
\]
\end{lemma}

A good example how this works in practice can be found in \cite[\S 6]{MR2320156}. Since we work over $\Z[h]$, we give an explicit example here, which will also highlight all the techniques that we are going to need further below in this paper.

\begin{example}\label{ex:contract}
Consider the following, infinitely generated, cochain complex $C_\infty$ over $\Cob(D^2_2)$. We set
\[
\begin{tikzpicture}
\node at (0,0) {$C^0_\infty =$};
\tmomega{0.6}{-0.3}{0.6}{very thick}
\node at (4,0) {$C^k_\infty = \, q^{2k-1}$};
\tmalpha{5}{-0.3}{0.6}{very thick}
\node at (6.7,0) {for $k\geq 1$};
\end{tikzpicture}
\]
and let the boundary be given by $d_0\colon C^0_\infty\to C^1_\infty$ as the surgery, while for $k\geq 1$ we set
\[
\begin{tikzpicture}
\node at (0,0) {$d_{2k-1} = $};
\node at (1.6,0) {$-$};
\tmalpha{0.8}{-0.2}{0.4}{thick}
\tmalpha{2}{-0.2}{0.4}{thick}
\node at (1,-0.075) {$\bullet$};
\node at (2.2, 0.075) {$\bullet$};
\node at (5,0) {$d_{2k} = $};
\node at (6.6, 0) {$+$};
\tmalpha{5.8}{-0.2}{0.4}{thick}
\tmalpha{7}{-0.2}{0.4}{thick}
\node at (6,-0.075) {$\bullet$};
\node at (7.2, 0.075) {$\bullet$};
\node at (7.9,0) {$- \,\,h$};
\end{tikzpicture}
\]
We claim that
\[
\begin{tikzpicture}
\node at (0,0) {$C_\infty\,\otimes $};
\tmalpha{0.6}{-0.2}{0.4}{thick}
\node at (1.4,0) {$\simeq 0.$};
\end{tikzpicture}
\]
To see this note that the tensor product turns it into a cochain complex
\[
\begin{tikzpicture}
\tmomega{0}{0}{0.6}{very thick}
\tmalpha{0}{0.6}{0.6}{very thick}
\tmalpha{3.5}{0}{0.6}{very thick}
\tmalpha{3.5}{0.6}{0.6}{very thick}
\tmalpha{7}{0}{0.6}{very thick}
\tmalpha{7}{0.6}{0.6}{very thick}
\tmalpha{10.5}{0}{0.6}{very thick}
\tmalpha{10.5}{0.6}{0.6}{very thick}
\node at (3.2,0.6) {$q$};
\node at (6.7, 0.6) {$q^3$};
\node at (10.2, 0.6) {$q^5$};
\node at (11.6,0.6) {$\cdots$};
\drawblack{0.8,0.6}{3,0.6}{ }{S}
\drawblack{4.3,0.6}{6.4,0.6}{ }{ }
\drawblack{7.8,0.6}{9.9,0.6}{ }{ }
\tmalpha{4.8}{0.7}{0.3}{thick}
\tmalpha{4.8}{1}{0.3}{thick}
\tmalpha{5.7}{0.7}{0.3}{thick}
\tmalpha{5.7}{1}{0.3}{thick}
\node[scale = 0.7] at (4.95,0.79) {$\bullet$};
\node[scale = 0.7] at (5.85,0.9) {$\bullet$};
\node[scale = 0.7]  at (5.4, 1) {$-$};
\tmalpha{8}{0.7}{0.3}{thick}
\tmalpha{8}{1}{0.3}{thick}
\tmalpha{8.9}{0.7}{0.3}{thick}
\tmalpha{8.9}{1}{0.3}{thick}
\node[scale = 0.7] at (8.15,0.79) {$\bullet$};
\node[scale = 0.7] at (9.05,0.9) {$\bullet$};
\node[scale = 0.7]  at (8.6, 1) {$+$};
\node[scale = 0.7]  at (9.55, 1) {$-\,\,h$};
\end{tikzpicture}
\]
This is then delooped to
\[
\begin{tikzpicture}
\tmalpha{0}{0.7}{0.6}{very thick}
\tmalpha{3.5}{0.7}{0.6}{very thick}
\tmalpha{3.5}{-0.3}{0.6}{very thick}
\tmalpha{7}{0.7}{0.6}{very thick}
\tmalpha{7}{-0.3}{0.6}{very thick}
\tmalpha{10.5}{0.7}{0.6}{very thick}
\tmalpha{10.5}{-0.3}{0.6}{very thick}
\node at (3.2,1) {$q^2$};
\node at (6.7, 1) {$q^4$};
\node at (10.2, 1) {$q^6$};
\node at (6.7, 0) {$q^2$};
\node at (10.2, 0) {$q^4$};
\node at (11.6, 0.5){$\cdots$};
\drawblack{0.8,1}{3,1}{ }{\bullet \hspace{4pt}-h}
\drawblack{0.8,0.9}{3, 0}{sloped}{\id}
\smcap{1.525}{1.035}{0.4}{thick}
\drawblack{4.3, 1}{6.4, 1}{ }{\bullet}
\drawblack{4.3, 0.9}{6.4, 0.1}{sloped}{-\id}
\drawblack{4.3, 0}{6.4, 0}{ }{\bullet \hspace{4pt}-h}
\smcap{5.245}{1.035}{0.4}{thick}
\smcap{4.975}{0.035}{0.4}{thick}
\drawblack{7.8, 0}{9.9, 0}{ }{\bullet}
\drawblack{7.8, 0.9}{9.9, 0.1}{sloped}{\id}
\drawblack{7.8, 1}{9.9, 1}{ }{\bullet \hspace{4pt}-h}
\smcap{8.745}{0.035}{0.4}{thick}
\smcap{8.475}{1.035}{0.4}{thick}
\end{tikzpicture}
\]
To see that we get these particular morphisms, notice that a split surgery followed by death of that new circle is the identity, so the first two morphisms have this form because of the Delooping Lemma \ref{lm:deloop} using $\Phi$. The next morphisms require us to first birth a circle using $\Phi^{-1}$ from Lemma \ref{lm:deloop}, then dot, then remove the circle. Notice that if the birth of the circle is followed by dotting this component, composition with $\Phi_+$ is $0$. Similarly, if the birth is not followed by a dot, composing with $\Phi_-$ is $0$. Also, two dots after the birth followed by $\Phi_-$ is multiplication by $h$. The remaining morphisms are similar. The reader may want to check that the boundary composed with itself is indeed $0$.

We can now use the Gaussian Elimination Lemma \ref{lm:gausselim} to all of the diagonal identities to get the required chain homotopy contraction.

If we let $C_{n,\infty}$ for $n\geq 0$ be the subcomplex generated by all objects of homological degree greater than $n$, we get that 
\[
\begin{tikzpicture}
\node at (0,0) {$C_\infty/C_{n,\infty} \,\otimes $};
\tmalpha{1}{-0.2}{0.4}{thick}
\node at (2,0) {$\simeq q^{2n}$};
\tmalpha{2.5}{-0.2}{0.4}{thick}
\end{tikzpicture}
\]
This is because we can still do the Gaussian eliminations, but after $n$ cancellations, we are left with one object.
\end{example}

We now show how to get the cochain complex over $\Cob(D^2_2)$ for the tangle $T_n$ corresponding to the braid word $a^n$ for $n\geq 1$. This is well known, a version of this already appeared in \cite[\S 6.2]{MR1740682}, but we prove it mostly to showcase the techniques.

\begin{proposition}\label{prp:toruslink2}
Let $n\geq 1$ and $T_n$ the tangle corresponding to the braid word $a^n$. Then $\CBN(T_n;\Z[h])$ is chain homotopy equivalent to the complex $q^nC_\infty/C_{n,\infty}$.
\end{proposition}

\begin{proof}
Let us write $C_n = C_\infty/C_{n,\infty}$. We need to show that $C_n\otimes C_1\simeq C_{n+1}$ for all $n\geq 1$. We do this by induction, but leave the induction start to the reader. For $n\geq 2$ the complex $C_n\otimes C_1$ ends in
\[
\begin{tikzpicture}
\tmalpha{0}{1.5}{0.6}{very thick}
\tmomega{0}{2.1}{0.6}{very thick}
\tmalpha{3.5}{1.5}{0.6}{very thick}
\tmomega{3.5}{2.1}{0.6}{very thick}
\tmalpha{3.5}{0}{0.6}{very thick}
\tmalpha{3.5}{0.6}{0.6}{very thick}
\tmalpha{7}{0}{0.6}{very thick}
\tmalpha{7}{0.6}{0.6}{very thick}
\node at (-0.5,2.1) {$q^{2n-3}$};
\node at (3, 2.1) {$q^{2n-1}$};
\node at (3, 0.6) {$q^{2n-2}$};
\node at (6.6, 0.6) {$q^{2n}$};
\drawblack{0.8, 2.1}{2.5,2.1}{ }{d_{n-1}}
\drawblack{0.8, 2}{2.5, 0.6}{sloped}{(-1)^{n-1}S}
\drawblack{4.3, 2.1}{6.2, 0.7}{sloped}{(-1)^nS}
\drawblack{4.3, 0.6}{6.2, 0.6}{ }{\tilde{d}_{n-1}}
\end{tikzpicture}
\]
After delooping and $n-1$ Gaussian eliminations as in Example \ref{ex:contract} this turns into
\[
\begin{tikzpicture}
\tmalpha{0}{1.5}{0.6}{very thick}
\tmalpha{3.5}{1.5}{0.6}{very thick}
\tmalpha{3.5}{0}{0.6}{very thick}
\tmalpha{7}{1.5}{0.6}{very thick}
\tmalpha{7}{0}{0.6}{very thick}
\node at (-0.5,1.8) {$q^{2n-3}$};
\node at (3, 1.8) {$q^{2n-1}$};
\node at (3, 0.3) {$q^{2n-1}$};
\node at (6.5, 1.8) {$q^{2n+1}$};
\node at (6.5, 0.3) {$q^{2n-1}$};
\drawblack{0.8,1.8}{2.5, 1.8}{ }{d_{n-1}}
\drawgray{0.8, 1.7}{2.5, 0.4}
\drawblack{4.3, 1.8}{6, 1.8}{ }{(-1)^n(\,\,\bullet\hspace{3pt}-\,h)}
\drawblack{4.3, 0.3}{6, 0.3}{ }{f}
\drawblack{4.3, 0.4}{6, 1.7}{sloped, near start}{g}
\drawblackw{4.3, 1.7}{6, 0.4}{sloped, near end}{(-1)^n\id}
\smcup{5.155}{1.9}{0.3}{thick}
\end{tikzpicture}
\]
We do not need to work out the gray morphism, as it will not survive the next Gaussian elimination, but we need to work out $f$ and $g$.

The morphism $f$ is obtained by birthing a circle, then performing $\tilde{d}_{n-1}$, then performing death on the circle. Only the part of $\tilde{d}_{n-1}$ that puts a dot on the circle leads to a sphere with a dot on it. Hence $f = (-1)^{n-1}\id$.

For $g$, notice that putting a dot on the circle the summand of  $\tilde{d}_{n-1}$ leads to $0$ after composition with $\Phi_+$. For the other summand (or summands if $n$ is even) only the dotted death composes to something non-zero. Therefore
\[
\begin{tikzpicture}
\node at (0,0) {$g=\hspace{3pt}\bullet$\hspace{1cm}$n$ even,};
\node at (5,0) {$g=\hspace{3pt}\bullet \hspace{6pt}-\,\,h$\hspace{1cm}$n$ odd.};
\smcap{-0.825}{-0.15}{0.6}{thick}
\smcap{3.75}{-0.185}{0.6}{thick}
\end{tikzpicture}
\]
If we do Gaussian elimination along $f$, the morphism between the remaining objects is obtained from the previous morphism by adding $g$. If $n$ is odd, we get the difference of dottings with two $h$ cancelling each other, and if $n$ is even, we get the sum of dottings minus $h$.
\end{proof}

For $3$-braids, the relevant quotient category is $\Cob(D^3_3)$. After delooping, there are only five possible smoothings, and for simplicity we use greek letters as in Figure \ref{fig:cheatsheet}, which also lists several objects over $D^1_3, D^3_1$, and $D_2^2$, as well as several common morphisms.

\begin{figure}[ht]
\begin{center}
\begin{tikzpicture}
\node at (0,0) {$\omega =$};
\smomega{0.5}{-0.3}{0.6}{0.6}{very thick}
\node at (2,0) {$\alpha =$};
\smalpha{2.5}{-0.3}{0.6}{0.6}{very thick}
\node at (4,0) {$\beta =$};
\smbeta{4.5}{-0.3}{0.6}{0.6}{very thick}
\node at (6,0) {$\gamma =$};
\smgamma{6.5}{-0.3}{0.6}{0.6}{very thick}
\node at (8,0) {$\delta =$};
\smdelta{8.6}{-0.3}{0.6}{0.6}{very thick}
\end{tikzpicture}
\\[0.2cm]
\begin{tikzpicture}
\node at (0,0) {$\varepsilon =$};
\smeps{0.5}{-0.3}{0.6}{0.6}{very thick}
\node at (2,0) {$\varepsilon^\ast = $};
\smepsd{2.6}{-0.3}{0.6}{0.6}{very thick}
\node at (5,0) {$\zeta =$};
\smzet{5.5}{-0.3}{0.6}{0.6}{very thick}
\node at (7,0) {$\zeta^\ast =$};
\smzetd{7.6}{-0.3}{0.6}{0.6}{very thick}
\end{tikzpicture}
\\[0.2cm]
\begin{tikzpicture}
\node at (0,0) {$\tomega = $};
\node at (4,0) {$\talpha = $};
\tmomega{0.5}{-0.3}{0.6}{very thick}
\tmalpha{4.5}{-0.3}{0.6}{very thick}
\end{tikzpicture}
\\[0.2cm]
\begin{tikzpicture}
\node at (0.5,1) {$\omega$};
\node at (2.3,2) {$\alpha$};
\node at (2.3,0) {$\beta$};
\node at (4.3,2) {$\gamma$};
\node at (4.3,0) {$\delta$};
\smomega{6}{0.7}{0.6}{0.6}{very thick}
\smalpha{8}{1.7}{0.6}{0.6}{very thick}
\smbeta{8}{-0.3}{0.6}{0.6}{very thick}
\smgamma{10}{1.7}{0.6}{0.6}{very thick}
\smdelta{10}{-0.3}{0.6}{0.6}{very thick}
\draw (0.8,1.1) -- node [above, sloped, scale = 0.7] {$S$} (2,1.9);
\draw (0.8, 0.9) -- node [above, sloped, scale = 0.7] {$S$} (2, 0.1);
\draw (2.6, 2) -- node [above, scale = 0.7] {$S$} (4, 2);
\draw (2.6, 0) -- node [above, scale = 0.7] {$S$} (4, 0);
\draw (2.3, 0.3) -- node [right, scale = 0.7] {$D$} (2.3, 1.7);
\draw (4.3, 0.3) -- node [right, scale = 0.7] {$D$} (4.3, 1.7);
\draw (2.6, 0.3) -- node [above, sloped, near end, scale = 0.7] {$S$} (4, 1.7);
\draw[-, line width=6pt, color = white] (2.6, 1.7) -- (4, 0.3);
\draw (2.6, 1.7) -- node [above, sloped, near end, scale = 0.7] {$S$} (4, 0.3);
\draw (6.8,1.1) -- node [above, sloped, scale = 0.7] {$S$} (7.8,1.9);
\draw (6.8, 0.9) -- node [above, sloped, scale = 0.7] {$S$} (7.8, 0.1);
\draw (8.8, 2) -- node [above, scale = 0.7] {$S$} (9.8, 2);
\draw (8.8, 0) -- node [above, scale = 0.7] {$S$} (9.8, 0);
\draw (8.3, 0.4) -- node [right, scale = 0.7] {$D$} (8.3, 1.6);
\draw (10.3, 0.4) -- node [right, scale = 0.7] {$D$} (10.3, 1.6);
\draw (8.8, 0.5) -- node [above, sloped, near end, scale = 0.7] {$S$} (9.8, 1.5);
\draw[-, line width=6pt, color = white] (8.8, 1.5) -- (9.8, 0.5);
\draw (8.8, 1.5) -- node [above, sloped, near end, scale = 0.7] {$S$} (9.8, 0.5);
\end{tikzpicture}
\\[0.2cm]
\begin{tikzpicture}
\node at (-3,0.2) {$c = $};
\smcap{-2.6}{0}{0.6}{thick}
\node at (-2,0.2) {$-$};
\smcup{-1.7}{0}{0.6}{thick}
\node at (-2.45,0.2) {$\bullet$};
\node at (-1.55, 0.2) {$\bullet$};
\node at (-1.2, 0) {,};
\node at (0,0.2) {$d = $};
\smcap{0.5}{0}{0.6}{thick}
\node at (1.1,0.2) {$+$};
\smcup{1.4}{0}{0.6}{thick}
\node at (2.2, 0.2) {$-\,h,$};
\node at (0.65,0.2) {$\bullet$};
\node at (1.55, 0.2) {$\bullet$};
\node at (4,0.2) {$e = 2 $};
\tmomega{4.5}{0}{0.4}{thick}
\node at (4.55, 0.2) {$\bullet$};
\node at (5.4,0.2) {$-\,h.$};
\end{tikzpicture}
\end{center}
\caption{\label{fig:cheatsheet}The standard smoothings and their morphisms. The letter $S$ stands for a surgery, and $D$ for two surgeries. The morphisms $c,d$, and $e$ have the same domain and codomain, which can be $\alpha,\beta,\gamma,\delta$, or $\talpha$ for $c,d$, and $\tomega$ for $e$.}
\end{figure}

With our convention for the tensor product we get
\[
\alpha = \varepsilon\otimes \varepsilon^\ast, \hspace{0.3cm} \beta = \zeta\otimes \zeta^\ast, \hspace{0.3cm} \gamma = \varepsilon\otimes \zeta^\ast, \hspace{0.3cm} \delta = \zeta\otimes\varepsilon^\ast.
\]
Notice that $\omega$ works as a unit where applicable. After delooping we also get
\begin{equation}\label{eq:tensor_greeks}
\begin{alignedat}{2}
\mbox{ }\alpha\otimes \varepsilon &\cong q\varepsilon \oplus q^{-1}\varepsilon & \hspace{2cm} & \alpha\otimes \zeta \cong \varepsilon \\
\beta\otimes \varepsilon &\cong \zeta & & \beta\otimes \zeta \cong q\zeta \oplus q^{-1}\zeta \\
\gamma\otimes \varepsilon &\cong \varepsilon  & & \gamma\otimes\zeta \cong q\varepsilon \oplus q^{-1}\varepsilon \\
\delta\otimes\varepsilon &\cong q\zeta\oplus q^{-1}\zeta & & \delta\otimes\zeta \cong \zeta.
\end{alignedat}
\end{equation}

\section{The Bar-Natan complex for torus tangles on three strands}

We denote by $\mathcal{B}$ the following cochain complex over $\Cob(D^3_3)$, where we omit the homological degrees to save space:
\[
\begin{tikzpicture}
\node at (0,3.6) {$\omega$};
\node at (2,4.4) {$q\alpha$};
\node at (2,2.8) {$q\beta$};
\node at (4,4.4) {$q^2\gamma$};
\node at (4,2.8) {$q^2\delta$};
\node at (6,4.4) {$q^4\gamma$};
\node at (6,2.8) {$q^4\delta$};
\node at (8,4.4) {$q^5\alpha$};
\node at (8,2.8) {$q^5\beta$};
\node at (10,4.4) {$q^7\alpha$};
\node at (10,2.8) {$q^7\beta$};
\node at (11,3.6) {$\cdots$};
\draw[->] (0.3, 3.7) -- node [above, scale = 0.7, sloped] {$S$} (1.6, 4.4);
\draw[->] (0.3, 3.5) -- node [above, scale = 0.7, sloped] {$S$} (1.6, 2.8);
\draw[->] (2.4, 4.4) -- node [above, scale = 0.7] {$-S$} (3.6, 4.4);
\draw[->] (2.4, 2.8) -- node [above, scale = 0.7] {$-S$} (3.6, 2.8);
\draw[->] (4.4, 4.4) -- node [above, scale = 0.7] {$d$} (5.6, 4.4);
\draw[->] (4.4, 2.8) -- node [above, scale = 0.7] {$d$} (5.6, 2.8);
\draw[->] (6.4, 4.4) -- node [above, scale = 0.7] {$-S$} (7.6, 4.4);
\draw[->] (6.4, 2.8) -- node [above, scale = 0.7] {$-S$} (7.6, 2.8);
\draw[->] (8.4, 4.4) -- node [above, scale = 0.7] {$d$} (9.6, 4.4);
\draw[->] (8.4, 2.8) -- node [above, scale = 0.7] {$d$} (9.6, 2.8);
\draw[->] (2.4, 3) -- node [above, scale = 0.7, sloped, near start] {$S$} (3.6, 4.2);
\draw[->] (4.4, 3) -- node [above, scale = 0.7, sloped, near start] {$D$} (5.6, 4.2);
\draw[->] (6.4, 3) -- node [above, scale = 0.7, sloped, near start] {$S$} (7.6, 4.2);
\draw[->] (8.4, 3) -- node [above, scale = 0.7, sloped, near start] {$D$} (9.6, 4.2);
\draw[-, line width=6pt, color = white] (2.4, 4.2) -- (3.6, 3);
\draw[-, line width=6pt, color = white] (4.4, 4.2) -- (5.6, 3);
\draw[-, line width=6pt, color = white] (6.4, 4.2) -- (7.6, 3);
\draw[-, line width=6pt, color = white] (8.4, 4.2) -- (9.6, 3);
\draw[->] (2.4, 4.2) -- node [above, scale = 0.7, sloped, near end] {$S$} (3.6, 3);
\draw[->] (4.4, 4.2) -- node [above, scale = 0.7, sloped, near end] {$D$} (5.6, 3);
\draw[->] (6.4, 4.2) -- node [above, scale = 0.7, sloped, near end] {$S$} (7.6, 3);
\draw[->] (8.4, 4.2) -- node [above, scale = 0.7, sloped, near end] {$D$} (9.6, 3);
\node at (0,1) {$\cdots$};
\node at (0.8, 1.8) {$q^{1+6m}\alpha$};
\node at (3.2, 1.8) {$q^{2+6m}\gamma$};
\node at (5.6, 1.8) {$q^{4+6m}\gamma$};
\node at (8, 1.8) {$q^{5+6m}\alpha$};
\node at (10.4, 1.8) {$q^{7+6m}\alpha$};
\node at (0.8, 0.2) {$q^{1+6m}\beta$};
\node at (3.2, 0.2) {$q^{2+6m}\delta$};
\node at (5.6, 0.2) {$q^{4+6m}\delta$};
\node at (8, 0.2) {$q^{5+6m}\beta$};
\node at (10.4, 0.2) {$q^{7+6m}\beta$};
\node at (11.2, 1) {$\cdots$};
\draw[->] (1.5, 1.8) -- node [above, scale = 0.7] {$-S$} (2.5, 1.8);
\draw[->] (1.5, 0.2) -- node [above, scale = 0.7] {$-S$} (2.5, 0.2);
\draw[->] (3.9, 1.8) -- node [above, scale = 0.7] {$d$} (4.9, 1.8);
\draw[->] (3.9, 0.2) -- node [above, scale = 0.7] {$d$} (4.9, 0.2);
\draw[->] (6.3, 1.8) -- node [above, scale = 0.7] {$-S$} (7.3, 1.8);
\draw[->] (6.3, 0.2) -- node [above, scale = 0.7] {$-S$} (7.3, 0.2);
\draw[->] (8.7, 1.8) -- node [above, scale = 0.7] {$d$} (9.7, 1.8);
\draw[->] (8.7, 0.2) -- node [above, scale = 0.7] {$d$} (9.7, 0.2);
\draw[->] (1.4, 0.4) -- node [above, scale = 0.7, near start, sloped] {$S$} (2.6, 1.6);
\draw[->] (3.8, 0.4) -- node [above, scale = 0.7, near start, sloped] {$D$} (5, 1.6);
\draw[->] (6.2, 0.4) -- node [above, scale = 0.7, near start, sloped] {$S$} (7.4, 1.6);
\draw[->] (8.6, 0.4) -- node [above, scale = 0.7, near start, sloped] {$D$} (9.8, 1.6);
\draw[-, line width=6pt, color = white] (1.4, 1.6) -- (2.6, 0.4);
\draw[-, line width=6pt, color = white] (3.8, 1.6) -- (5, 0.4);
\draw[-, line width=6pt, color = white] (6.2, 1.6) -- (7.4, 0.4);
\draw[-, line width=6pt, color = white] (8.6, 1.6) -- (9.8, 0.4);
\draw[->] (1.4, 1.6) -- node [above, scale = 0.7, near end, sloped] {$S$} (2.6, 0.4);
\draw[->] (3.8, 1.6) -- node [above, scale = 0.7, near end, sloped] {$D$} (5, 0.4);
\draw[->] (6.2, 1.6) -- node [above, scale = 0.7, near end, sloped] {$S$} (7.4, 0.4);
\draw[->] (8.6, 1.6) -- node [above, scale = 0.7, near end, sloped] {$D$} (9.8, 0.4);
\end{tikzpicture}
\]
The object labelled $\omega$ has homological degree $0$, and the objects in $q$-degree $1+6m$ have homological degree $1+4m$. Morphisms are as in Figure \ref{fig:cheatsheet}.

This cochain complex is infinitely generated, but is periodic in that $\mathcal{B}^{i+4} = q^6\mathcal{B}^i$ and $\partial^{i+4} = \partial^i$ for $i>0$. We can form the following finitely generated quotient complexes which turn out to have the right chain homotopy type for $3$-torus braids, apart from a shift in $q$-degree. We can compare this complex with the categorification of the third Jones-Wenzl projector given in \cite[\S 4.4]{MR2901969} and think of it as a version for Bar-Natan homology.

For $k\geq 0$ let $\mathcal{B}_{3k}$ be the quotient complex of $\mathcal{B}$ by the subcomplex generated by all objects of homological degree greater than $4k$. Note that for $k=0$, this means there is only the generator $\omega$ in bidegree $(0,0)$. For $k\geq 1$ the objects of highest homological degree are $u^{4k}q^{5+6(k-1)}\alpha$ and $u^{4k}q^{5+6(k-1)}\beta$.

We also let $\mathcal{B}_{3k+1}$ be the quotient complex of $\mathcal{B}$ by the subcomplex generated by all objects of homological degree greater than $2+4k$ and the object $u^{2+4k}q^{2+6k}\delta$.  So this complex ends with
\[
\begin{tikzpicture}
\node at (0,1.2) {$u^{1+4k}q^{1+6k}\alpha$};
\node at (0,0.2) {$u^{1+4k}q^{1+6k}\beta$};
\node at (4,1.2) {$u^{2+4k}q^{2+6k}\gamma$};
\draw[->] (1.1, 1.2) -- node [above, scale = 0.7] {$-S$} (2.9, 1.2);
\draw[->] (1.1, 0.3) -- node [above, scale = 0.7, sloped] {$S$} (2.9, 1.1);
\end{tikzpicture}
\]

Because of $\Omega_3$ and for proof-technical purposes, let $\mathcal{B}_{3k+1}^a$ be the quotient complex of $\mathcal{B}$ by the subcomplex generated by all objects of homological degree greater than $2+4k$.

Finally, we let $\mathcal{B}_{3k+2}$ be the quotient complex of $\mathcal{B}$ by the subcomplex generated by all objects of homological degree greater than $3+4k$ and the object $u^{3+4k}q^{4+6k}\delta$. So this complex ends with
\[
\begin{tikzpicture}
\node at (0,1.2) {$u^{2+4k}q^{2+6k}\gamma$};
\node at (0,0.2) {$u^{2+4k}q^{2+6k}\delta$};
\node at (4,1.2) {$u^{3+4k}q^{4+6k}\gamma$};
\draw[->] (1.1, 1.2) -- node [above, scale = 0.7] {$d$} (2.9, 1.2);
\draw[->] (1.1, 0.3) -- node [above, scale = 0.7, sloped] {$D$} (2.9, 1.1);
\end{tikzpicture}
\]

\begin{theorem}\label{thm:torusbraid}
Let $w=(ab)^m$ with $m\geq 0$. Then 
\begin{enumerate}
\item $\CBN(T_w;\Z[h])$ is chain homotopy equivalent to $q^{2m}\mathcal{B}_m$.
\item for $m = 3k+1$, $\CBN(T_{wa};\Z[h])$ is chain homotopy equivalent to $q^{6k+3}\mathcal{B}_{3k+1}^a$.
\end{enumerate}
\end{theorem}

The proof is a lengthy induction starting with $m = 0$, then cycling through the four quotient complexes that we defined. The key observation is that $\mathcal{B}\otimes \chi$ is contractible for $\chi$ either being $\varepsilon$ or $\zeta$. To see this use (\ref{eq:tensor_greeks}) with $\mathcal{B}\otimes \varepsilon$ to get Figure \ref{fig:contract_eps}.
\begin{figure}[ht]
\begin{tikzpicture}
\node at (0,1) {$\cdots$};
\node at (0.8, 2) {$q^{2+6m}\varepsilon$};
\node at (3.2, 2) {$q^{2+6m}\varepsilon$};
\node at (5.6, 2) {$q^{4+6m}\varepsilon$};
\node at (8, 2) {$q^{6+6m}\varepsilon$};
\node at (10.4,2) {$q^{8+6m}\varepsilon$};
\node at (0.8,1) {$q^{6m}\varepsilon$};
\node at (3.2,1) {$q^{3+6m}\zeta$};
\node at (5.6, 1) {$q^{5+6m}\zeta$};
\node at (8, 1) {$q^{4+6m}\varepsilon$};
\node at (10.4,1) {$q^{6+6m}\varepsilon$};
\node at (0.8,0) {$q^{1+6m}\zeta$};
\node at (3.2,0) {$q^{1+6m}\zeta$};
\node at (5.6, 0) {$q^{3+6m}\zeta$};
\node at (8, 0) {$q^{5+6m}\zeta$};
\node at (10.4,0) {$q^{7+6m}\zeta$};
\node at (11.4, 1) {$\cdots$};
\drawgray{1.5, 0.1}{2.5, 0.9}
\drawgray{1.5, 0.2}{2.5, 1.8}
\drawgrayw{1.5, 0.9}{2.5,0.1}
\drawgrayw{1.5, 1}{2.5, 1}
\drawgrayw{1.5, 1.1}{2.5, 1.9} 
\drawgrayw{1.5, 1.9}{2.5, 1.1}
\drawblack{1.5, 0}{2.5, 0}{ }{-\id}
\drawblack{1.5, 2}{2.5, 2}{ }{-\id}
\drawgray{3.9, 0}{4.9, 0}
\drawgray{3.9, 0.2}{4.9, 1.8}
\drawgrayw{3.9, 1}{4.9, 1}
\drawgray{3.9, 1.1}{4.9, 1.9}
\drawgrayw{3.9, 1.8}{4.9, 0.2}
\drawgrayw{3.9, 1.9}{4.9, 1.1}
\drawgray{3.9, 2}{4.9,2}
\drawblackw{3.9, 0.9}{4.9, 0.1}{sloped}{\id}
\drawgray{6.3, 0}{7.3, 0}
\drawgray{6.3, 0.1}{7.3, 0.9}
\drawgray{6.3, 0.2}{7.3, 1.8}
\drawgray{6.3, 1.1}{7.3, 1.9}
\drawgrayw{6.3, 1.8}{7.3, 0.2}
\drawgray{6.3, 2}{7.3, 2}
\drawblackw{6.3, 0.9}{7.3, 0.1}{sloped}{-\id}
\drawblackw{6.3, 1.9}{7.3, 1.1}{sloped}{-\id}
\drawgray{8.7, 0}{9.7, 0}
\drawgray{8.7, 0.1}{9.7, 0.9}
\drawgray{8.7, 0.2}{9.7, 1.8}
\drawgrayw{8.7, 0.9}{9.7, 0.1}
\drawgrayw{8.7, 1}{9.7, 1}
\drawgrayw{8.7, 1.8}{9.7, 0.2}
\drawgray{8.7, 2}{9.7, 2}
\drawblackw{8.7, 1.9}{9.7, 1.1}{sloped}{\id}
\end{tikzpicture}
\caption{\label{fig:contract_eps}The complex $\mathcal{B}\otimes \varepsilon$ after delooping.}
\end{figure}
The $-\id$-morphisms are either the result of a birth composed with a surgery, or a surgery composed with a death. The $\id$-morphisms arise as a composition of birth, dot, and death. The lightly shown arrows can be worked out in detail, but are irrelevant since they do not survive the cancellations.

We can now perform Gaussian elimination along the various $\pm\id$-morphisms. It is important to note that there are no morphisms which lower the $q$-degree. This allows us to cancel the two $-\id$-morphisms in the first and third column in parallel. Also note that for $m = 0$ the remaining $\varepsilon$ on the left can be cancelled with $\omega\otimes \varepsilon$, as the surgery will turn into an identity.

The argument for $\mathcal{B}\otimes \zeta$ is similar, and we will see a bit more detail in the proof of Lemma \ref{lm:threektensor}.

\begin{lemma}
\label{lm:threektensor}
For $k\geq 0$ and $i,j\in \Z$ we have
\begin{align*}
\mathcal{B}_{3k}\otimes u^iq^j\varepsilon &\simeq u^{4k+i}q^{6k+j}\varepsilon\\
\mathcal{B}_{3k}\otimes u^iq^j\zeta & \simeq u^{4k+i}q^{6k+j}\zeta
\end{align*}
Furthermore, the chain homotopy equivalences can be obtained by finitely many Gauss eliminations, each cancelling a pair of objects in increasing homological degrees.
\end{lemma}

\begin{proof}
We may assume $i = 0 = j$. For $\varepsilon$ this follows from Figure \ref{fig:contract_eps}, where we stop with $q^{6+6(k-1)}\varepsilon$. For $\zeta$ this is similar, see Figure \ref{fig:contract_zet}.
\end{proof}

\begin{figure}[ht]
\begin{tikzpicture}
\node at (-0.4,1) {$\cdots$};
\node at (0.8, 2) {$q^{3+6m}\varepsilon$};
\node at (4.4, 2) {$q^{5+6m}\varepsilon$};
\node at (8, 2) {$q^{5+6m}\varepsilon$};
\node at (0.8, 1) {$q^{1+6m}\varepsilon$};
\node at (4.4, 1) {$q^{3+6m}\varepsilon$};
\node at (8, 1) {$q^{6+6m}\zeta$};
\node at (0.8, 0) {$q^{2+6m}\zeta$};
\node at (4.4, 0) {$q^{4+6m}\zeta$};
\node at (8, 0) {$q^{4+6m}\zeta$};
\drawgray{1.5,0}{3.7,0}
\drawgray{1.5,0.1}{3.7, 0.9}
\drawgray{1.5,0.2}{3.7, 1.8}
\drawgrayw{1.5,0.9}{3.7, 0.1}
\drawgrayw{1.5,1}{3.7,1}
\drawgray{1.5,1.1}{3.7,1.9}
\drawgrayw{1.5,1.8}{3.7,0.2}
\drawgray{1.5,2}{3.7,2}
\drawblackw{1.5, 1.9}{3.7, 1.1}{sloped}{\id}
\drawblack{5.1, 0}{7.3,0}{ }{-\id}
\drawblack{5.1, 2}{7.3,2}{ }{-\id}
\drawgray{5.1, 0.9}{7.3, 0.1}
\drawgray{5.1, 1}{7.3, 1}
\drawgray{5.1, 1.1}{7.3, 1.9}
\drawblackw{5.1, 0.2}{7.3, 1.8}{sloped}{S}
\drawblackw{5.1, 0.1}{7.3, 0.9}{sloped}{-/\hspace{-7pt}\bullet+h}
\drawblackw{5.1, 1.9}{7.3, 1.1}{sloped}{S}
\end{tikzpicture}
\caption{\label{fig:contract_zet}The complex $\mathcal{B}_{3k}\otimes \zeta$ after delooping.}
\end{figure}
We note that Figure \ref{fig:contract_zet} contains slightly more information in the morphisms between the last two columns. This will come in handy later.  For this reason, we also add this information for $\mathcal{B}_{3k}\otimes \varepsilon$ in Figure \ref{fig:contract_epsdet}.
\begin{figure}[ht]
\begin{tikzpicture}
\node at (-0.4,1) {$\cdots$};
\node at (0.8, 2) {$q^{2+6m}\varepsilon$};
\node at (4.4, 2) {$q^{4+6m}\varepsilon$};
\node at (8, 2) {$q^{6+6m}\varepsilon$};
\node at (0.8, 1) {$q^{3+6m}\zeta$};
\node at (4.4, 1) {$q^{5+6m}\zeta$};
\node at (8, 1) {$q^{4+6m}\varepsilon$};
\node at (0.8, 0) {$q^{1+6m}\zeta$};
\node at (4.4, 0) {$q^{3+6m}\zeta$};
\node at (8, 0) {$q^{5+6m}\zeta$};
\drawgray{1.5,0}{3.7,0}
\drawgray{1.5,0.1}{3.7, 0.9}
\drawgray{1.5,0.2}{3.7, 1.8}
\drawgrayw{1.5,1.9}{3.7, 1.1}
\drawgrayw{1.5,1}{3.7,1}
\drawgray{1.5,1.1}{3.7,1.9}
\drawgrayw{1.5,1.8}{3.7,0.2}
\drawgray{1.5,2}{3.7,2}
\drawgray{5.1, 0}{7.3, 0.}
\drawgray{5.1, 0.1}{7.3, 0.9}
\drawgray{5.1, 0.2}{7.3, 1.8}
\drawblackw{5.1, 1.1}{7.3, 1.9}{sloped, near end}{S}
\drawblackw{1.5, 0.9}{3.7, 0.1}{sloped}{\id}
\drawblackw{5.1, 0.9}{7.3,0.1}{sloped}{-\id}
\drawblackw{5.1, 1.9}{7.3,1.1}{sloped, near end}{-\id}
\drawblack{5.1, 2}{7.3, 2}{ }{-\backslash\hspace{-7pt}\bullet+h}
\drawblackw{5.1, 1.8}{7.3, 0.2}{sloped, near end}{S}
\end{tikzpicture}
\caption{\label{fig:contract_epsdet}The complex $\mathcal{B}_{3k}\otimes \varepsilon$ after delooping.}
\end{figure}

\begin{lemma}\label{lm:3kto3k+1}
Let $k\geq 0$. If $\CBN(T_{(ab)^{3k}};\Z[h])$ is chain homotopy equivalent to $q^{6k}\mathcal{B}_{3k}$, then $\CBN(T_{(ab)^{3k+1}};\Z[h])$ is chain homotopy equivalent to $q^{6k+2}\mathcal{B}_{3k+1}$.
\end{lemma}

\begin{proof}
We have 
\begin{align*}
\CBN(T_{(ab)^{3k+1}};\Z[h])& \cong \CBN(T_{(ab)^{3k}};\Z[h])\otimes \CBN(T_{ab};\Z[h]) \\
& \simeq q^{6k}\mathcal{B}_{3k}\otimes \CBN(T_{ab};\Z[h]).
\end{align*}
Now $\CBN(T_{ab};\Z[h])$ is given by
\[
\begin{tikzpicture}
\node at (0, 1) {$q^2\omega$};
\node at (2.5, 1) {$uq^3\alpha$};
\node at (2.5, 0) {$uq^3\beta$};
\node at (5, 0) {$u^2q^4\gamma$};
\draw[->] (0.4, 1) -- node [above, scale = 0.7] {$S$} (2, 1);
\draw[->] (0.4, 0.9) -- node [above, scale = 0.7, sloped] {$S$} (2, 0);
\draw[->] (3, 1) -- node[above, scale = 0.7, sloped] {$-S$} (4.5, 0.1);
\draw[->] (3, 0) -- node[above, scale = 0.7] {$S$} (4.5, 0);
\end{tikzpicture}
\]
and from this it follows that $q^{6k}\mathcal{B}_{3k}\otimes \CBN(T_{ab};\Z[h])$ is the total complex of a diagram of cochain complexes given by
\[
\begin{tikzpicture}
\node at (0, 1.5) {$q^{6k+2}\mathcal{B}_{3k}$};
\node at (5, 1.5) {$q^{6k+3}\mathcal{B}_{3k}\otimes \varepsilon\otimes \varepsilon^\ast$};
\node at (5, 0) {$q^{6k+3}\mathcal{B}_{3k}\otimes \zeta\otimes \zeta^\ast$};
\node at (10, 0) {$q^{6k+4}\mathcal{B}_{3k}\otimes \varepsilon \otimes \zeta^\ast$};
\draw[->] (0.8, 1.5) -- node [above, scale = 0.7] {$S$} (3.5, 1.5);
\draw[->] (0.8, 1.4) -- node [above, scale = 0.7, sloped] {$S$} (3.5, 0);
\draw[->] (6.5, 1.5) -- node [above, scale = 0.7, sloped] {$-\id\otimes \id\otimes S$} (8.5, 0.1);
\draw[->] (6.5, 0) -- node [above, scale = 0.7] {$\id \otimes S \otimes \id$} (8.5, 0); 
\end{tikzpicture}
\]
By Lemma \ref{lm:threektensor} this total complex is chain homotopy equivalent by a sequence of Gaussian eliminations to
\[
\begin{tikzpicture}
\node at (0, 0.75) {$\cdots$};
\node at (1, 1.5) {$q^{1+12k} \alpha$};
\node at (1, 0) {$q^{1+12k} \beta$};
\node at (5, 1.5) {$q^{3+12k} \alpha$};
\node at (9, 1.5) {$q^{4+12k} \gamma$};
\node at (5, 0) {$q^{3+12k} \beta$};
\drawblack{1.8, 1.5}{4.2, 1.5} { }{f_1}
\drawblack{1.8, 0.1}{4.2, 1.4} {sloped, near start}{f_2}
\drawblack{1.8,0}{4.2,0} { }{g_2}
\drawblackw{1.8, 1.4}{4.2, 0.1}{sloped, near end}{g_1}
\drawblack{5.8, 1.5}{8.2, 1.5}{ }{h_1}
\drawblack{5.8, 0}{8.2, 1.4}{sloped}{h_2}
\end{tikzpicture}
\]
This has the right generators for $q^{6k+2}\mathcal{B}_{3k+1}$, but we need to work out the morphisms to see that it agrees with it as a cochain complex. To work out $f_1$, note that $q^{3+12k}\alpha$ is the object corresponding to $q^{6+6m}\varepsilon$ in Figure \ref{fig:contract_epsdet}. Before the cancellations, $q^{1+12k}\alpha$ maps to $q^{3+12k}\alpha$ with $\begin{tikzpicture}\smcup{-0.15}{-0.5}{0.6}{thick}\node at (0,-0.3) {$\bullet$};\end{tikzpicture}-h$, but it also maps to the soon to be cancelled $q^{1+12k}\alpha$ with $\id$. Cancelling this object (the $q^{4+6m}\varepsilon$ in Figure \ref{fig:contract_epsdet}) lead to the morphism between $q^{1+6k}\alpha$ and $q^{3+6k}\alpha$ being given by 
\[
\begin{tikzpicture}
\smcup{-0.15}{-0.5}{0.6}{thick}
\node at (0,-0.3) {$\bullet$}; 
\node at (0.6,-0.2) {$-h-$};
\draw[thick] (1.1, 0) -- (1.1, -0.4);
\node at (1.1, -0.2) {$\bullet$};
\node at (1.5, -0.2) {$+h.$};
\end{tikzpicture}
\]
But there is also now a morphism $S$ from $q^{1+12k}\alpha$ to the soon to be cancelled $q^{2+12k}\delta$ (the $q^{5+6m}\zeta$ in Figure \ref{fig:contract_epsdet}). Cancelling this object adds an extra double surgery $D$ from $q^{1+12k}\alpha$ to $q^{3+12k}\alpha$. But this double surgery is simply going from $\varepsilon$ to $\zeta$ and back, so from the neck-cutting relation we can replace this with
\[
\begin{tikzpicture}
\smcap{-0.15}{-0.5}{0.6}{thick}
\node at (0,-0.3) {$\bullet$}; 
\node at (0.5,-0.3) {$+$};
\draw[thick] (0.8, -0.1) -- (0.8, -0.5);
\node at (0.8, -0.3) {$\bullet$};
\node at (1.2, -0.3) {$-h.$};
\end{tikzpicture}
\]
Therefore
\[
\begin{tikzpicture}
\node at (0,0.2) {$f_1 = $};
\smcap{0.5}{0}{0.6}{thick}
\node at (1.1,0.2) {$+$};
\smcup{1.4}{0}{0.6}{thick}
\node at (2.2, 0.2) {$-\,h.$};
\node at (0.65,0.2) {$\bullet$};
\node at (1.55, 0.2) {$\bullet$};
\end{tikzpicture}
\]
as required. Also notice that $q^{1+12k}\beta$ did not map to $q^{3+12k}\alpha$ until this last cancellation. Since it does map to $q^{2+12k}\delta$ with a surgery, it follows from this cancellation that $f_2 = D$ is a double surgery from $\beta$ to $\alpha$.

The argument that $g_1$ and $g_2$ are as in $\mathcal{B}_{3k+1}$ is completely analogous, using Figure \ref{fig:contract_zet} instead of Figure \ref{fig:contract_epsdet}.

To see that $h_1$ and $h_2$ are as needed, assume we have done the cancellations in $q^{6k+3}\mathcal{B}_{3k}\otimes \alpha$ and $q^{6k+3}\mathcal{B}_{3k}\otimes \beta$, but $q^{6k+4}\mathcal{B}_{3k}\otimes \gamma$ still looks as in Figure \ref{fig:contract_epsdet} (note that since this complex has no morphisms to the other two complexes, so the cancellations in those two complexes do not affect it). So we now want to perform the cancellations in $q^{6k+4}\mathcal{B}_{3k}\otimes \gamma$.

The surviving generator $q^{4+12k}\gamma$ corresponds to $q^{6+6m}\varepsilon$ in Figure \ref{fig:contract_epsdet}, and $q^{3+12k}\alpha$ maps to it with $-S$. Since $q^{3+12k}\alpha$ does not map to the other two generators, the cancellations do not change the morphism, and $h_1 = -S$, as required.

The generator $q^{3+12k}\beta$, which comes from $q^{6+6m}\zeta$ in Figure \ref{fig:contract_zet} is mapped to what corresponds to $q^{5+6m}\zeta$ in Figure \ref{fig:contract_epsdet}, which in the total complex is of the form $q^{3+12k}\beta$ and this morphism is simply the identity (a surgery followed by death). After cancelling this element, we can read off Figure \ref{fig:contract_epsdet} that $h_2 = S$, as required.
\end{proof}

\begin{lemma}\label{lm:3k+1to3k+2}
Let $k\geq 0$. 
\begin{enumerate}
\item If $\CBN(T_{(ab)^{3k+1}};\Z[h])$ is chain homotopy equivalent to $q^{6k+2}\mathcal{B}_{3k+1}$, then $\CBN(T_{(ab)^{3k+1}a};\Z[h])$ is chain homotopy equivalent to $q^{6k+3}\mathcal{B}^a_{3k+1}$.
\item If $\CBN(T_{(ab)^{3k+1}a};\Z[h])$ is chain homotopy equivalent to $q^{6k+3}\mathcal{B}^a_{3k+1}$, then $\CBN(T_{(ab)^{3k+2}};\Z[h])$ is chain homotopy equivalent to $q^{6k+4}\mathcal{B}_{3k+2}$.
\end{enumerate}
\end{lemma}

\begin{proof}
As in the proof of Lemma \ref{lm:3kto3k+1} we have
\[
\CBN(T_{(ab)^{3k+1}a};\Z[h])\simeq q^{6k+2}\mathcal{B}_{3k+1}\otimes \CBN(T_a;\Z[h]),
\]
and $\CBN(T_a;\Z[h])$ is given by
\[
\begin{tikzpicture}
\node at (0,0) {$q\omega$};
\node at (2,0) {$uq^2\alpha.$};
\draw[->] (0.3,0) -- node [above, scale = 0.7] {$S$} (1.5, 0);
\end{tikzpicture}
\]
Note that $\mathcal{B}_{3k+1}\otimes \alpha = \mathcal{B}_{3k+1}\otimes \varepsilon\otimes \varepsilon^\ast$, and from Figure \ref{fig:contract_eps}, where we only need the objects in the first column, and the top object in the second column, we see that this complex collapses to just one object, the one corresponding to $q^{1+6m}\zeta$.

The extra element after cancellations is therefore of the form $u^{2+4k}q^{5+12k}\delta$ and we need to check what the morphisms starting from $u^{1+4k}q^{4+12k}\alpha$ and $u^{1+4k}q^{4+12k}\beta$ are. The morphism from $\beta$ to $\delta$ is $-S$: the minus sign comes from the odd homological degree that $\beta$ is in, and the surgery is the one from $\CBN(T_a;\Z[h])$. 

To get the morphism starting at $\alpha$ note that the surgery from $\CBN(T_a;\Z[h])$ turns into a $-\id$ to an $\alpha$ which is going to get cancelled. Indeed, we can think of this $\alpha$ to correspond to $q^{6+6m}\varepsilon$ in Figure \ref{fig:contract_eps}, and we need to work out the light arrow from $q^{6+6m}\varepsilon$ to $q^{7+6m}\zeta$. This morphism is the composition of a double surgery with a birth, which turns into just one surgery. So, after cancellation, the morphism from $u^{1+4k}q^{4+12k}\alpha$ to $u^{2+4k}q^{5+12k}\delta$ is $S$, as required.

The proof of (2) is similar, using tensor product with $\CBN(T_b;\Z[h])$. One observes that $\mathcal{B}^a_{3k+1}\otimes \beta$ is chain homotopy equivalent to a complex with exactly one generator, whose smoothing is $\gamma$ and which is corresponding to the $q^{3+6m}\varepsilon$ in the top-left corner of Figure \ref{fig:contract_zet}.
\end{proof}

\begin{lemma}\label{lm:3k+2to3k+3}
Let $k\geq 0$. If $\CBN(T_{(ab)^{3k+2}};\Z[h])$ is chain homotopy equivalent to $q^{6k+4}\mathcal{B}_{3k+2}$, then $\CBN(T_{(ab)^{3k+3}};\Z[h])$ is chain homotopy equivalent to $q^{6k+6}\mathcal{B}^a_{3k+3}$.
\end{lemma}

\begin{proof}
As in Lemma \ref{lm:3k+1to3k+2} we get from $(ab)^{3k+2}$ to $(ab)^{3k+3}$ in two steps. The behaviour, particularly in the first step, is slightly different, so we give a bit more detail. Consider $q^{6k+4}\mathcal{B}_{3k+2}\otimes \CBN(T_a;\Z[h])$. After several Gaussian eliminations in lower homological degrees this ends with
\[
\scalebox{1}[0.95]{
\begin{tikzpicture}
\node at (0,3.75) {$\cdots$};
\node at (1, 4.5) {$q^{5+12k}\alpha$};
\node at (4, 4.5) {$q^{6+12k}\gamma$};
\node at (7, 4.5) {$q^{8+12k}\gamma$};
\node at (1,3) {$q^{5+12k}\beta$};
\node at (4, 3) {$q^{6+12k}\delta$};
\node at (7, 3) {$q^{7+12k}\alpha$};
\node at (10, 3) {$q^{9+12k}\alpha$};
\node at (4, 1.5) {$q^{7+12k}\alpha$};
\node at (7, 1.5) {$q^{8+12k}\delta$};
\node at (4, 0) {$q^{6+12k}\delta$};
\node at (7,0) {$q^{6+12k}\delta$};
\drawblack{1.7, 4.5}{3.3, 4.5}{ }{-S}
\drawgray{1.7, 4.4}{3.3, 1.5}
\drawgray{1.7, 4.3}{3.3, 0.1}
\drawgray{1.7, 2.9}{3.3,0}
\drawblackw{1.7, 3}{3.3, 3}{ }{-S}
\drawblackw{1.7, 3.1}{3.3, 4.4}{sloped, near end}{S}
\drawblackw{1.7, 4.4}{3.3, 3.1}{sloped, near end}{S}
\drawblack{4.7, 4.5}{6.3,4.5}{ }{d}
\drawblack{4.7, 4.4}{6.3, 3.1}{sloped, near end}{S}
\drawblackw{4.7, 3.1}{6.3, 4.4}{sloped, near start}{D}
\drawblack{4.7,3}{6.3, 1.6}{sloped, near start}{\bullet\hspace{3pt}-h}
\drawblack{4.7,2.9}{6.3, 0.1}{sloped, near end}{\id}
\drawblack{4.7, 0}{6.3, 0}{ }{-\id}
\drawblackw{4.7, 0.2}{6.3, 2.8}{sloped, near start}{S}
\drawblackw{4.7, 0.1}{6.3, 1.4}{sloped}{-/\hspace{-7pt}\bullet+h}
\drawblackw{4.7, 1.5}{6.3, 1.5}{near start}{S}
\drawblackw{4.7, 1.6}{6.3, 3}{sloped, near end}{-\id}
\smcup{4.925}{2.8}{0.4}{thick}
\drawblack{7.7,4.5}{9.3,3.1}{sloped}{-S}
\drawgray{7.7, 3}{9.3, 3}
\drawblack{7.7, 1.5}{9.3, 2.9}{sloped}{S}
\drawgray{7.7, 0}{9.3, 2.8}
\end{tikzpicture}
}
\]
Gaussian elimination on $-\id$ between the bottom objects $q^{6+12k}\delta$ leads to the complex
\[
\scalebox{1}[0.95]{
\begin{tikzpicture}
\node at (0,3.75) {$\cdots$};
\node at (1, 4.5) {$q^{5+12k}\alpha$};
\node at (4, 4.5) {$q^{6+12k}\gamma$};
\node at (7, 4.5) {$q^{8+12k}\gamma$};
\node at (1,3) {$q^{5+12k}\beta$};
\node at (4, 3) {$q^{6+12k}\delta$};
\node at (7, 3) {$q^{7+12k}\alpha$};
\node at (10, 4.5) {$q^{9+12k}\alpha$};
\node at (4, 1.5) {$q^{7+12k}\alpha$};
\node at (7, 1.5) {$q^{8+12k}\delta$};
\drawblack{1.7, 4.5}{3.3, 4.5}{ }{-S}
\drawgray{1.7, 4.4}{3.3, 1.5}
\drawblackw{1.7, 3}{3.3, 3}{ }{-S}
\drawblackw{1.7, 3.1}{3.3, 4.4}{sloped, near end}{S}
\drawblackw{1.7, 4.4}{3.3, 3.1}{sloped, near end}{S}
\drawblack{4.7, 4.5}{6.3,4.5}{ }{d}
\drawblack{4.7, 4.4}{6.3, 3.1}{sloped, near end}{S}
\drawblackw{4.7, 3.1}{6.3, 4.4}{sloped, near start}{D}
\drawblack{4.7, 3}{6.3, 3}{ }{S}
\drawblack{4.7, 1.5}{6.3, 1.5}{ }{S}
\drawblack{4.7, 1.6}{6.3, 2.9}{sloped, near start}{-\id}
\drawblackw{4.7, 2.9}{6.3, 1.6}{sloped, near end}{\bullet\hspace{3pt}-/\hspace{-5pt}\bullet}
\smcup{5.7375}{2.1}{0.4}{thick}
\drawblack{7.7, 4.5}{9.3, 4.5}{ }{-S}
\drawblack{7.7, 1.5}{9.3, 4.3}{sloped}{S}
\drawgray{7.7, 3}{9.3, 4.4}
\end{tikzpicture}
}
\]
which after Gaussian elimination of the two objects $q^{7+12k}\alpha$ turns into
\[
\begin{tikzpicture}
\node at (3,3.75) {$\cdots$};
\node at (4, 4.5) {$q^{6+12k}\gamma$};
\node at (7, 4.5) {$q^{8+12k}\gamma$};
\node at (4, 3) {$q^{6+12k}\delta$};
\node at (7, 3) {$q^{8+12k}\delta$};
\node at (10, 4.5) {$q^{9+12k}\alpha$};
\drawblack{4.7, 4.5}{6.3,4.5}{ }{d}
\drawblack{4.7, 3.1}{6.3, 4.4}{sloped, near start}{D}
\drawblackw{4.7, 4.4}{6.3, 3.1}{sloped, near end}{D}
\drawblack{4.7, 3}{6.3, 3}{ }{d}
\drawblack{7.7, 4.5}{9.3, 4.5}{ }{-S}
\drawblack{7.7, 3}{9.3, 4.4}{sloped}{S}
\end{tikzpicture}
\]
Tensoring this complex with $\CBN(T_b;\Z[h])$ and Gaussian eliminations in lower homological degrees gives the complex
\[
\begin{tikzpicture}
\node at (0,3.75) {$\cdots$};
\node at (1, 4.5) {$q^{10+12k}\gamma$};
\node at (4, 4.5) {$q^{11+12k}\alpha$};
\node at (1, 3) {$q^{10+12k}\delta$};
\node at (4, 3) {$q^{12+12k}\gamma$};
\node at (7, 3) {$q^{12+12k}\gamma$};
\node at (1, 1.5) {$q^{10+12k}\gamma$};
\node at (4, 1.5) {$q^{10+12k}\gamma$};
\node at (4, 0) {$q^{11+12k}\beta$};
\drawblack{1.8, 4.5}{3.2, 4.5}{ }{-S}
\drawgray{1.8, 1.6}{3.2, 2.9}
\drawgray{1.8, 4.4}{3.2, 3.1}
\drawblackw{1.8, 4.3}{3.2, 1.6}{sloped}{-\id}
\drawblackw{1.8, 3.1}{3.2, 4.4}{sloped, near end}{S}
\drawblackw{1.8, 2.9}{3.2, 0.1}{sloped, near end}{-S}
\drawblackw{1.8, 1.5}{3.2, 1.5}{near end}{\id}
\drawblack{1.8, 1.4}{3.2, 0}{sloped}{S}
\drawblack{4.8, 3}{6.2, 3}{ }{-\id}
\drawgray{4.8, 4.5}{6.2, 3.1}
\drawgray{4.8, 1.5}{6.2, 2.9}
\drawgray{4.8, 0}{6.2, 2.8}
\end{tikzpicture}
\]
Two Gaussian eliminations later this complex is exactly $q^{6k+6}\mathcal{B}_{3k+3}$.
\end{proof}

\begin{proof}[Proof of Theorem \ref{thm:torusbraid}]
Part (1) is done by induction on $m$, with the case $m = 0$ trivial. The induction step uses Lemma \ref{lm:3kto3k+1}, \ref{lm:3k+1to3k+2}, and \ref{lm:3k+2to3k+3}, depending on whether $m = 3k$, $3k+1$, or $3k+2$.

Part (2) follows from part (1) together with Lemma \ref{lm:3k+1to3k+2}(2).
\end{proof}

\section{The Bar-Natan complex for torus links on three strands}\label{sec:toruslinks}

We have the functor $C_L\colon \Cob(D^3_3)\to \Cob(D^2_2)$ obtained by linking up the two points at the left outside of the rectangle. Our first goal is to describe the cochain complex $C_L(\mathcal{B})$ over $\Cob(D^2_2)$. After delooping there are two possible smoothings in $\Cob(D^2_2)$, namely
\[
\begin{tikzpicture}
\node at (0,0) {$\tomega = $};
\node at (2.5,0) {and};
\node at (4,0) {$\talpha = $};
\tmomega{0.5}{-0.3}{0.6}{very thick}
\tmalpha{4.5}{-0.3}{0.6}{very thick}
\end{tikzpicture}
\]
It is easy to check that
\begin{align*}
C_L(\omega) & = q \tomega \oplus q^{-1}\tomega & C_L(\alpha) &= \tomega\\
C_L(\beta) & = q\talpha\oplus q^{-1}\talpha & C_L(\gamma) &= \talpha = C_L(\delta)
\end{align*}

From this we get that $C_L(\mathcal{B})$ starts with
\[
\begin{tikzpicture}
\node at (0,3) {$q\tomega$};
\node at (2,3) {$q\tomega$};
\node at (4,3) {$q^2\talpha$};
\node at (6,3) {$q^4\talpha$};
\node at (8,3) {$q^5\tomega$};
\node at (10,3) {$q^7\tomega$};
\node at (0, 1.5) {$q^{-1}\tomega$};
\node at (2, 1.5) {$q^2\talpha$};
\node at (4, 1.5) {$q^2\talpha$};
\node at (6, 1.5) {$q^4\talpha$};
\node at (8, 1.5) {$q^6\talpha$};
\node at (10, 1.5) {$q^8\talpha$};
\node at (2, 0) {$\talpha$};
\node at (8, 0) {$q^4\talpha$};
\node at (10, 0) {$q^6\talpha$};
\node at (11, 1.5) {$\cdots$};
\drawgray{0.5, 1.6}{1.6, 2.9}
\drawgrayw{0.3, 2.9}{1.6, 1.5}
\drawblack{0.5, 1.4}{1.8, 0} {sloped}{S}
\drawblack{0.3, 3}{1.6, 3}{ }{\id}
\drawgray{2.4, 3}{3.6, 3}
\drawgray{2.4, 2.9}{3.6, 1.6}
\drawblackw{2.3, 0.1}{3.6, 2.8}{sloped, near start}{\bullet}
\drawblack{2.3, 0}{3.6, 1.4}{sloped, below}{-\hspace{3pt}\bullet}
\drawblackw{2.4, 1.5}{3.6, 1.5}{near end}{-\id}
\drawblackw{2.4, 1.6}{3.6, 2.9}{sloped, near start}{\id}
\smcap{2.4}{0.725}{0.35}{thick}
\smcup{3.075}{0.56}{0.35}{thick}
\drawgray{4.4, 1.5}{5.6, 1.5}
\drawgray{4.4, 1.6}{5.6, 2.9}
\drawgrayw{4.4, 2.9}{5.6, 1.6}
\drawblack{4.4,3}{5.6,3}{ }{d}
\drawblack{6.4,3}{7.6,3}{ }{-S}
\drawblack{6.4,1.6}{7.6,2.9}{sloped, near end}{S}
\drawblack{6.4,1.5}{7.6,1.5}{ }{-\bullet+h}
\drawblack{6.4,1.4}{7.6,0}{sloped, near start}{-\id}
\drawblackw{6.4,2.9}{7.6,1.6}{sloped, near end}{\bullet \,\,- h}
\drawblackw{6.4, 2.7}{7.6, 0.2}{sloped, near end}{\id}
\node[scale = 0.7] at (6.9375,1.65) {$\bullet$};
\smcup{6.85}{1.525}{0.35}{thick}
\smcap{7.155}{2.125}{0.35}{thick}
\drawgray{8.4, 0}{9.6,0}
\drawgray{8.4, 0.1}{9.6, 2.8}
\drawgrayw{8.4, 1.5}{9.6, 1.5}
\drawgrayw{8.4, 2.9}{9.6, 1.6}
\drawblackw{8.4, 1.6}{9.6, 2.9}{sloped, near end}{S}
\drawblackw{8.4, 2.8}{9.6,0.2}{sloped, near end}{S}
\drawblack{8.4,3}{9.6,3}{ }{2\bullet\hspace{5pt}-h}
\tmomega{8.81}{3.065}{0.2}{thick}
\end{tikzpicture}
\]
where $d$ is as before, and continues as
\[
\begin{tikzpicture}
\node at (0,1.5) {$\cdots$};
\node at (1, 3) {$q^{1+6m}\tomega$};
\node at (3.4, 3) {$q^{2+6m}\talpha$};
\node at (5.8, 3) {$q^{4+6m}\talpha$};
\node at (8.2, 3) {$q^{5+6m}\tomega$};
\node at (10.6, 3) {$q^{7+6m}\tomega$};
\node at (11.6, 1.5) {$\cdots$};
\node at (1, 1.5) {$q^{2+6m}\talpha$};
\node at (3.4, 1.5) {$q^{2+6m}\talpha$};
\node at (5.8, 1.5) {$q^{4+6m}\talpha$};
\node at (8.2, 1.5) {$q^{6+6m}\talpha$};
\node at (10.6, 1.5) {$q^{8+6m}\talpha$};
\node at (1, 0) {$q^{6m}\talpha$};
\node at (8.2, 0) {$q^{4+6m}\talpha$};
\node at (10.6, 0) {$q^{6+6m}\talpha$};
\drawblack{1.7,3}{2.7,3}{ }{-S}
\drawblack{1.6,0}{2.7,1.4}{sloped, below}{-\,\,\bullet}
\drawblack{1.6, 0.2}{2.7,2.8}{sloped, near start}{\bullet}
\drawblackw{1.7,1.5}{2.7, 1.5}{near start}{-\id}
\drawblack{1.7, 1.6}{2.7, 2.9}{sloped, near start}{\id} 
\drawblackw{1.7, 2.9}{2.7,1.6}{sloped, near start}{S}
\smcup{2.275}{0.575}{0.35}{thick}
\smcap{1.655}{0.8}{0.35}{thick}
\drawblack{4.1, 3}{5.1,3}{ }{d}
\drawgray{4.1, 1.5}{5.1, 1.5}
\drawgray{4.1, 1.6}{5.1, 2.9}
\drawgrayw{4.1, 2.9}{5.1, 1.6}
\drawblack{6.5,3}{7.5,3}{ }{-S}
\drawblack{6.5,1.6}{7.5,2.9}{sloped, near end}{S}
\drawblack{6.5,1.5}{7.5,1.5}{ }{-\bullet+h}
\drawblack{6.5,1.4}{7.5,0}{sloped, near start}{-\id}
\drawblackw{6.5,2.9}{7.5,1.6}{sloped, near end}{\bullet \,\,- h}
\drawblackw{6.5, 2.7}{7.5, 0.2}{sloped, near end}{\id}
\node[scale = 0.7] at (6.9375,1.65) {$\bullet$};
\smcup{6.85}{1.525}{0.35}{thick}
\smcap{7.135}{2.125}{0.35}{thick}
\drawgray{8.9, 0}{9.9,0}
\drawgray{8.9, 0.1}{9.9, 2.8}
\drawgrayw{8.9, 1.5}{9.9, 1.5}
\drawgrayw{8.9, 2.9}{9.9, 1.6}
\drawblackw{8.9, 1.6}{9.9, 2.9}{sloped, near end}{S}
\drawblackw{8.9, 2.8}{9.9,0.2}{sloped, near end}{S}
\drawblack{8.9,3}{9.9,3}{ }{2\bullet\hspace{5pt}-h}
\tmomega{9.215}{3.065}{0.2}{thick}
\end{tikzpicture}
\]
Again, we do not work out morphisms that will not survive the Gaussian eliminations. Using Gaussian elimination, this complex is chain homotopy equivalent to the following complex that we call $\tilde{\mathcal{B}}$.
\[
\begin{tikzpicture}
\node at (0.1, 1.5) {$\cdots$};
\node at (1,1.5) {$q^{2+6m}\talpha$};
\node at (3.4, 1.5) {$q^{4+6m}\talpha$};
\node at (5.8, 1.5) {$q^{6+6m}\talpha$};
\node at (8.2, 1.5) {$q^{6+6m}\talpha$};
\node at (10.6, 1.5) {$q^{8+6m}\talpha$};
\node at (11.5, 1.5) {$\cdots$};
\node at (5.8, 0.5) {$q^{5+6m}\tomega$};
\node at (8.2, 0.5) {$q^{7+6m}\tomega$};
\node at (2.2, 2.5) {$q^{-1}\tomega$};
\node at (4.6, 2.5) {$\talpha$};
\node at (7, 2.5) {$q^2\talpha$};
\node at (9.4, 2.5) {$q^4\talpha$};
\node at (10.4, 2.5) {$\cdots$};
\drawblack{2.8, 2.5}{4.3, 2.5}{ }{S}
\drawblack{4.9, 2.5}{6.6, 2.5}{ }{c}
\drawblack{7.4, 2.5}{9, 2.5}{ }{d}
\drawblack{1.7, 1.5}{2.7, 1.5}{ }{d}
\drawblack{4.1, 1.5}{5.1, 1.5}{ }{c}
\drawblack{6.5, 0.6}{7.5, 1.5}{sloped, near end}{S}
\drawblack{6.5, 0.5}{7.5, 0.5}{ }{e}
\drawblackw{6.5, 1.5}{7.5, 0.6}{sloped, near end}{S}
\drawblack{8.9, 1.5}{9.9, 1.5}{ }{c}
\end{tikzpicture}
\]
Notice that after Gaussian elimination the morphism between $q^{4+6m}\talpha$ and $q^{5+6m}\tomega$ is $S-S = 0$, and similarly for the morphism between $q^{7+6m}\tomega$ and $q^{8+6m}\talpha$.

For $k\geq 1$ let $\tilde{\mathcal{B}}_{3k}$ be the quotient complex of $\tcal{B}$ by the subcomplex consisting of all objects of homological degree greater than $4k$. That is, it ends in
\[
\begin{tikzpicture}
\node at (-0.3,0) {$\cdots$};
\node at (1,0) {$q^{4+6(k-1)}\talpha$};
\node at (4,0) {$q^{6k}\talpha$};
\node at (4, -0.75) {$q^{6k-1}\tomega$};
\drawblack{2, 0}{3.4, 0}{ }{c}
\end{tikzpicture}
\]
In the special case $k=0$ we let $\tcal{B}_0$ be $q\tomega\oplus q^{-1}\tomega$ in homological degree $0$. We now assume $k\geq 0$.

Let $\tcal{B}_{3k+1}$ be the quotient complex of $\tcal{B}$ by the subcomplex consisting of all objects of homological degree greater than $4k+1$. For $k\geq 1$ it ends in
\[
\begin{tikzpicture}
\node at (4.8, 1.5) {$\cdots$};
\node at (5.8, 1.5) {$q^{6k}\talpha$};
\node at (8.6, 1.5) {$q^{6k}\talpha$};
\node at (5.8, 0.5) {$q^{6k-1}\tomega$};
\node at (8.6, 0.5) {$q^{6k+1}\tomega$};
\drawblack{6.5, 0.6}{8.1, 1.5}{sloped, near end}{S}
\drawblack{6.5, 0.5}{8, 0.5}{ }{e}
\drawblackw{6.3, 1.5}{8, 0.6}{sloped, near end}{S}
\end{tikzpicture}
\]
Let $\tcal{B}_{3k+1}^a$ be the quotient complex of $\tcal{B}$ by the subcomplex consisting of all objects of homological degree greater than $4k+2$. For $k\geq 1$ it ends in
\[
\begin{tikzpicture}
\node at (11, 1.5) {$q^{6k+2}\talpha$};
\node at (7.2, 1) {$\cdots$};
\node at (8.2, 1.5) {$q^{6k}\talpha$};
\node at (8.2, 0.5) {$q^{6k+1}\tomega$};
\drawblack{8.7, 1.5}{10.3, 1.5}{ }{c}
\end{tikzpicture}
\]
Finally, let $\tcal{B}_{3k+2}$ be the quotient complex of $\tcal{B}$ by the subcomplex consisting of all objects of homological degree greater than $4k+3$. It ends in
\[
\begin{tikzpicture}
\node at (0.1, 1.5) {$\cdots$};
\node at (1,1.5) {$q^{6k+2}\talpha$};
\node at (3.9, 1.5) {$q^{6k+4}\talpha$};
\drawblack{1.7, 1.5}{3.3, 1.5}{ }{d}
\end{tikzpicture}
\]
\begin{lemma}\label{lm:torus2ends}
Let $m\geq 0$. Then $C_L(\mathcal{B}_m)\simeq \tcal{B}_m$, and $C_L(\mathcal{B}^a_{3m+1})\simeq \tcal{B}^a_{3m+1}$.
\end{lemma}

\begin{proof}
This follows from the Gaussian eliminations in $C_L(\mathcal{B})$ after carefully checking where the cancellations stop for the various cases.
\end{proof}

Now let $G\colon \Cob(D^2_2)\to \mathfrak{Mod}^q_{\Z[X,h]/(X^2-Xh)}$ be the functor obtained by linking up the two points on the right outside of the rectangle followed by identification with the graded-modules category. After delooping there is only one possible smoothing, an arc, which gets identified with $A$. The morphism $G(c) = 0$, so the complex $G(\tcal{B})$ simplifies considerably. Furthermore,
\[
G(d) = G(e) = 2X-h.
\]

\begin{proposition}\label{prp:torusonestrand}
The cochain complex $G(\tcal{B})$ is chain homotopy equivalent to
\[
q^{-2}A \oplus \bigoplus_{m=0}^\infty u^{4m+2}q^{6m+2} A(1)\oplus \bigoplus_{m=1}^\infty u^{4m}q^{6m-2}A(2).
\]
\end{proposition}

\begin{proof}
Since $G(c)=0$ the complex $G(\tcal{B})$ falls into small blocks of at most two consecutive homological degrees. In homological degrees $0$ and $1$ we have
\[
\begin{tikzpicture}
\node at (-0.3,1.3) {$A$};
\node at (2,0.8) {$A$};
\node at (-0.3,0.3) {$q^{-2}A$};
\drawblack{0.3,1.3}{1.7,0.9}{sloped}{\id}
\drawblack{0.3,0.3}{1.7,0.7}{sloped}{X}
\end{tikzpicture}
\]
After the Gaussian elimination this is $q^{-2}A$. Even more directly, we have
\[
G\left( u^{4m+2}q^{6m+2}\talpha \stackrel{d}{\longrightarrow}u^{4m+3}q^{6m+4}\talpha\right) = u^{4m+2}q^{6m+2} A(1).
\]
Therefore it remains to show that
\[
\begin{tikzpicture}
\node at (5.4, 1.5) {$u^{4m}q^{6m}\talpha$};
\node at (9.2, 1.5) {$u^{4m+1}q^{6m}\talpha$};
\node at (5.4, 0.5) {$u^{4m}q^{6m-1}\tomega$};
\node at (9.2, 0.5) {$u^{4m+1}q^{6m+1}\tomega$};
\drawblack{6.5, 0.6}{8.1, 1.5}{sloped, near end}{S}
\drawblack{6.5, 0.5}{8, 0.5}{ }{e}
\drawblackw{6.3, 1.5}{8, 0.6}{sloped, near end}{S}
\end{tikzpicture}
\]
is mapped to $u^{4m}q^{6m}A(2)$ up to chain homotopy equivalence. Ignoring the homological degrees, applying $G$ gives
\[
\begin{tikzpicture}
\node at (0,3) {$q^{6m}A$};
\node at (0,1.5) {$q^{6m}A$};
\node at (-0.1,0) {$q^{6m-2}A$};
\node at (4.2, 3) {$q^{6m}A$};
\node at (4.1, 1.5) {$q^{6m+2}A$};
\node at (4.2,0) {$q^{6m}A$};
\drawblack{0.8,0}{3.5,0}{ }{2X - h}
\drawblack{0.8, 0.2}{3.5, 2.8}{sloped, near start}{X}
\drawblackw{0.7,1.5}{3.4,1.5}{near start}{2X - h}
\drawblack{0.7, 1.7}{3.5, 3}{sloped, near end}{\id}
\drawblackw{0.7, 2.8}{3.5, 0.2}{sloped, near end}{\id}
\drawblackw{0.7, 3}{3.4, 1.7}{sloped, near start}{X - h}
\end{tikzpicture}
\]
By Gaussian elimination of the identity from the top left to the bottom right, and using that $X (X - h) = 0$, we get this to be chain homotopy equivalent to
\[
\begin{tikzpicture}
\node at (0,2) {$q^{6m}A$};
\node at (-0.1,0) {$q^{6m-2}A$};
\node at (4,2) {$q^{6m}A$};
\node at (4.1,0) {$q^{6m+2}A$};
\drawblack{0.8, 0}{3.3, 0}{ }{Xh - \,h^2}
\drawblack{0.8, 0.2}{3.4, 1.8}{sloped, near start}{X}
\drawblack{0.7, 2}{3.4, 2}{ }{\id}
\drawblackw{0.7, 1.8}{3.3, 0.2}{sloped, near end}{2X -\,h}
\end{tikzpicture}
\]
After one more Gaussian elimination the morphism between the two surviving objects is
\[
Xh - h^2-X(2X-h) = Xh-h^2-Xh = -h^2.
\]
After a change of basis to fix the minus sign, we get the desired result.
\end{proof}

It is now straightforward to work out the various chain homotopy types of the complexes $G(\tcal{B}_m)$. In terms of notation, we write $T(3,3k+1.5)$ for the $2$-component link corresponding to the braid word $(ab)^{3k+1}a$. It is easy to see that one of the components is an unknot, and the other component is the torus knot $T(2,2k+1)$.


\begin{theorem}\label{thm:khovhomtorus}
The following hold as chain homotopy equivalences of $q$-graded $\Z[X,h]/(X^2-Xh)$-cochain complexes.
\begin{enumerate}
\item For $k\geq 1$
\begin{align*}
\CBN(T(3,3k;\Z[h])\simeq &\,\, q^{6k-2}A\oplus  u^{4k}q^{12k}A\oplus u^{4k}q^{12k}A\oplus u^{4k}q^{12k-2}A \,\oplus \\
& \bigoplus_{m=0}^{k-1}u^{4m+2}q^{6(k+m)+2}A(1) \oplus
\bigoplus_{m=1}^{k-1}u^{4m}q^{6(k+m)-2}A(2).
\end{align*}
\item For $k\geq 0$
\begin{align*}
\CBN(T(3,3k+1);\Z[h]) \simeq & \,\,q^{6k} A \oplus \bigoplus_{m=0}^{k-1} u^{4m+2}q^{6(k+m)+4}A(1)\,\oplus\\
& \bigoplus_{m=1}^k u^{4m}q^{6(k+m)}A(2).
\end{align*} 
\item For $k\geq 0$
\begin{align*}
\CBN(T(3,3k+1.5);\Z[h]) \simeq & \,\, q^{6k+1} A \oplus \bigoplus_{m=0}^{k-1} u^{4m+2}q^{6(k+m)+5}A(1)\,\oplus\\
&  u^{4k+2}q^{12k+5}A \oplus \bigoplus_{m=1}^k u^{4m}q^{6(k+m)+1}A(2).
\end{align*}
\item For $k\geq 0$
\begin{align*}
\CBN(T(3,3k+2);\Z[h]) \simeq & \,\,q^{6k+2} A \oplus \bigoplus_{m=0}^{k} u^{4m+2}q^{6(k+m+1)}A(1)\,\oplus\\
& \bigoplus_{m=1}^k u^{4m}q^{6(k+m)+2}A(2).
\end{align*}
\end{enumerate}
\end{theorem}

\begin{proof}
This follows from Theorem \ref{thm:torusbraid} and Proposition \ref{prp:torusonestrand}, after carefully checking how the various cochain complexes end.
\end{proof}
\begin{remark}
Theorem \ref{thm:khovhomtorus} shows that \cite[Prop.4.2]{lewark2024gordian} holds also with integral coefficients and thus with any coefficients.
\end{remark}

The Khovanov homology can be obtained from this by using the change of ring homomorphism $\eta\colon \Z[h] \to \Z$ sending $h$ to $0$ and identifying $A= q\Z[h]\oplus q^{-1}\Z[h]$. Then 
\[
A(2)\otimes_{\Z[h]}\Z \cong q^{-1}\Z\oplus q\Z\oplus uq^3\Z\oplus uq^5\Z,
\]
while $A(1)\otimes_{\Z[h]}\Z$ has a multiplication by $2$ between two copies of $q\Z$, as well as a $q^{-1}\Z$ and a $uq^3\Z$.

For negative $k$ we can use that the Khovanov homology of a mirror link can be obtained by dualizing, \cite{MR2232858}. The missing case of $k=0$ for $T(3,3k)$, where the formula does not quite work, is left to the reader.


\section{The Bar-Natan complex for words in $\Omega_5^+$}
\label{sec:omega_five}

Let us now assume that $w\in \Omega_5^+$, that is, $w = (ab)^{3k}b^l$ with $k,l>0$. Then
\[
\CBN(T_w;\Z[h]) \cong \CBN(T_{(ab)^{3k}};\Z[h])\otimes \CBN(T_{b^l};\Z[h]),
\]
and by Theorem \ref{thm:torusbraid} and Proposition \ref{prp:toruslink2} we have
\[
\CBN(T_w;\Z[h]) \simeq q^{6k+l} \mathcal{B}_{3k}\otimes \mathcal{D}_{b^l},
\]
where $\mathcal{D}_{b^l}$ is the cochain complex
\[
\omega \stackrel{S}{\longrightarrow} q\beta \stackrel{d_1}{\longrightarrow}q^3\beta \stackrel{d_2}{\longrightarrow}\cdots \stackrel{d_{l-1}}{\longrightarrow} q^{2l-1}\beta.
\]
We can think of $\mathcal{B}_{3k}\otimes \mathcal{D}_{b^l}$ as a diagram of cochain complexes
\[
\mathcal{B}_{3k}\stackrel{\tilde{S}}{\longrightarrow}q\tbar{B}_{3k}\stackrel{\tilde{d}_1}{\longrightarrow} q^3\tbar{B}_{3k}\longrightarrow \cdots\longrightarrow q^{2l-1}\tbar{B}_{3k},
\]
where $\tbar{B}_{3k} = \mathcal{B}\otimes \beta$. By Lemma \ref{lm:threektensor} $\tbar{B}_{3k}\simeq q^{2l-1}\beta$, so this complex simplifies considerably.

\begin{proposition}\label{prp:formomega5}
Let $k,l\geq 1$, and $w=(ab)^{3k}b^l$. Then
\[
\CBN(T_w;\Z[h]) \simeq q^{6k+l} \mathcal{C}_{k,l},
\]
where $\mathcal{C}_{k,l}$ is the cochain complex over $\Cob(D^3_3)$ which agrees with $\mathcal{B}_{3k}$ in homological degrees less than $4k+1$, and which continues as
\[
\begin{tikzpicture}
\node at (-1,0.75) {$\cdots$};
\node at (0,1.5) {$q^{6k-1}\alpha$};
\node at (0,0) {$q^{6k-1}\beta$};
\node at (3,0.75) {$q^{6k+1}\beta$};
\node at (6,0.75) {$q^{6k+3}\beta$};
\node at (9,0.75) {$q^{6k+5}\beta$};
\node at (10,0.75) {$\cdots$};
\drawblack{0.7,1.5}{2.3,0.85}{sloped}{D}
\drawblack{0.7,0}{2.3, 0.65}{sloped}{d}
\drawblack{3.7,0.75}{5.3, 0.75}{ }{d_1}
\drawblack{6.7,0.75}{8.3, 0.75}{ }{d_2}
\end{tikzpicture}
\]
where $D$ is a double surgery and
\[
\begin{tikzpicture}
\node at (0,0) {$d_{2m-1} = $};
\node at (1.6,0) {$-$};
\tmalpha{0.8}{-0.2}{0.4}{thick}
\tmalpha{2}{-0.2}{0.4}{thick}
\node at (1,-0.075) {$\bullet$};
\node at (2.2, 0.075) {$\bullet$};
\node at (4.8,0) {$d = d_{2m} = $};
\node at (6.6, 0) {$+$};
\tmalpha{5.8}{-0.2}{0.4}{thick}
\tmalpha{7}{-0.2}{0.4}{thick}
\node at (6,-0.075) {$\bullet$};
\node at (7.2, 0.075) {$\bullet$};
\node at (7.9,0) {$- \,\,h.$};
\end{tikzpicture}
\]
\end{proposition}

\begin{proof}
By the discussion above, after using the Gaussian eliminations of Lemma \ref{lm:threektensor}, we have the cochain complex $\mathcal{C}_{k,l}$, which agrees with $\mathcal{B}_{3k}$ in homological degrees less than $4k+1$, and the objects in higher degrees are as stated. It remains to show that the boundaries in homological degrees bigger than $4k$ are also as stated.

We begin with the cochain map $\tilde{S}\colon \mathcal{B}_{3k}\to\tbar{B}_{3k}$. After Gaussian eliminations (compare Figure \ref{fig:contract_zet}) in lower homological degrees, this ends in
\[
\begin{tikzpicture}
\node at (0,4) {$q^{6k-1}\alpha$};
\node at (0,3) {$q^{6k-1}\beta$};
\node at (0,2) {$q^{6k}\gamma$};
\node at (0,0) {$q^{6k-1}\beta$};
\node at (4, 2) {$q^{6k}\gamma$};
\node at (4,1) {$q^{6k+1}\beta$};
\node at (4,0) {$q^{6k-1}\beta$};
\drawblack{0.7, 0}{3.3,0}{ }{-\id}
\drawblack{0.7, 2}{3.3,2}{near end}{-\id}
\drawblack{0.7, 4}{3.3, 2.2}{sloped}{S}
\drawblackw{0.7, 0.2}{3.3, 1.8}{sloped}{S}
\drawblackw{0.7, 0.1}{3.3, 0.9}{sloped}{-/\hspace{-7pt}\bullet+h}
\drawblackw{0.7, 1.9}{3.3, 1.05}{sloped, near start, below}{S}
\drawblackw{0.7, 3}{3.3, 1.2}{sloped, near start}{\bullet \hspace{4pt}-h}
\drawblackw{0.7, 2.8}{3.3,0.2}{sloped, near end}{\id} 
\smcup{1.125}{2.7}{0.4}{thick}
\end{tikzpicture}
\]
After cancelling $-\id$ at the bottom, this is
\[
\begin{tikzpicture}
\node at (0,4) {$q^{6k-1}\alpha$};
\node at (0,3) {$q^{6k-1}\beta$};
\node at (0,2) {$q^{6k}\gamma$};
\node at (4, 3) {$q^{6k}\gamma$};
\node at (4,2) {$q^{6k+1}\beta$};
\drawblack{0.7,4}{3.3, 3.1}{sloped}{S}
\drawblack{0.7,3}{3.3,3}{ }{S}
\drawblack{0.7,2}{3.3,2}{ }{S}
\drawblack{0.7,2.1}{3.3, 2.9}{sloped, very near start}{-\id}
\drawblackw{0.7,2.9}{3.3, 2.1}{sloped, very near end}{\bullet\hspace{4pt}-/\hspace{-5pt}\bullet}
\smcup{2.68}{2.33}{0.4}{thick}
\end{tikzpicture}
\]
Cancelling $-\id$ leads to a double surgery from $\alpha$ to $\beta$, and an extra back and forth surgery from $\beta$ to $\beta$. Since this involves the cap and the straight line, the new morphism is $d$. This shows the boundary in homological degree $4k$ is as stated.

It remains to show that the cochain map $\tilde{d}_m\colon \tbar{B}_{3k}\to q^2\tbar{B}_{3k}$ simplifies to $d_m\colon q^{6k}\beta\to q^{6k+2}\beta$. We can look at the various summands, dotting of the cap, dotting of the cup, and multiplication by $h$.

It turns out that dotting the cup and multiplication by $h$ carry over easily. This is because as a morphism $\tbar{B}_{3k}\to q^2\tbar{B}_{3k}$, each object in $\tbar{B}_{3k}$ maps to the corresponding object in $q^2\tbar{B}_{3k}$ with the same behaviour, namely dotting the cup or multiplication by $h$. After performing Gaussian eliminations, the surviving objects $q^{6k}\beta$ and $q^{6k+2}\beta$ never have their morphism changed.

But for dotting the cap this is more complicated, since the surviving $q^{6k}\beta$ maps it identically to $q^{6k}\beta$, which is going to be Gaussian eliminated, leading to a zig-zag morphism from $q^{6k}\beta$ to $q^{6k+2}\beta$. Let us show that the end result is again dotting the cap.

We can first perform the Gaussian eliminations in $\tbar{B}_{3k}$, then proceed with $q^2\tbar{B}_{3k}$. After Gaussian eliminations in lower homological degrees the complex looks as in
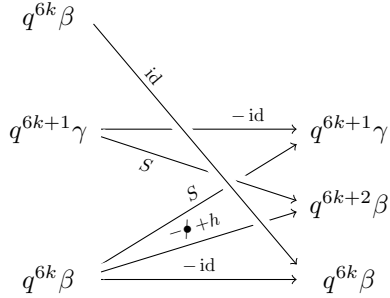
\[
\begin{tikzpicture}
\node at (0,3.5) {$q^{6k}\beta$};
\node at (0,2) {$q^{6k+1}\gamma$};
\node at (0,0) {$q^{6k}\beta$};
\node at (4, 2) {$q^{6k+1}\gamma$};
\node at (4,1) {$q^{6k+2}\beta$};
\node at (4,0) {$q^{6k}\beta$};
\drawblack{0.7, 0}{3.3,0}{ }{-\id}
\drawblack{0.7, 2}{3.3,2}{near end}{-\id}
\drawblackw{0.7, 0.2}{3.3, 1.8}{sloped}{S}
\drawblackw{0.7, 0.1}{3.3, 0.9}{sloped}{-/\hspace{-7pt}\bullet+h}
\drawblackw{0.7, 1.9}{3.3, 1.05}{sloped, near start, below}{S}
\drawblackw{0.6, 3.4}{3.3,0.2}{sloped, near start}{\id} 
\end{tikzpicture}
\]
Gaussian elimination of the bottom $-\id$ leads to
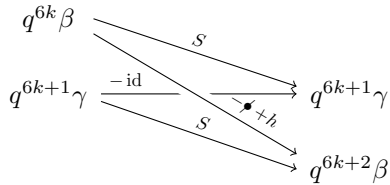
\[
\begin{tikzpicture}
\node at (0,3) {$q^{6k}\beta$};
\node at (0,2) {$q^{6k+1}\gamma$};
\node at (4, 2) {$q^{6k+1}\gamma$};
\node at (4,1) {$q^{6k+2}\beta$};
\drawblack{0.6,3}{3.3, 2.1}{sloped}{S}
\drawblack{0.7,2}{3.3,2}{very near start}{-\id}
\drawblack{0.7,1.9}{3.3, 1}{sloped}{S}
\drawblackw{0.6, 2.8}{3.3,1.2}{sloped, near end}{-/\hspace{-7pt}\bullet+h}
\end{tikzpicture}
\]
Cancelling the final $-\id$ adds a back and forth surgery between $q^{6k}\beta$ and $q^{6k+2}\beta$, involving the cap and the straight line. So by the neck-cutting relation the remaining morphism is just dotting the cap, as claimed.
\end{proof}

\begin{remark}
In the last argument the cancelling of the $q^{6k+1}\gamma$ (and in fact also other objects in homological degrees) results also in more morphisms from $q^{6k}\beta$ to other objects, provided the tensor complex has more parts $q^2\tbar{B}\to q^4\tbar{B}$. Since in our situation, $q^4\tbar{B}$ contracts to only one object in a higher homological degree, these extra morphisms do not survive. However, when dealing with words in $\Omega^4_+$ these morphisms will come back to haunt us, and will make the argument more complicated in this situation.
\end{remark}

We are now going to apply the functor $C_L\colon \Cob(B^3_3)\to \Cob(B^2_2)$, to $\mathcal{C}_{k,l}$. In homological degrees at most $4k$ this behaves exactly as in Section \ref{sec:toruslinks}, so we only need to focus on larger homological degrees.

\begin{proposition}\label{prp:omega5twostrands}
Let $k,l\geq 1$. Then $C_L(\mathcal{C}_{k,l})$ is chain homotopy equivalent to the cochain complex $\tcal{C}_{k,l}$, which agrees with $\tcal{B}_{3k}$ in homological degrees at most $4k$, and which continues as
\[
\begin{tikzpicture}
\node at (0,1) {$q^{6k-2}\talpha$};
\node at (-0.8,1) {$\cdots$};
\node at (2.5,1) {$q^{6k}\talpha$};
\node at (2.5,0) {$q^{6k-1}\tomega$};
\node at (5,1) {$q^{6k+2}\talpha$};
\node at (5,0) {$q^{6k}\talpha$};
\node at (7.5,1) {$q^{6k+4}\talpha$};
\node at (7.5,0) {$q^{6k+2}\talpha$};
\node at (10,1) {$q^{6k+6}\talpha$};
\node at (10,0) {$q^{6k+4}\talpha$};
\node at (10.8,0.5) {$\cdots$}; 
\drawblack{0.7,1}{1.8,1}{ }{c}
\drawblack{3.2,1}{4.3,1}{ }{d}
\drawblack{3.2,0.1}{4.3,0.9}{sloped, near start}{(\hspace{3pt}\bullet - h)S}
\smcap{3.025}{0.175}{0.4}{thick}
\drawblack{3.2,0}{4.3,0}{ }{S}
\drawblack{5.7,1}{6.8,1}{ }{c}
\drawblack{5.7,0}{6.8,0}{ }{c}
\drawblack{8.2,1}{9.3,1}{ }{d}
\drawblack{8.2,0}{9.3,0}{ }{d}
\end{tikzpicture}
\]
\end{proposition}

\begin{proof}
We note that $C_L(\beta) \cong q\talpha\oplus q^{-1}\talpha$. Furthermore, the morphisms between objects of homological degree larger than $4k$ are not affected by the new circle. This explains the two parallel branches starting with $q^{6k+2}\talpha$ and $q^{6k}\talpha$. Also, the $c$-morphism starting at $q^{6k-2}\talpha$ is part of $\mathcal{B}_{3k}$ and is thus from the result in Section \ref{sec:toruslinks}.

It remains to investigate the effect of $C_L$ on morphisms starting in homological degree $4k$. Consider $C_L(D)\colon C_L(q^{6k-1}\alpha)\to C_L(q^{6k+1}\beta)$. This is
\[
\begin{tikzpicture}
\node at (0,1) {$q^{6k-1}$};
\node at (3.5,1) {$q^{6k}$};
\node at (7,1) {$q^{6k+1}$};
\node at (10.5,2) {$q^{6k+2}$};
\node at (10.5,0) {$q^{6k}$};
\smalpha{0.7}{0.75}{0.5}{0.5}{very thick}
\leftcl{0.7}{0.75}{0.5}{very thick}
\smgamma{4.2}{0.75}{0.5}{0.5}{very thick}
\leftcl{4.2}{0.75}{0.5}{very thick}
\smbeta{7.7}{0.75}{0.5}{0.5}{very thick}
\leftcl{7.7}{0.75}{0.5}{very thick}
\smcup{11}{-0.125}{0.5}{very thick}
\smcap{11}{-0.25}{0.5}{very thick}
\smcup{11}{1.875}{0.5}{very thick}
\smcap{11}{1.75}{0.5}{very thick}
\drawblack{1.4,1}{3.2,1}{ }{S}
\drawblack{4.9,1}{6.5,1}{ }{S}
\drawblack{8.4,1.1}{10,1.9}{sloped, below}{\Phi_+}
\drawblack{8.4,0.9}{10, 0.1}{sloped}{\Phi_-}
\drawblack{4.9,1.1}{10, 2}{sloped}{\bullet \hspace{4pt}-h}
\drawblack{4.9, 0.9}{10,0}{sloped, below}{\id}
\smcap{7.055}{1.55}{0.4}{thick}
\end{tikzpicture}
\]
which explains the morphisms starting at $q^{6k-1}\tomega$. The remaining morphism from $q^{6k}\talpha$ to $q^{6k+2}\talpha$ is direct application of $C_L$.
\end{proof}

\begin{theorem}\label{thm:decoomega5}
Let $k,l\geq 1$ and $w = (ab)^{3k}b^{2l-1}$, and $L_w$ the closure of the braid diagram $T_w$. Then
\begin{align*}
\CBN(L_w;\Z[h])& \simeq \, q^{6k+2l-3}A \oplus u^{4k}q^{12k+2l-3}A \oplus\\
& \bigoplus_{m=0}^{k-1}u^{4m+2}q^{6(k+m)+2l+1}A(1) \oplus
\bigoplus_{m=1}^{k-1}u^{4m}q^{6(k+m)+2l-3}A(2)\oplus \\
&\bigoplus_{m=0}^{l-1}u^{4k+2m}q^{12k+2l+4m-1}A(1)\oplus
\bigoplus_{m=1}^{l-1}u^{4k+2m}q^{12k+2l+4m-3}A(1).
\end{align*}
and
\[
\CBN(L_{wb};\Z[h]) \simeq q \CBN(L_w;\Z[h]) \oplus u^{4k+2(l-1)}(q^{12k+6(l-1)+3}A \oplus q^{12k+6(l-1)+1}A).
\]
\end{theorem}

\begin{proof}
We take the complex $\tcal{C}_{k,2l-1}$ from Proposition \ref{prp:omega5twostrands} and apply the functor $C_R\colon \Cob(B^2_2)\to \Cob(B^1_1)$ which connects the two right-most points. The resulting complex is (apart from the parts that behave as in Section \ref{sec:toruslinks})
\[
\begin{tikzpicture}
\node at (0,1) {$q^{6k-2}$};
\dmomega{0.5}{0.7}{0.6}{very thick}
\node at (-0.8,1) {$\cdots$};
\node at (2.5,1) {$q^{6k}$};
\dmomega{3}{0.7}{0.6}{very thick}
\node at (2.5,0) {$q^{6k}$};
\dmomega{3}{-0.3}{0.6}{very thick}
\node at (2.5, -1) {$q^{6k-2}$};
\dmomega{3}{-1.3}{0.6}{very thick}
\node at (5,1) {$q^{6k+2}$};
\dmomega{5.5}{0.7}{0.6}{very thick}
\node at (5,0) {$q^{6k}$};
\dmomega{5.5}{-0.3}{0.6}{very thick}
\node at (7.5,1) {$q^{6k+4}$};
\dmomega{8}{0.7}{0.6}{very thick}
\node at (7.5,0) {$q^{6k+2}$};
\dmomega{8}{-0.3}{0.6}{very thick}
\node at (10,1) {$q^{6k+6}$};
\dmomega{10.5}{0.7}{0.6}{very thick}
\node at (10,0) {$q^{6k+4}$};
\dmomega{10.5}{-0.3}{0.6}{very thick}
\node at (11,0.5) {$\cdots$}; 
\drawblack{0.7,1}{2.2,1}{ }{0}
\drawblack{3.2,1}{4.5,1}{ }{2\boldsymbol{)}\hspace{-6pt}\bullet - \,h}
\drawblack{3.2,0.1}{4.5,0.9}{sloped}{\boldsymbol{)}\hspace{-6pt}\bullet - \,h}
\drawblack{3.2,0}{4.5,0}{ }{\id}
\drawblack{3.2,-1}{4.5,-0.1}{sloped}{\boldsymbol{)}\hspace{-4.5pt}\bullet}
\drawblack{5.7,1}{7,1}{ }{0}
\drawblack{5.7,0}{7,0}{ }{0}
\drawblack{8.2,1}{9.5,1}{ }{2\boldsymbol{)}\hspace{-6pt}\bullet - \,h}
\drawblack{8.2,0}{9.5,0}{ }{2\boldsymbol{)}\hspace{-6pt}\bullet - \,h}
\end{tikzpicture}
\]
After one Gaussian elimination and identifying the arc with $A$, the complex looks as in
\[
\begin{tikzpicture}
\node at (0,1) {$q^{6k}A$};
\node at (0,0) {$q^{6k-2}A$};
\node at (3,1) {$q^{6k+2}A$};
\node at (6,1) {$q^{6k+4}A$};
\node at (6,0) {$q^{6k+2}A$};
\node at (9,1) {$q^{6k+6}A$};
\node at (9,0) {$q^{6k+4}A$};
\node at (10,0.5) {$\cdots$};
\drawblack{0.7,1}{2.3,1}{ }{2X-h}
\drawblack{6.7,1}{8.3,1}{ }{2X-h}
\drawblack{6.7,0}{8.3,0}{ }{2X-h}
\end{tikzpicture}
\]
Note that this portion of the complex starts in homological degree $4k$, and it ends in homological degree $4k+2l-1$. Since $2l-1$ is odd, the direct summand complex of highest homological degree is an $A(1)$ type complex. For the case $2l$ it ends in two copies of $A$. The statement of the theorem now follows after noting that we need to shift the complex by $q^{6k+2l-1}$ for $\CBN(L_w;\Z[h])$ and $q^{6k+2l}$ for $\CBN(L_{wb};\Z[h])$.
\end{proof}

\section{The Bar-Natan complex for words in $\Omega_4^+$}
\label{sec:omega_four}

For $w = (ab)^{3k}a^{-l}\in \Omega_4^+$ with $k,l\geq 1$ we want to use the same technique as in Section \ref{sec:omega_five}. That is, we use
\[
\CBN(T_w;\Z[h]) \simeq q^{6k-l} \mathcal{B}_{3k}\otimes \mathcal{D}_{a^{-l}},
\]
where $\mathcal{D}_{a^{-l}}$ is the cochain complex
\[
q^{1-2l}\alpha \stackrel{d_{l-1}}{\longrightarrow} \cdots \stackrel{d_2}{\longrightarrow} q^{-3}\alpha \stackrel{d_1}{\longrightarrow}q^{-1}\alpha \stackrel{S}{\longrightarrow} \omega.
\]
The complex $\mathcal{B}_{3k}\otimes \mathcal{D}_{a^{-l}}$ is the total complex of the diagram of complexes
\[
q^{1-2l} \tbar{B}_{3k} \longrightarrow \cdots \longrightarrow q^{-3}\tbar{B}_{3k} \longrightarrow q^{-1}\tbar{B}_{3k} \longrightarrow \mathcal{B}_{3k},
\]
where $\tbar{B}_{3k} = \mathcal{B}_{3k}\otimes \alpha \simeq q^{6k}\alpha$ by Lemma \ref{lm:threektensor}. Indeed, we can use Gaussian elimination to get $\mathcal{B}_{3k}\otimes \mathcal{D}_{a^{-l}}$ chain homotopy equivalent to a complex $\mathcal{C}_{k,-l}$ that contains $\mathcal{B}_{3k}$ as a subcomplex, and so that the quotient is up to grading shift $\mathcal{D}_{a^{-1}}$, except for the $\omega$ object. This complex ends in
\[
\begin{tikzpicture}
\node at (-0.5,1) {$\cdots$};
\node at (-0.5, 2.8) {$\cdots$};
\node at (0.8, 1.8) {$q^{6k-5}\alpha$};
\node at (3.2, 1.8) {$q^{6k-4}\gamma$};
\node at (5.6, 1.8) {$q^{6k-2}\gamma$};
\node at (8, 1.8) {$q^{6k-1}\alpha$};
\node at (0.8, 0.2) {$q^{6k-5}\beta$};
\node at (3.2, 0.2) {$q^{6k-4}\delta$};
\node at (5.6, 0.2) {$q^{6k-2}\delta$};
\node at (8, 0.2) {$q^{6k-1}\beta$};
\node at (5.6, 2.8) {$q^{6k-1}\alpha$};
\node at (3.2, 2.8) {$q^{6k-3}\alpha$};
\node at (0.8, 2.8) {$q^{6k-5}\alpha$};
\draw[->] (1.5, 1.8) -- node [above, scale = 0.7] {$-S$} (2.5, 1.8);
\draw[->] (1.5, 0.2) -- node [above, scale = 0.7] {$-S$} (2.5, 0.2);
\draw[->] (3.9, 1.8) -- node [above, scale = 0.7] {$d$} (4.9, 1.8);
\draw[->] (3.9, 0.2) -- node [above, scale = 0.7] {$d$} (4.9, 0.2);
\draw[->] (6.3, 1.8) -- node [above, scale = 0.7] {$-S$} (7.3, 1.8);
\draw[->] (6.3, 0.2) -- node [above, scale = 0.7] {$-S$} (7.3, 0.2);
\draw[->] (1.4, 0.4) -- node [above, scale = 0.7, near start, sloped] {$S$} (2.6, 1.6);
\draw[->] (3.8, 0.4) -- node [above, scale = 0.7, near start, sloped] {$D$} (5, 1.6);
\draw[->] (6.2, 0.4) -- node [above, scale = 0.7, near start, sloped] {$S$} (7.4, 1.6);
\draw[-, line width=6pt, color = white] (1.4, 1.6) -- (2.6, 0.4);
\draw[-, line width=6pt, color = white] (3.8, 1.6) -- (5, 0.4);
\draw[-, line width=6pt, color = white] (6.2, 1.6) -- (7.4, 0.4);
\draw[->] (1.4, 1.6) -- node [above, scale = 0.7, near end, sloped] {$S$} (2.6, 0.4);
\draw[->] (3.8, 1.6) -- node [above, scale = 0.7, near end, sloped] {$D$} (5, 0.4);
\draw[->] (6.2, 1.6) -- node [above, scale = 0.7, near end, sloped] {$S$} (7.4, 0.4);
\draw[->] (1.4, 2.8) -- node [above, scale = 0.7] {$d_2$} (2.6, 2.8);
\draw[->] (6.2, 2.6) -- node [above, scale = 0.7, sloped] {$f_0$} (7, 2.2);
\draw[->] (3.8,  2.8) -- node [above, scale = 0.7] {$d_1$} (5, 2.8);
\draw[->] (3.8, 2.6) -- node [above, scale = 0.7, sloped] {$f_1$} (4.6, 2.2);
\draw[->] (1.4, 2.6) -- node [above, scale = 0.7, sloped] {$f_2$}(2.2, 2.2);
\end{tikzpicture}
\]
but it is not immediately clear what the morphisms $f_i$ between the objects in the top row and the objects in the $\mathcal{B}_{3k}$ subcomplex are. To work out these morphisms, we need to take a closer look how this complex arises from Gaussian eliminations. Note we can start by performing all Gaussian eliminations in $q^{1-2l}\tbar{B}_{3k}$, resulting in a single object $q^{6k+1-2l}\alpha$, then continue with Gaussian eliminations in $q^{3-2l}\tbar{B}_{3k}$, and so on.

Let us continue until we have done all the Gaussian eliminations up until $q^{-3}\tbar{B}_{3k}$. We then have a subcomplex of $\mathcal{B}_{3k}\otimes \mathcal{D}_{a^{-l}}$ that looks as in Figure \ref{fig:last_step}.

\begin{figure}[ht]
\begin{tikzpicture}[scale = 0.875, transform shape]
\node at (-3.3,3.5) {$\cdots$};
\node at (-2.5,2.5) {$q^{6k-8}\delta$};
\node at (-2.5,3.5) {$q^{6k-9}\alpha$};
\node at (-2.5,4.5) {$q^{6k-7}\alpha$};
\node at (-0.8, 0.5) {$\cdots$};
\node at (0,0) {$q^{6k-7}\beta$};
\node at (0,1) {$q^{6k-7}\alpha$};
\node at (0,2.5) {$q^{6k-6}\delta$};
\node at (0,3.5) {$q^{6k-7}\alpha$};
\node at (0,4.5) {$q^{6k-5}\alpha$};
\node at (2.5,0) {$q^{6k-5}\beta$};
\node at (2.5,1) {$q^{6k-5}\alpha$};
\node at (2.5,2.5) {$q^{6k-6}\delta$};
\node at (2.5,3.5) {$q^{6k-4}\delta$};
\node at (2.5, 4.5) {$q^{6k-5}\alpha$};
\node at (5,0) {$q^{6k-4}\delta$};
\node at (5,1) {$q^{6k-4}\gamma$};
\node at (5,2.5) {$q^{6k-4}\delta$};
\node at (5,3.5) {$q^{6k-2}\delta$};
\node at (5,4.5) {$q^{6k-3}\alpha$};
\node at (7.5,0) {$q^{6k-2}\delta$};
\node at (7.5,1) {$q^{6k-2}\gamma$};
\node at (7.5,2.5) {$q^{6k-2}\delta$};
\node at (7.5,3.5) {$q^{6k-3}\alpha$};
\node at (7.5,4.5) {$q^{6k-1}\alpha$};
\node at (10,0) {$q^{6k-1}\beta$};
\node at (10,1) {$q^{6k-1}\alpha$};
\drawblack{-1.8,4.5}{-0.6,3.5}{sloped}{\id}
\drawblack{-1.8,4.3}{-0.6,1.2}{sloped}{\id}
\drawblack{0.7,0}{1.9,0}{}{}
\drawblack{0.7,0.1}{1.9,0.9}{}{}
\drawblackw{0.7,0.9}{1.9,0.1}{}{}
\drawblack{0.7,1}{1.9,1}{}{}
\drawblack{3.2,0.1}{4.4,0.9}{}{}
\drawblack{3.2,1}{4.4,1}{}{}
\drawblackw{3.2,0.9}{4.4,0.1}{}{}
\drawblack{3.2,0}{4.4,0}{ }{ }
\drawblack{5.7,3.6}{6.9,4.4}{sloped, near end}{S}
\drawblackw{5.7,4.3}{6.9,2.8}{sloped, near end}{S}
\drawblack{5.7,4.5}{6.9,4.5}{ }{-\boldsymbol{(}\hspace{-7pt}\bullet-h}
\drawblack{5.7,1}{6.9,1}{}{}
\drawblack{5.7,0.1}{6.9,0.9}{}{}
\drawblackw{5.7,0.9}{6.9,0.1}{}{}
\drawblack{5.7,0}{6.9,0}{}{}
\drawblack{8.2,1}{9.4,1}{}{}
\drawblack{8.2,0.1}{9.4,0.9}{}{}
\drawblackw{8.2,0.9}{9.4,0.1}{}{}
\drawblack{8.2,0}{9.4,0}{}{}
\drawblack{8.2,4.3}{9.4,1.2}{sloped}{\id}
\drawblackw{5.7,4.1}{6.9,1.2}{sloped, near end}{-S}
\drawblackw{5.7,4.4}{6.9,3.6}{sloped, near end}{-\id}
\drawblackw{5.7,3.5}{6.9,2.6}{sloped}{-\id}
\drawblackw{5.7,3.4}{6.9,0.2}{sloped}{-\id}
\drawblack{3.2,3.4}{4.4,2.6}{sloped}{\id}
\drawblackw{3.2,3.3}{4.4,0.2}{sloped}{\id}
\drawblack{0.7,4.3}{1.9,1.2}{sloped, near end}{-\id}
\drawblackw{0.7,2.7}{1.9,4.3}{sloped}{S}
\drawblack{0.7,4.5}{1.9,4.5}{}{-\id}
\drawblackw{0.7,2.5}{1.9,2.5}{near end}{-\id}
\drawblackw{0.7,2.3}{1.9,0.2}{sloped}{-S}
\end{tikzpicture}
\caption{\label{fig:last_step} Surgery from $q^{-1}\tbar{B}_{3k}$ after delooping to $\mathcal{B}_{3k}$. Only relevant morphisms are shown, compare Figure \ref{fig:contract_eps}.}
\end{figure}

We first want to work out how the objects $q^{6k+1-2l}\alpha$ map into this complex.

\begin{lemma}\label{lm:intoqminusone}
After Gaussian eliminations in all $q^{1-2j}\tbar{B}_{3k}$ for $j = 2,\ldots,l$, the objects $q^{6k+1-2j}\alpha$ with $2\leq j\leq 4k+1$ map into $q^{-1}\tbar{B}_{3k}$ as follows ($i = \lfloor (j-2)/4\rfloor$).
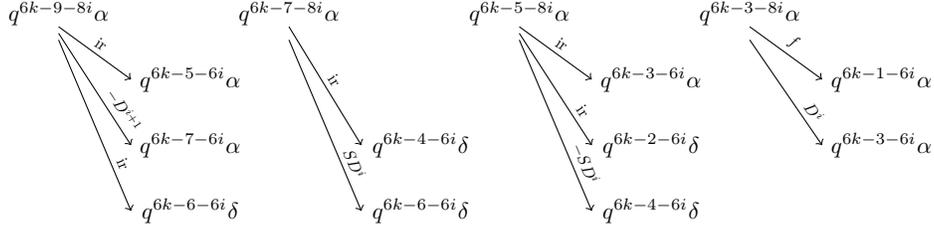
\[
\begin{tikzpicture}[scale = 0.875, transform shape]
\node at (10.5,3) {$q^{6k-3-8i}\alpha$};
\node at (12.5,2) {$q^{6k-1-6i}\alpha$};
\node at (12.5,1) {$q^{6k-3-6i}\alpha$};
\node at (7,3) {$q^{6k-5-8i}\alpha$};
\node at (9,2) {$q^{6k-3-6i}\alpha$};
\node at (9,1) {$q^{6k-2-6i}\delta$};
\node at (9,0) {$q^{6k-4-6i}\delta$};
\node at (3.5,3) {$q^{6k-7-8i}\alpha$};
\node at (5.5,1) {$q^{6k-4-6i}\delta$};
\node at (5.5,0) {$q^{6k-6-6i}\delta$};
\node at (0,3) {$q^{6k-9-8i}\alpha$};
\node at (2,2) {$q^{6k-5-6i}\alpha$};
\node at (2,1) {$q^{6k-7-6i}\alpha$};
\node at (2,0) {$q^{6k-6-6i}\delta$};
\drawblack{10.5,2.8}{11.6,2}{sloped}{f}
\drawblack{10.5,2.6}{11.6,1}{sloped, near end}{D^i}
\drawblack{7,2.8}{8.1,2}{sloped}{\ir}
\drawblack{7,2.7}{8.1,1}{sloped, near end}{\ir}
\drawblack{7,2.6}{8.1,0}{sloped, near end}{-SD^i}
\drawblack{3.5,2.8}{4.6,1}{sloped}{\ir}
\drawblack{3.5,2.6}{4.6,0}{sloped,near end}{SD^i}
\drawblack{0,2.8}{1.1,2}{sloped}{\ir}
\drawblack{0,2.7}{1.1,1}{sloped, near end}{-D^{i+1}}
\drawblack{0,2.6}{1.1,0}{sloped, near end}{\ir}
\end{tikzpicture}
\]
Here $D\colon \alpha\to q^2\alpha$ is a double surgery via $q\delta$, $\ir$ simply means that the morphism is non-zero and the exact form is not important, while $f$ is dotting the cup for $i = 0$ and $\ir$ otherwise.
\end{lemma}

\begin{proof}
For $q^{6k-3}\alpha$ in $q^{-3}\tbar{B}_{3k}$ this follows from Figure \ref{fig:inbetweensteps} with $m = 6k-3$.
\begin{figure}[ht]
\begin{tikzpicture}[scale = 0.885, transform shape] 
\node at (-3.3,3.5) {$\cdots$};
\node at (-2.5,2.5) {$q^{m-7}\delta$};
\node at (-2.5,3.5) {$q^{m-8}\alpha$};
\node at (-2.5,4.5) {$q^{m-6}\alpha$};
\node at (0,-1) {$q^{m-5}\delta$};
\node at (0,0) {$q^{m-6}\alpha$};
\node at (0,1) {$q^{m-4}\alpha$};
\node at (0,2.5) {$q^{m-5}\delta$};
\node at (0,3.5) {$q^{m-6}\alpha$};
\node at (0,4.5) {$q^{m-4}\alpha$};
\node at (-0.8, 0) {$\cdots$};
\node at (2.5,-1) {$q^{m-3}\delta$};
\node at (2.5,0) {$q^{m-4}\alpha$};
\node at (2.5,1) {$q^{m-2}\alpha$};
\node at (2.5,2.5) {$q^{m-5}\delta$};
\node at (2.5,3.5) {$q^{m-3}\delta$};
\node at (2.5, 4.5) {$q^{m-4}\alpha$};
\node at (5,-1) {$q^{m-3}\delta$};
\node at (5,0) {$q^{m-1}\delta$};
\node at (5,1) {$q^{m-2}\alpha$};
\node at (5,2.5) {$q^{m-3}\delta$};
\node at (5,3.5) {$q^{m-1}\delta$};
\node at (5,4.5) {$q^{m-2}\alpha$};
\node at (7.5,-1) {$q^{m-1}\delta$};
\node at (7.5,0) {$q^{m+1}\delta$};
\node at (7.5,1) {$q^{m}\alpha$};
\node at (7.5,2.5) {$q^{m-1}\delta$};
\node at (7.5,3.5) {$q^{m-2}\alpha$};
\node at (7.5,4.5) {$q^{m}\alpha$};
\node at (10,-1) {$q^{m+1}\delta$};
\node at (10,0) {$q^{m}\alpha$};
\node at (10,1) {$q^{m+2}\alpha$};
\drawblack{-1.8,4.5}{-0.6,3.6}{sloped}{\id}
\drawgray{-1.8,4.4}{-0.6,1.2}
\drawblack{-1.8,4.3}{-0.6,0.2}{sloped}{\id}
\drawblack{0.7,0.9}{1.9,0.1}{sloped, near start}{\id}
\drawblack{0.7,4.3}{1.9,0.2}{sloped, below}{-\id}
\drawgray{0.7,4.4}{1.9,1.2}
\drawblackw{0.7,2.7}{1.9,4.3}{sloped}{S}
\drawblackw{0.7,2.5}{1.9,2.5}{near end}{-\id}
\drawgray{0.7,2.3}{1.9,-0.8}
\drawblack{0.7,4.5}{1.9,4.5}{}{-\id}
\drawblack{3.2,-1}{4.4,-1}{}{-\id}
\drawblackw{3.2,-0.8}{4.4,0.8}{sloped}{S}
\drawgray{3.2,3.4}{4.4,0.2}
\drawblack{3.2,1}{4.4,1}{very near start}{-\id}
\drawblackw{3.1,3.3}{4.4,-0.8}{sloped, very near end}{\id}
\drawblack{3.2,3.5}{4.4,2.6}{sloped}{\id}
\drawblack{5.7,-0.1}{6.9,-0.9}{sloped}{\id}
\drawblack{5.7,4.3}{6.9,2.7}{sloped}{S}
\drawgray{5.7,4.1}{6.9,1.2}
\drawgray{5.6,3.4}{6.9,0.2}
\drawblackw{5.6,3.3}{7,-0.7}{sloped, below}{-\id}
\drawblack{5.7,3.5}{6.9,2.5}{sloped, near start}{-\id}
\drawblack{5.7,4.5}{6.9,3.6}{sloped}{-\id}
\drawblack{8.2,0.1}{9.4,0.9}{sloped, very near start}{S}
\drawblackw{8.2,0.8}{9.4,-0.8}{sloped, near end}{S}
\drawblackw{8.2,0.9}{9.4,0.1}{sloped, near start}{-\id}
\drawblack{8.2,1}{9.4,1}{near start}{-\boldsymbol{(}\hspace{-7pt}\bullet-h}
\drawblack{8.2,-0.1}{9.4,-0.9}{sloped, near start}{-\id}
\drawblackw{7.9,4.2}{9.5,0.2}{sloped, below}{\id}
\drawblack{8,4.4}{9.4,1.2}{sloped}{f}
\end{tikzpicture}
\caption{\label{fig:inbetweensteps} Dotting morphisms between $q^x\tbar{B}_{3k}$ and $q^{x+2}\tbar{B}_{3k}$. Only relevant morphisms are shown in black, some of the irrelevant induced by the dotting morphism are shown.}
\end{figure}

We now need to do four induction steps, namely showing that the statement for $q^{6k-j-8i}\alpha$ implies the statement for $q^{6k-(j-2)-8i}\alpha$ for $j = 3, 5, 7, 9$. All of these implications are very similar, so we will only show the case for $j=7$, leaving the rest to the reader.

So assume that $q^{6k-7-8i}\alpha$ maps into $q^{-1}\tbar{B}_{3k}$ as stated. Then $q^{6k-9-8i}\alpha$ maps into $q^{-3}\tbar{B}_{3k}$ in the same way, that is, there is an $\ir$-morphism to $q^{6k-6-6i}\delta$ and a $SD^i$-morphism to $q^{6k-8-6i}\delta$. We can think of these two objects as being in the third column of the top half of Figure \ref{fig:inbetweensteps} with $m = 6k-3-6i$. The top half of this figure contains the objects of $q^{-3}\tbar{B}_{3k}$ while the lower half contains the objects of $q^{-1}\tbar{B}_{3k}$.

We now perform the Gaussian eliminations in $q^{-3}\tbar{B}_{3k}$, starting from the left. When we get to the two $-\id$-morphisms between the second and third column, let us begin with the lower one between $q^{m-5=6k-8-6i}\delta$ objects. After this cancellation, the $q^{6k-9-8i}\alpha$ object maps to $q^{m-4=6k-7-6i}\alpha$ in $q^{-3}\tbar{B}_{3k}$ with $SSD^i = D^{i+1}$, and to $q^{m-3=6k-6-6i}\delta$ in $q^{-1}\tbar{B}_{3k}$ with $\ir$ (composing with the gray morphism).

Performing the other $-\id$-Gaussian elimination between $q^{6k-7-6i}\alpha$ objects creates a morphism $\ir$ from $q^{6k-9-8i}\alpha$ to $q^{m-2=6k-5-6i}\alpha$ in $q^{-1}\tbar{B}_{3k}$, and a $-\id\circ D^{i+1}=-D^{i+1}$ morphism to $q^{m-4=6k-7-6i}\alpha$ in $q^{-1}\tbar{B}_{3k}$. This finishes this induction step. Notice that the $\ir$-morphism to $q^{6k-6-6i}\delta$ in $q^{-3}\tbar{B}_{3k}$ does not lead to further morphisms after the cancellations in $q^{-3}\tbar{B}_{3k}$.
\end{proof}

\begin{proposition}\label{prp:formomega4}
Let $w=(ab)^{3k}a^{-l}$ with $k,l\geq 1$. Then, if $l \geq 4$, the complex $\mathcal{C}_{k,-l}$ has the form
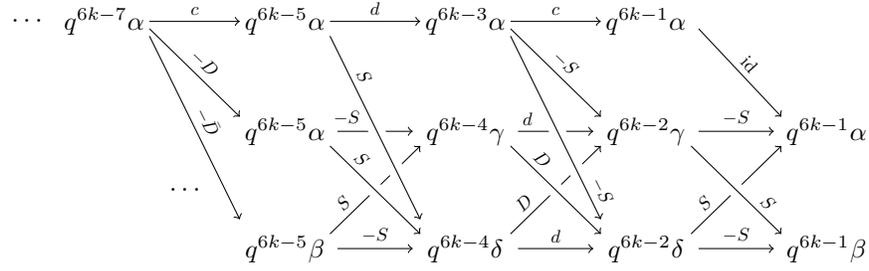
\[
\begin{tikzpicture}
\node at (-0.5,1) {$\cdots$};
\node at (-2.6, 3.3) {$\cdots$};
\node at (0.8, 1.8) {$q^{6k-5}\alpha$};
\node at (3.2, 1.8) {$q^{6k-4}\gamma$};
\node at (5.6, 1.8) {$q^{6k-2}\gamma$};
\node at (8, 1.8) {$q^{6k-1}\alpha$};
\node at (0.8, 0.2) {$q^{6k-5}\beta$};
\node at (3.2, 0.2) {$q^{6k-4}\delta$};
\node at (5.6, 0.2) {$q^{6k-2}\delta$};
\node at (8, 0.2) {$q^{6k-1}\beta$};
\node at (5.6, 3.3) {$q^{6k-1}\alpha$};
\node at (3.2, 3.3) {$q^{6k-3}\alpha$};
\node at (0.8, 3.3) {$q^{6k-5}\alpha$};
\node at (-1.6, 3.3) {$q^{6k-7}\alpha$};
\draw[->] (1.5, 1.8) -- node [above, scale = 0.7, very near start] {$-S$} (2.5, 1.8);
\draw[->] (1.5, 0.2) -- node [above, scale = 0.7] {$-S$} (2.5, 0.2);
\draw[->] (3.9, 1.8) -- node [above, scale = 0.7, very near start] {$d$} (4.9, 1.8);
\draw[->] (3.9, 0.2) -- node [above, scale = 0.7] {$d$} (4.9, 0.2);
\draw[->] (6.3, 1.8) -- node [above, scale = 0.7] {$-S$} (7.3, 1.8);
\draw[->] (6.3, 0.2) -- node [above, scale = 0.7] {$-S$} (7.3, 0.2);
\draw[->] (1.4, 0.4) -- node [above, scale = 0.7, near start, sloped] {$S$} (2.6, 1.6);
\draw[->] (3.8, 0.4) -- node [above, scale = 0.7, near start, sloped] {$D$} (5, 1.6);
\draw[->] (6.2, 0.4) -- node [above, scale = 0.7, near start, sloped] {$S$} (7.4, 1.6);
\draw[-, line width=6pt, color = white] (1.4, 1.6) -- (2.6, 0.4);
\draw[-, line width=6pt, color = white] (3.8, 1.6) -- (5, 0.4);
\draw[-, line width=6pt, color = white] (6.2, 1.6) -- (7.4, 0.4);
\draw[->] (1.4, 1.6) -- node [above, scale = 0.7, near start, sloped] {$S$} (2.6, 0.4);
\draw[->] (3.8, 1.6) -- node [above, scale = 0.7, near start, sloped] {$D$} (5, 0.4);
\draw[->] (6.2, 1.6) -- node [above, scale = 0.7, near end, sloped] {$S$} (7.4, 0.4);
\draw[->] (1.4, 3.3) -- node [above, scale = 0.7] {$d$} (2.6, 3.3);
\draw[->] (3.8,  3.3) -- node [above, scale = 0.7] {$c$} (5, 3.3);
\draw[->] (-1, 3.3) -- node [above, scale = 0.7] {$c$} (0.2, 3.3);
\drawblack{6.3, 3.2}{7.4, 2}{sloped}{\id} 
\drawblackw{3.8, 3.1}{5, 0.6}{sloped, very near end}{-S}
\drawblack{3.8, 3.2}{5, 2}{sloped}{-S}
\drawblackw{1.4, 3.1}{2.6, 0.6}{sloped, near start}{S}
\drawblack{-1,3.2}{0.2,2}{sloped}{-D}
\drawblack{-1,3.1}{0.2,0.6}{sloped}{-\bar{D}}
\end{tikzpicture}
\]
If $l \leq3$, this holds if we consider the subcomplex after removing the objects in the top row of the form $q^{6k-1-2j}\alpha$, with $j \geq l$.

If $l > 4$, we get morphisms from $q^{6k+1-2j}$ with $j <4k+1$ into $\mathcal{B}_{3k}$ as follows.
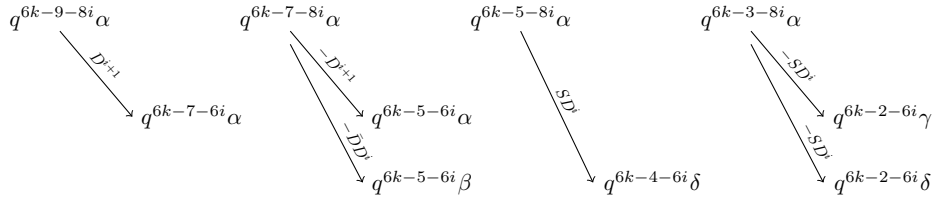
\[
\begin{tikzpicture}[scale = 0.875, transform shape]
\node at (10.5,3) {$q^{6k-3-8i}\alpha$};
\node at (12.5,1.5) {$q^{6k-2-6i}\gamma$};
\node at (12.5,0.5) {$q^{6k-2-6i}\delta$};
\node at (7,3) {$q^{6k-5-8i}\alpha$};
\node at (9,0.5) {$q^{6k-4-6i}\delta$};
\node at (3.5,3) {$q^{6k-7-8i}\alpha$};
\node at (5.5,1.5) {$q^{6k-5-6i}\alpha$};
\node at (5.5,0.5) {$q^{6k-5-6i}\beta$};
\node at (0,3) {$q^{6k-9-8i}\alpha$};
\node at (2,1.5) {$q^{6k-7-6i}\alpha$};
\drawblack{10.5,2.8}{11.6,1.5}{sloped}{-SD^i}
\drawblack{10.5,2.6}{11.6,0.5}{sloped, near end}{-SD^i}
\drawblack{7,2.8}{8.1,0.5}{sloped}{SD^i}
\drawblack{3.5,2.8}{4.6,1.5}{sloped}{-D^{i+1}}
\drawblack{3.5,2.6}{4.6,0.5}{sloped,near end}{-\bar{D}D^i}
\drawblack{0,2.8}{1.1,1.5}{sloped}{D^{i+1}}
\end{tikzpicture}
\]
Finally, if $l\geq 4k+1$ there is a morphism
\[
\begin{tikzpicture}
\node at (0,3) {$q^{-2k-1}\alpha$};
\node at (2,1.5) {$\omega.$};
\drawblack{0,2.8}{1.6,1.5}{sloped}{SD^{k}}
\end{tikzpicture}
\]
Here, $D$ is a back-and-forth surgery from $\alpha$ to $q^2\alpha$ via $q\delta$, while $\bar{D}$ is a double surgery from $\alpha$ to $q^2\beta$. Also, $S$ stands for a surgery determined by domain and codomain.
\end{proposition}

\begin{proof}
The proof is similar to the proof of Lemma \ref{lm:intoqminusone}. From this lemma we know how $q^{6k-j-8i}\alpha$ maps into $q^{-1}\tbar{B}_{3k}$. In particular, we know how these objects map into Figure \ref{fig:last_step} (we can shift each object there by a $q^{-6i}$ to get the general case). Now performing Gaussian eliminations in $q^{-1}\tbar{B}_{3k}$ leads to morphisms as prescribed. We leave the details to the reader.
\end{proof}

We now apply the functor $C_L\colon \Cob(D^3_3)\to \Cob(D^2_2)$ obtained by connecting the left-most strands. Applying the same Gaussian eliminations as in Section \ref{sec:toruslinks}, as well as the Gaussian elimination of the two objects $q^{6k-1}\alpha$ leads to the following.

\begin{proposition}
Let $k,l\geq 1$. Then $C_L(\mathcal{C}_{k,-l})$ is chain homotopy equivalent to a cochain complex $\tcal{C}_{k,-l}$ of the form
\[
\begin{tikzpicture}[scale = 0.75, transform shape]
\node at (-0.8,0.5) {$\cdots$};
\node at (0,0) {$q^{6k-10}\talpha$};
\node at (0,1) {$q^{6k-11}\tomega$};
\node at (2.5,0) {$q^{6k-8}\talpha$};
\node at (2.5,1) {$q^{6k-9}\tomega$};
\node at (5,-1) {$q^{6k-7}\tomega$};
\node at (5,0) {$q^{6k-6}\talpha$};
\node at (5,1) {$q^{6k-7}\tomega$};
\node at (7.5,-1) {$q^{6k-5}\tomega$};
\node at (7.5,0) {$q^{6k-6}\talpha$};
\node at (7.5,1) {$q^{6k-5}\tomega$};
\node at (10,0) {$q^{6k-4}\talpha$};
\node at (10,1) {$q^{6k-3}\tomega$};
\node at (12.5,0) {$q^{6k-2}\talpha$};
\node at (15,0) {$q^{6k}\talpha$};
\drawblack{0.7,0}{1.9,0}{}{d}
\drawblack{0.7,0.9}{1.9,0.1}{sloped}{-SD}
\drawblack{3.1,0.8}{4.4,-0.8}{sloped, near start}{D}
\drawblackw{3.1,0}{4.4,0}{near end}{c}
\drawblack{3.1,1}{4.4,1}{}{e}
\drawblack{5.6,1}{6.9,0.1}{sloped}{-S}
\drawblack{5.6,0.8}{6.9,-0.8}{sloped, near start}{-D}
\drawblackw{5.6,-0.9}{6.9,-0.1}{sloped,near start}{S}
\drawblack{5.6,-1}{6.9,-1}{}{e}
\drawblackw{5.6,-0.1}{6.9,-0.9}{sloped, near start}{S}
\drawblack{8.1,0}{9.4,0}{}{c}
\drawblack{8.1,1}{9.4,1}{}{e}
\drawblack{8.1,0.9}{9.4,0.1}{sloped}{S}
\drawblack{10.6,0.9}{11.9,0.1}{sloped}{-S}
\drawblack{10.6,0}{11.9,0}{}{d}
\drawblack{13.1,0}{14.4,0}{}{c}
\end{tikzpicture}
\]
and which starts with
\[
\begin{tikzpicture}
\node at (0,1) {$q^{-2k-1}\tomega$};
\node at (2.5,0) {$q^{-1}\tomega$};
\node at (2.5,1) {$q^{-2k+1}\tomega$};
\node at (5,0) {$\talpha$};
\node at (5,1) {$q^{-2k+3}\tomega$};
\node at (7.5,0) {$q^2\talpha$};
\node at (7.5,1) {$q^{-2k+5}\tomega$};
\node at (8.5,0.5) {$\cdots$};
\drawblack{0.7,0.9}{2,0.1}{sloped}{D^k}
\drawblack{0.7,1}{1.8,1}{}{e}
\drawblack{3,0}{4.6,0}{}{S}
\drawblack{3.2,0.9}{4.6,0.1}{sloped}{-SD^{k-1}}
\drawblack{5.4,0}{7,0}{}{c}
\drawblack{5.7,1}{6.8,1}{}{e}
\drawblack{5.7,0.9}{7,0.1}{sloped}{SD^{k-1}}
\end{tikzpicture}
\]
\end{proposition}
The shown parts of $\tcal{C}_{k,-l}$ implicitly assume $k\geq 2$ and $l = 4k+1$. For $l < 4k+1$ one has to consider the subcomplex where in the top row the objects $q^{6k-2r+1}\tomega$ with $r > l$ are missing. For $l > 4k+1$ the top row will continue further to the left with objects of the form $u^{4k-r}q^{6k-2r+1}\tomega$ for $4k+2\leq r \leq l$ and morphisms $e$ starting at objects in odd homological degrees. Also, for $k=1$ one has to remove the objects $q^{6k-2j}\talpha$ for $j \geq 3$ and $q^{6k-5}\tomega$, and overlap the beginning part with the end part.

\begin{proof}
In addition to the Gaussian elimination that can be performed on the identity morphism between the $q^{6k-1}\alpha$ objects, we do the same Gaussian eliminations in $C_L(\mathcal{B}_{3k})$ that were done in Section \ref{sec:toruslinks}. It is straightforward to check that the objects $q^{6k-3-2j}\tomega$ map into $\tcal{B}_{3k}$ as claimed.
\end{proof}

\begin{theorem}\label{thm:decoomega4}
Let $k,l\geq 1$ and $w=(ab)^ka^{-l}$. Then there is a chain homotopy equivalence
\[
\CBN(L_w;\Z[h]) \simeq C_0\oplus C_1 \oplus C_2
\]
as cochain complexes over $\Z[h]$, such that $C_1$ consists only of direct summands of suitably shifted $A(1)$ complexes, $C_2$ consists only of direct summands of suitably shifted $A(2)$ complexes, and $C_0$ consists of two copies of $A$, suitably shifted, if $l$ is odd, and four copies if $l$ is even.
\end{theorem}

\begin{proof}
Let $G\colon \Cob(D^2_2)\to\mathfrak{Mod}^q_{\Z[X,h]/(x^-Xh)}$ be the functor obtained by composing $C_R$ with the identification functor. Then $G(c)=0$, and it remains to analyze the image under $G$ of the following three diagrams.
\begin{equation}
\label{eq:startdiag}
\begin{tikzpicture}[baseline={([yshift=-.5ex]current bounding box.center)}]
\node at (0,1) {$q^{-2k-1}\tomega$};
\node at (3,1) {$q^{-2k+1}\tomega$};
\node at (3,0) {$q^{-1}\tomega$};
\node at (6,0) {$\talpha$};
\drawblack{0.7,1}{2.2,1}{}{e}
\drawblack{0.7,0.9}{2.5,0}{sloped}{D^k}
\drawblack{3.8,1}{5.7,0.1}{sloped}{-SD^{k-1}}
\drawblack{3.5,0}{5.7,0}{}{S}
\end{tikzpicture}
\end{equation}

\begin{equation}
\label{eq:easydiag}
\begin{tikzpicture}[baseline={([yshift=-.5ex]current bounding box.center)}]
\node at (0,1) {$q^{6k-5-8i}\tomega$};
\node at (3,1) {$q^{6k-3-8i}\tomega$};
\node at (3,0) {$q^{6k-4-6i}\talpha$};
\node at (6,0) {$q^{6k-2-6i}\talpha$};
\drawblack{0.9,1}{2.1,1}{}{e}
\drawblack{0.9,0.9}{2.1,0.1}{sloped}{SD^i}
\drawblack{3.9,0.9}{5.1,0.1}{sloped}{-SD^i}
\drawblack{3.9,0}{5.1,0}{}{d}
\end{tikzpicture}
\end{equation}
for $i=0,\ldots,k-1$, and
\begin{equation}
\label{eq:diffdiag}
\begin{tikzpicture}[baseline={([yshift=-.5ex]current bounding box.center)}]
\node at (0,1) {$q^{6k-9-8i}\tomega$};
\node at (3,1) {$q^{6k-7-8i}\tomega$};
\node at (3,0) {$q^{6k-6-6i}\talpha$};
\node at (3,-1) {$q^{6k-7-6i}\tomega$};
\node at (6,0) {$q^{6k-6-6j}\talpha$};
\node at (6,-1) {$q^{6k-5-6i}\tomega$};
\drawblack{0.9,1}{2.1,1}{}{e}
\drawblack{0.9,0.9}{2.1,-0.9}{sloped}{D^{i+1}}
\drawblack{3.9,0.9}{5.1,0.1}{sloped}{-SD^i}
\drawblack{3.9,0.8}{5.1,-0.8}{sloped}{-D^{i+1}}
\drawblackw{3.9,-0.9}{5.1,-0.1}{sloped, near start}{S}
\drawblackw{3.9,0}{5.1,-0.9}{sloped, near start}{S}
\drawblack{3.9,-1}{5.1,-1}{}{e}
\end{tikzpicture}
\end{equation}
for $i=0,\ldots,k-2$.

First observe that applying $G$ to $\tomega\stackrel{D^i}{\longrightarrow} q^{2i}\tomega$ with $i\geq 1$ results in
\[
\begin{tikzpicture}
\node at (0,1.5) {$q A$};
\node at (0,0) {$q^{-1} A$};
\node at (4,1.5) {$q^{2i+1}A$};
\node at (4,0) {$q^{2i-1}A$};
\drawblack{0.5,1.5}{3.3,1.5}{}{(X-h)^i}
\drawblack{0.5,1.4}{3.3,0.1}{sloped}{(2X-h)^{i-1}}
\drawblack{0.5,0}{3.3,0}{}{X^i}
\end{tikzpicture}
\]
which can be checked by induction, using that $(X-h)^i+X^i = (2X-h)^i$. So applying $G$ to (\ref{eq:startdiag}) leads to
\[
\begin{tikzpicture}
\node at (0,3) {$q^{-2k}A$};
\node at (0,2) {$q^{-2k-2}A$};
\node at (4,3) {$q^{-2k+2}A$};
\node at (4,2) {$q^{-2k}A$};
\node at (4,1) {$A$};
\node at (4,0) {$q^{-2}A$};
\node at (8,0.5) {$A$};
\drawblack{0.6,3}{3.2,3}{}{2X-h}
\drawblack{0.7,2}{3.35,2}{near end}{2X-h}
\drawblackw{0.6,2.9}{3.7,1.1}{sloped, very near end}{(X-h)^k}
\drawblackw{0.6,2.7}{3.4,0.2}{sloped, near end}{(2X-h)^{k-1}}
\drawblack{0.7,1.9}{3.4,0}{sloped}{X^k}
\drawblack{4.6,0}{7.7,0.4}{sloped}{X}
\drawblack{4.3,1}{7.7,0.5}{sloped}{\id}
\drawblack{4.7,2}{7.7,0.6}{sloped}{-X^k}
\drawblack{4.8,3}{7.7,0.7}{sloped}{-(2X-h)^{k-1}}
\end{tikzpicture}
\]
After Gaussian elimination of the identity morphism on the right this turns into
\[
\begin{tikzpicture}
\node at (0,2) {$q^{-2k}A$};
\node at (0,0) {$q^{2k-2}A$};
\node at (3.5,2) {$q^{-2k+2}A$};
\node at (3.5,1) {$q^{-2}A$};
\node at (3.5,0) {$q^{-2k}A$};
\drawblack{0.6,2}{2.7,2}{}{2X-h}
\drawblack{0.6,1.9}{2.9,1.1}{sloped}{(2X-h)^{k-1}}
\drawblack{0.7,0}{2.8,0}{}{2X-h}
\drawblack{0.7,0.1}{2.9,0.9}{sloped}{X^k}
\end{tikzpicture}
\]
If $k=1$ we can do another Gaussian elimination along the $(2X-h)^0=\id$, resulting in one copy of $A$ and one copy of $A(1)$. If $k\geq 2$ we can remove the $X^k$ morphism and the $(2X-h)^{k-1}$ morphism with a change of basis, resulting in one copy of $q^{-2}A$ and two copies of $A(1)$ (shifted by $q^{-2k}$ and $q^{-2k-2}$).

Depending on $l$, we may only have one or no object in the top row of (\ref{eq:startdiag}). In this case we can still perform the first cancellation of $\id$, and the complex will have the required form.

Applying $G$ to (\ref{eq:easydiag}) is slightly simpler, we get the diagram in Figure \ref{fig:twoxminush}.
\begin{figure}[ht]
\begin{tikzpicture}
\node at (0,2) {$q^{6k-4-8i}A$};
\node at (0,1) {$q^{6k-6-8i}A$};
\node at (4,2) {$q^{6k-2-8i}A$};
\node at (4,1) {$q^{6k-4-8i}A$};
\node at (4,-0.5) {$q^{6k-4-6i}A$};
\node at (8,-0.5) {$q^{6k-2-6i}A$};
\drawblack{0.9,1.8}{3.1,-0.3}{sloped, near end}{(2X-h)^i}
\drawblack{0.9,2}{3.1,2}{}{2X-h}
\drawblackw{0.9,1}{3.1,1}{near end}{2X-h}
\drawblack{0.9,0.9}{3.1,-0.4}{sloped}{X^{i+1}}
\drawblack{4.9,1.9}{7.1,-0.3}{sloped}{-(2X-h)^i}
\drawblack{4.9,0.9}{7.1,-0.4}{sloped}{-X^{i+1}}
\drawblack{4.9,-0.5}{7.1,-0.5}{}{2X-h}
\end{tikzpicture}
\caption{\label{fig:twoxminush}The diagram after applying $G$ to (\ref{eq:easydiag}).}
\end{figure}

If $i=0$, we can perform two Gaussian eliminations, leaving us with one summand of a complex $A(1)$ (suitably shifted). Otherwise we get three copies of $A(1)$ after changing the basis. Again, depending on $l$ there may only be one or no objects in the top row of (\ref{eq:easydiag}), in which case a similar argument applies.

Applying $G$ to (\ref{eq:diffdiag}) is slightly more involved; we get the diagram in Figure \ref{fig:htwodiag}.
\begin{figure}[ht]
\begin{tikzpicture}
\node at (0,4.5) {$q^{6k-8-8i}A$};
\node at (0,3.5) {$q^{6k-10-8i}A$};
\node at (5,0) {$q^{6k-8-6i}A$};
\node at (5,1) {$q^{6k-6-6i}A$};
\node at (5,2) {$q^{6k-6-6i}A$};
\node at (5,3.5) {$q^{6k-8-8i}A$};
\node at (5,4.5) {$q^{6k-6-8i}A$};
\node at (10,0) {$q^{6k-6-6i}A$};
\node at (10,1) {$q^{6k-4-6i}A$};
\node at (10,2) {$q^{6k-6-6i}A$};
\drawblack{0.9,4.5}{4.1,4.5}{}{2X-h}
\drawblack{0.9,4.3}{4.1,1.2}{sloped, near end}{(X-h)^{i+1}}
\drawblack{0.9,4.1}{4.1,0.2}{sloped, near end}{(2X-h)^i}
\drawblackw{0.9,3.5}{4.1,3.5}{near end}{2X-h}
\drawblack{0.9,3.3}{4.1,0}{sloped, below}{X^{i+1}}
\drawblack{5.9,4.5}{9.1,2.3}{sloped}{-(2X-h)^i}
\drawblack{5.9,4.3}{9.1,1.4}{sloped}{-(X-h)^{i+1}}
\drawblack{5.9,4.1}{9.1,0.4}{sloped, below}{(2X-h)^i}
\drawblackw{5.9,3.4}{9.1,2.1}{sloped, very near start}{-X^{i+1}}
\drawblack{5.9,3.2}{9.1,0.2}{sloped, near start, below}{-X^{i+1}}
\drawblackw{5.9,2}{9.1,1.2}{sloped, near start}{X-h}
\drawblackw{5.9,0.2}{9.1,1.7}{sloped, near start, below}{X}
\drawblack{5.9,0}{9.1,0}{}{2X-h}
\drawblackw{5.9,1}{9.1,1}{very near start, below}{2X-h}
\drawblackw{5.9,1.1}{9.1,1.9}{sloped, very near start}{\id}
\drawblackw{5.9,1.8}{9.1,0.1}{sloped, near start}{\id}
\end{tikzpicture}
\caption{\label{fig:htwodiag}The diagram after applying $G$ to (\ref{eq:diffdiag}).}
\end{figure}

Gaussian elimination on the two $\id$-morphisms as in Section \ref{sec:toruslinks} leads to the diagram in Figure \ref{fig:htworeduced}.
\begin{figure}[ht]
\begin{tikzpicture}
\node at (0,2.5) {$q^{6k-8-8i}A$};
\node at (0,1.5) {$q^{6k-10-8i}A$};
\node at (4,0) {$q^{6k-8-6i}A$};
\node at (4,1.5) {$q^{6k-8-8i}A$};
\node at (4,2.5) {$q^{6k-6-8i}A$};
\node at (8,0) {$q^{6k-4-6i}A$}; 
\drawblack{0.9,2.5}{3.1,2.5}{}{2X-h}
\drawblack{0.9,2.3}{3.1,0.2}{sloped, near end}{(2X-h)^i}
\drawblackw{0.9,1.5}{3.1,1.5}{near end}{2X-h}
\drawblack{0.9,1.3}{3.1,0}{sloped}{X^{i+1}}
\drawblack{4.9,2.4}{7.1,0.2}{sloped}{(2X-h)^{i+1}}
\drawblack{4.9,1.4}{7.1,0.1}{sloped}{X^{i+2}}
\drawblack{4.9,0}{7.1,0}{}{-h^2}
\end{tikzpicture}
\caption{\label{fig:htworeduced}Figure \ref{fig:htwodiag} after two Gaussian eliminations.}
\end{figure}
If $i=0$ we can perform one more Gaussian elimination, and after a change of basis we get two copies of $A(1)$, suitably shifted. If $l\leq 3$, we do not have the four objects in the upper rows, and we simply get one summand $A(2)$. For $l = 4$ the two objects in the left corner are missing, and we get one copy of $A(1)$ and two copies of $A$.

If $i\geq 1$, we can remove the diagonal morphisms with change of bases, and we get two summands of $A(1)$ and one summand of $A(2)$. Again, if only one or no object is in the top row of (\ref{eq:diffdiag}), we get a copy of $A(2)$, and possibly two copies of $A$.
\end{proof}

\section{The Bar-Natan complex for proper alternating words}

Recall that a braid word is called proper alternating, if it is of the form 
\[
w=a^{-n_1}b^{m_1}\cdots a^{-n_j}b^{m_j}
\]
 with $j\geq 1$ and $n_i,m_i\geq 1$ for $i=1,\ldots,j$. We define
 \[
 n(w) = \sum_{i=1}^j n_j \hspace{1cm}\mbox{and}\hspace{1cm}
 m(w) = \sum_{i=1}^j m_j.
 \]
 In the special case $j=1$ the tangle complex $\CBN(T_{a^{-n}b^m};\Z[h])$ is chain homotopy equivalent to $q^{m-n}\mathcal{D}_{a^{-n}}\otimes \mathcal{D}_{b^m}$. In Figure \ref{fig:altcomplex} we show a special case of such a complex.
 \begin{figure}[ht]
 \begin{tikzpicture}
 \node at (0,3) {$q^{-6}\alpha$};
 \node at (2,2) {$q^{-4}\alpha$};
 \node at (2,4) {$q^{-5}\gamma$};
 \node at (4,1) {$q^{-2}\alpha$};
 \node at (4,3) {$q^{-3}\gamma$};
 \node at (4,5) {$q^{-3}\gamma$};
 \node at (6,0) {$q^{-1}\omega$};
 \node at (6,2) {$q^{-1}\gamma$};
 \node at (6,4) {$q^{-1}\gamma$};
 \node at (8,1) {$\beta$};
 \node at (8,3) {$q \gamma$};
 \node at (10,2) {$q^2\beta$};
 \drawblack{0.5,3.25}{1.5,3.75}{sloped}{-S}
 \drawblack{0.5,2.75}{1.5,2.25}{sloped}{d}
 \drawblack{2.5,4.25}{3.5,4.75}{sloped}{-c}
 \drawblack{2.5,3.75}{3.5,3.25}{sloped}{d}
 \drawblack{2.5,2.25}{3.5,2.75}{sloped}{S}
 \drawblack{2.5,1.75}{3.5,1.25}{sloped}{c}
 \drawblack{4.5,4.75}{5.5,4.25}{sloped}{d}
 \drawblack{4.5,3.25}{5.5,3.75}{sloped}{c}
 \drawblack{4.5,2.75}{5.5,2.25}{sloped}{c}
 \drawblack{4.5,1.25}{5.5,1.75}{sloped}{-S}
 \drawblack{4.5,0.75}{5.5,0.25}{sloped}{S}
 \drawblack{6.5,3.75}{7.5,3.25}{sloped}{c}
 \drawblack{6.5,2.25}{7.5,2.75}{sloped}{-c}
 \drawblack{6.5,1.75}{7.5,1.25}{sloped}{S}
 \drawblack{6.5,0.25}{7.5,0.75}{sloped}{S}
 \drawblack{8.5,2.75}{9.5,2.25}{sloped}{S}
 \drawblack{8.5,1.25}{9.5,1.75}{sloped}{c}
 \end{tikzpicture}
 \caption{\label{fig:altcomplex}The complex $\CBN(T_{a^{-3}b^2};\Z[h])$ over $\Cob(D^3_3)$.}
 \end{figure}
 
For $j\geq 2$ we can describe the tangle complex for $T_w$ by continuing to tensor with $q^{m_i-n_i}\mathcal{D}_{a^{-n_i}}\otimes \mathcal{D}_{b^{m_i}}$. However, it does not seem clear to us that this leads to a nice formula from which the Khovanov homology can be read off. Nevertheless, we can show that there is a cochain complex chain homotopy equivalent to $\CBN(T_w;\Z[h])$ for any proper alternating word which is nice enough for our purposes.

\begin{proposition}\label{prp:niceproperalt}
Let $w$ be a proper alternating word. Then $\CBN(T_w;\Z[h])$ is chain homotopy equivalent to a cochain complex $\mathcal{C}_w$ over $\Cob(D^3_3)$ satisfying the following properties.
\begin{enumerate}
\item All generators have smoothings without loops.
\item Morphisms are of the form $q^j\varphi_1\to q^{j+i}\varphi_2$ with $i\in \{0,1,2\}$. Furthermore,
\begin{enumerate}
\item if $i=0$, then $\varphi_1=\varphi_2$ and the morphism is a multiple of $\id$.
\item if $i=1$, the morphism is $\pm S$, with $S$ a surgery.
\item if $i=2$, the morphism is a linear combination of dottings and $h$.
\end{enumerate}
\item There is a unique generator whose smoothing is $\omega$, and it is in bidegree $(0,m(w)-n(w))$.
\item There is a unique morphism with domain $q^{m(w)-n(w)}\omega$, and it is a surgery to a generator with smoothing $\beta$.
\item There is a finitely generated free subcomplex $\tcal{C}_w$ which fits into a short exact sequence
\[
0\longrightarrow \tcal{C}_w\longrightarrow \mathcal{C}_w\longrightarrow q^{m(w)-n(w)}\mathcal{D}_{a^{-n(w)}}\longrightarrow 0.
\]
\end{enumerate}
\end{proposition}

\begin{proof}
We are going to show this for any words $w$ that start in $a^{-1}$ and contain at least one $b$. The induction start is for words of the form $a^{-n}b^m$ and we can see from Figure \ref{fig:altcomplex} that the statements hold in this case.

Now let $w$ be a word starting in $a^{-1}$ and containing a $b$ for which the cochain complex $\mathcal{C}_w$ with properties (1)-(5) exists. We need to show that the complexes $\mathcal{C}_{wa^{-1}}$ and $\mathcal{C}_{wb}$ with (1)-(5) exist.

We have
\[
\CBN(T_{wa^{-1}};\Z[h])\simeq \mathcal{C}_w\otimes \CBN(T_{a^{-1}};\Z[h]),
\]
so we need to analyze the tensor product. The complex on the right consists of a surgery $q^{-2}\alpha \to q^{-1}\omega$, and tensoring with $\omega$ does not change the smoothings. Tensoring with $\alpha$ can create a loop, so we first need to deloop any objects of the form $\alpha\otimes \alpha$ and $\delta\otimes \alpha$. We need to check that after delooping all morphisms are of the form as in (2). If the $q$-grading changes by $0$ or $2$, domain and codomain smoothings are the same, and the morphism keeps the same form, unless it consists of dotting a circle that has been delooped. If we denote the smoothing by $\chi$, we then get two objects $q\chi$ and $q^{-1}\chi$ for the domains, and $q^3\chi$ and $q\chi$ for the codomain. The morphism arising from dotting the circle leads to an $\id$ morphism between the $q\chi$ objects.

A surgery in $\mathcal{C}_w$ remains a surgery, if the surgery does not involve a delooped circle, and if it involves a circle, it turns into an $\id$ morphism and a linear combination of a dotting and possibly an $h$ after delooping. Similarly, the surgery in $\CBN(T_{a^{-1}};\Z[h])$ remains a surgery, if there is no delooping, and produces an $\id$ morphism as well as a linear combination of a dotting and an $h$.

The only way we can get a smoothing $\omega$ is through the tensor $q^{m(w)-n(w)}\omega\otimes q^{-1}\omega = q^{m(wa^{-1})-n(wa^{-1})}\omega$, so (3) is satisfied. There is also still one morphism going out of this object, and it remains a surgery to a $\beta$ smoothing.

To satisfy (5) we need to perform several Gaussian eliminations. These Gaussian eliminations will only involve objects in $\mathcal{C}_w\otimes q^{-2}\alpha$, in fact in the quotient $q^{m(w)}\mathcal{D}_{a^{-n(w)}}$. The tensor complex looks as in
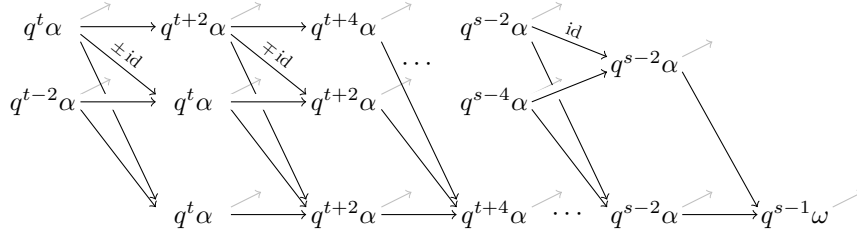
\[
\begin{tikzpicture}
\node at (0,2.5) {$q^{t}\alpha$};
\node at (0,1.5) {$q^{t-2}\alpha$};
\node at (2,0) {$q^{t}\alpha$};
\node at (2,2.5) {$q^{t+2}\alpha$};
\node at (2,1.5) {$q^{t}\alpha$};
\node at (4,0) {$q^{t+2}\alpha$};
\node at (4,2.5) {$q^{t+4}\alpha$};
\node at (4,1.5) {$q^{t+2}\alpha$};
\node at (6,0) {$q^{t+4}\alpha$};
\node at (5,2) {$\cdots$};
\node at (7,0) {$\cdots$};
\node at (6,2.5) {$q^{s-2}\alpha$};
\node at (6,1.5) {$q^{s-4}\alpha$};
\node at (8,0) {$q^{s-2}\alpha$};
\node at (8,2) {$q^{s-2}\alpha$};
\node at (10,0) {$q^{s-1}\omega$};
\drawgray{0.5,1.6}{0.9,1.8}
\drawgray{0.5,2.6}{0.9,2.8}
\drawgray{2.5,1.6}{2.9,1.8}
\drawgray{2.5,2.6}{2.9,2.8}
\drawgray{4.5,1.6}{4.9,1.8}
\drawgray{4.5,2.6}{4.9,2.8}
\drawgray{6.5,1.6}{6.9,1.8}
\drawgray{6.5,2.6}{6.9,2.8}
\drawgray{2.5,0.1}{2.9,0.3}
\drawgray{4.5,0.1}{4.9,0.3}
\drawgray{6.5,0.1}{6.9,0.3}
\drawgray{8.5,0.1}{8.9,0.3}
\drawgray{8.5,2.1}{8.9,2.3}
\drawgray{10.5,0.1}{10.9,0.3}
\drawblack{0.5,1.4}{1.5,0.1}{}{}
\drawblack{0.5,2.3}{1.5,0.2}{}{}
\drawblackw{0.5,1.5}{1.5,1.5}{}{}
\drawblack{0.5,2.4}{1.5,1.6}{sloped}{\pm\id}
\drawblack{0.5,2.5}{1.5,2.5}{}{}
\drawblack{2.5,1.4}{3.5,0.1}{}{}
\drawblack{2.5,2.3}{3.5,0.2}{}{}
\drawblackw{2.5,1.5}{3.5,1.5}{}{}
\drawblack{2.5,2.4}{3.5,1.6}{sloped}{\mp\id}
\drawblack{2.5,2.5}{3.5,2.5}{}{}
\drawblack{2.5,0}{3.5,0}{}{}
\drawblack{4.5,0}{5.5,0}{}{}
\drawblack{4.5,1.4}{5.5,0.1}{}{}
\drawblack{4.5,2.3}{5.5,0.2}{}{}
\drawblack{6.5,2.3}{7.5,0.2}{}{}
\drawblack{6.5,1.4}{7.5,0.1}{}{}
\drawblackw{6.5,1.5}{7.5,1.9}{}{}
\drawblack{6.5,2.5}{7.5,2.1}{sloped}{\id}
\drawblack{8.5,0}{9.5,0}{}{}
\drawblack{8.5,1.9}{9.5,0.1}{}{}
\end{tikzpicture}
\]
Here $t=1+m(w)-3n(w)$ and $s = m(w)-n(w)$. The black morphisms are a mirrored version of the morphisms used in Proposition \ref{prp:toruslink2}, and after Gaussian elimination of the shown $\id$-morphisms we get condition (5) with $\mathcal{D}_{a^{-n(w)-1}}$. Note that there are no further morphisms going into the codomains of the $\id$-morphisms apart from the ones shown, so we do not get any further complicated morphisms and the properties (1)-(4) still hold for the resulting complex.

The argument for $\mathcal{C}_w\otimes \CBN(T_b;\Z[h])$ is similar.  Indeed, conditions (1)-(3) and (5) hold after delooping. The main difference is that now we have two morphisms going out of the object $q^{m(w)-n(w)+1}\omega$, both surgeries into objects $q^{m(w)-n(w)+2}\beta$. But one of the latter objects can be cancelled via an $\id$-morphism in $\mathcal{C}_w\otimes q^2\beta$ that used to be the surgery from $\omega$ to $\beta$. Indeed, this is the only cancellation we need to do, and it does not affect any of the properties (1)-(5).
\end{proof}

\begin{proposition}\label{prp:altBarNatan}
Let $w$ be a proper alternating word, and $L_w$ the braid closure of $w$ with basepoint on the middle strand. Then there is a finitely generated free $A$-cochain complex $C_w$ with
\[
\CBN(L_w;\Z[h]) \simeq C_w \oplus q^{m(w)-n(w)}A
\]
as $A$-cochain complexes, and such that there is a short exact sequence of finitely generated free $A$-cochain complexes
\[
0\longrightarrow \tilde{C}_w\longrightarrow C_w\longrightarrow S_w\longrightarrow 0
\]
with
\[
S_w\cong \bigoplus_{i=1}^{(n(w)-1)/2} (u^{-1-2i} q^{m(w)-n(w)-4i} A(1)\oplus u^{-1-2i} q^{m(w)-n(w)-4i} A(1))
\]
if $n(w)$ is odd, and
\begin{align*}
S_w\cong &\, u^{-n(w)}q^{m(w)-3n(w)}A \oplus u^{-n(w)}q^{2+m(w)-3n(w)}A \oplus \\
& \bigoplus_{i=1}^{(n(w)-2)/2} (u^{-1-2i} q^{m(w)-n(w)-4i} A(1)\oplus u^{-1-2i} q^{m(w)-n(w)-4i} A(1))
\end{align*}
if $n(w)$ is even.
\end{proposition}

\begin{proof}
Let $\mathcal{C}_w$ be the cochain complex over $\Cob(D^3_3)$ from Proposition \ref{prp:niceproperalt}. We can visualize this complex as
\[
\begin{tikzpicture}
\node at (0,0) {$\cdots$};
\node at (1,0) {$q^{t-7}\alpha$};
\node at (3,0) {$q^{t-5}\alpha$};
\node at (5,0) {$q^{t-3}\alpha$};
\node at (7,0) {$q^{t-1}\alpha$};
\node at (9,0) {$q^t\omega$};
\node at (11,1) {$q^{t+1}\beta$};
\drawgray{11.5,1.1}{11.9,1.5}
\drawgray{10.1,1.5}{10.5,1.1}
\drawblack{9.4,0}{10.5,0.9}{sloped}{S}
\drawblack{7.5,0}{8.6,0}{}{S}
\drawblack{5.5,0}{6.5,0}{}{c}
\drawblack{3.5,0}{4.5,0}{}{d}
\drawblack{1.5,0}{2.5,0}{}{c}
\drawgray{1.5,0.1}{2,0.6}
\drawgray{3.5,0.1}{4,0.6}
\drawgray{5.5,0.1}{6,0.6}
\drawgray{7.5,0.1}{8,0.6}
\end{tikzpicture}
\]
where $t=m(w)-n(w)$ and we show parts of the quotient complex $q^{m(w)}\mathcal{D}_{a^{-n(w)}}$ as well as the one object being mapped to by $q^t\omega$. Gray arrows indicate morphisms into the subcomplex $\tcal{C}_w$.

Applying the functor $C_L\colon \Cob(D^3_3)\to \Cob(B^2_2)$ leads to the complex
\[
\begin{tikzpicture}
\node at (0,0) {$\cdots$};
\node at (1,0) {$q^{t-7}\tomega$};
\node at (3,0) {$q^{t-5}\tomega$};
\node at (5,0) {$q^{t-3}\tomega$};
\node at (7,0) {$q^{t-1}\tomega$};
\node at (9,0.5) {$q^{t+1}\tomega$};
\node at (9,-0.5) {$q^{t-1}\tomega$};
\node at (11,1) {$q^{t+2}\talpha$};
\node at (11,0) {$q^{t}\talpha$};
\drawgray{11.5,1.1}{11.9,1.5}
\drawgray{11.5,0.1}{11.9,0.5}
\drawgray{10.1,1.5}{10.5,1.1}
\drawgray{10.1,0.5}{10.5,0.1}
\drawblack{9.5,0.5}{10.5,0.9}{sloped}{S}
\drawblack{9.5,-0.5}{10.5,-0.1}{sloped}{S}
\drawblack{7.5,0.1}{8.5,0.5}{sloped}{)(-h\hspace{-22pt}\bullet}
\drawblack{7.5,-0.1}{8.5,-0.5}{sloped}{\id}
\drawblack{5.5,0}{6.5,0}{}{0}
\drawblack{3.5,0}{4.5,0}{}{e}
\drawblack{1.5,0}{2.5,0}{}{0}
\drawgray{1.5,0.1}{2,0.6}
\drawgray{3.5,0.1}{4,0.6}
\drawgray{5.5,0.1}{6,0.6}
\drawgray{7.5,0.2}{8,1}
\end{tikzpicture}
\]
Gaussian elimination of the identity morphism, then applying $G$ leads to the complex
\[
\begin{tikzpicture}
\node at (0,0.5) {$\cdots$};
\node at (1,1) {$q^{t-6}A$};
\node at (1,0) {$q^{t-8}A$};
\node at (3,1) {$q^{t-4}A$};
\node at (3,0) {$q^{t-6}A$};
\node at (5,1) {$q^{t-2}A$};
\node at (5,0) {$q^{t-4}A$};
\node at (7,0.5) {$0$};
\node at (9,1) {$q^{t+2}A$};
\node at (9,0) {$q^tA$};
\node at (11,1) {$q^{t+2}A$};
\node at (11,0) {$q^{t}A$};
\drawgray{1.5,1.1}{1.9,1.5}
\drawgray{1.5,0.1}{1.9,0.5}
\drawgray{3.5,1.1}{3.9,1.5}
\drawgray{3.5,0.1}{3.9,0.5}
\drawgray{5.5,1.1}{5.9,1.5}
\drawgray{5.5,0.1}{5.9,0.5}
\drawgray{11.5,1.1}{11.9,1.5}
\drawgray{11.5,0.1}{11.9,0.5}
\drawgray{10.1,1.5}{10.5,1.1}
\drawgray{10.1,0.5}{10.5,0.1}
\drawblack{3.5,1}{4.5,1}{}{2X-h}
\drawblack{3.5,0}{4.5,0}{}{2X-h}
\drawblack{9.5,1}{10.5,1}{}{\id}
\drawblack{9.5,0}{10.5,0.9}{sloped}{X}
\end{tikzpicture}
\]
We can perform one more Gaussian elimination here without creating new morphisms, and the result is the cochain complex pictured in Figure \ref{fig:theC_w}.
\begin{figure}[ht]
\begin{tikzpicture}
\node at (0,0.5) {$\cdots$};
\node at (1,1) {$q^{t-6}A$};
\node at (1,0) {$q^{t-8}A$};
\node at (3,1) {$q^{t-4}A$};
\node at (3,0) {$q^{t-6}A$};
\node at (5,1) {$q^{t-2}A$};
\node at (5,0) {$q^{t-4}A$};
\node at (7,0.5) {$0$};
\node at (9,0.5) {$q^tA$};
\node at (11,0.5) {$q^{t}A$};
\drawgray{1.5,1.1}{1.9,1.5}
\drawgray{1.5,0.1}{1.9,0.5}
\drawgray{3.5,1.1}{3.9,1.5}
\drawgray{3.5,0.1}{3.9,0.5}
\drawgray{5.5,1.1}{5.9,1.5}
\drawgray{5.5,0.1}{5.9,0.5}
\drawgray{11.5,0.6}{11.9,1}
\drawgray{10.1,1}{10.5,0.6}
\drawblack{3.5,1}{4.5,1}{}{2X-h}
\drawblack{3.5,0}{4.5,0}{}{2X-h}
\end{tikzpicture}
\caption{\label{fig:theC_w}The complex $C_w$ with a direct summand $q^tA$.}
\end{figure}
The second $q^tA$ object from the right has no further morphisms going in or out, and we call the remaining direct summand $C_w$. The objects to the left of the $0$ form the quotient complex $S_w$, which is a direct sum of complexes of type $A(1)$ and (for even $n(w)$) of type $A(0)$, suitably shifted.
\end{proof}

\begin{corollary}\label{cor:redKhovalt}
Let $w$ be a proper alternating word, and $L_w$ the braid closure of $w$ with basepoint on the middle strand. The reduced Khovanov homology of $L_w$ is free abelian and concentrated in bidegrees $(i, 2i+m(w)-n(w))$.
\end{corollary}

\begin{proof}
Since $w$ is proper, $L_w$ is a non-split alternating link. By \cite[Thm 1]{MR2509750} the reduced Khovanov homology of $L_w$ is free abelian and concentrated in bidegrees $(i,2i+s)$, where $s$ is the signature of $L_w$. By Proposition \ref{prp:altBarNatan} we get a copy of $\Z$ in bidegree $(0,m(w)-n(w))$, hence $s=m(w)-n(w)$.
\end{proof}

\section{The Bar-Natan complex for words in $\Omega_6^+$}

Let $k\geq 1$ and $w$ a proper alternating word, so that $(ab)^kw\in \Omega_6^+$. Then
\[
\CBN(T_{(ab)^kw};\Z[h]) \simeq q^{6k}\mathcal{B}_{3k} \otimes \mathcal{C}_w,
\]
where $\mathcal{C}_w$ is the cochain complex from Proposition \ref{prp:niceproperalt}. Since the subcomplex $\tcal{C}_w$ of $\mathcal{C}_w$ contains no generators with smoothing $\omega$, we get
\[
q^{6k}\mathcal{B}_{3k} \otimes \tcal{C}_w \simeq u^{4k}q^{12k}\tcal{C}_w,
\]
by Lemma \ref{lm:threektensor} and the way the possible morphisms behave under the Gaussian eliminations. Also, for the quotient complex $\mathcal{C}_w/\tcal{C}_w$ tensoring with $\mathcal{B}_{3k}$ has the effect as in Section \ref{sec:omega_four}. In particular, we get the following combination of Proposition \ref{prp:formomega5} and Proposition \ref{prp:formomega4}.

\begin{proposition}\label{prp:formomega6}
Let $w$ be a proper alternating word and $k\geq 1$. Then the tangle complex $\CBN(T_{(ab)^{3k}w};\Z[h])$ is chain homotopy equivalent to a complex $\mathcal{C}_{k,w}$ of the form
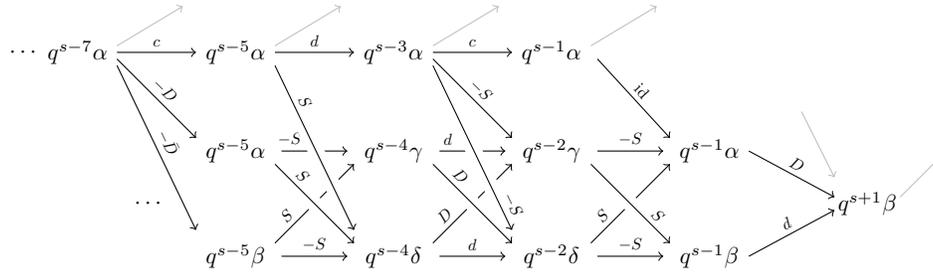
\[
\begin{tikzpicture}[scale = 0.875, transform shape]
\node at (-0.5,1) {$\cdots$};
\node at (-2.4, 3.3) {$\cdots$};
\node at (0.8, 1.8) {$q^{s-5}\alpha$};
\node at (3.2, 1.8) {$q^{s-4}\gamma$};
\node at (5.6, 1.8) {$q^{s-2}\gamma$};
\node at (8, 1.8) {$q^{s-1}\alpha$};
\node at (0.8, 0.2) {$q^{s-5}\beta$};
\node at (3.2, 0.2) {$q^{s-4}\delta$};
\node at (5.6, 0.2) {$q^{s-2}\delta$};
\node at (8, 0.2) {$q^{s-1}\beta$};
\node at (10.4,1) {$q^{s+1}\beta$};
\node at (5.6, 3.3) {$q^{s-1}\alpha$};
\node at (3.2, 3.3) {$q^{s-3}\alpha$};
\node at (0.8, 3.3) {$q^{s-5}\alpha$};
\node at (-1.6, 3.3) {$q^{s-7}\alpha$};
\draw[->] (1.5, 1.8) -- node [above, scale = 0.7, very near start] {$-S$} (2.5, 1.8);
\draw[->] (1.5, 0.2) -- node [above, scale = 0.7] {$-S$} (2.5, 0.2);
\draw[->] (3.9, 1.8) -- node [above, scale = 0.7, very near start] {$d$} (4.9, 1.8);
\draw[->] (3.9, 0.2) -- node [above, scale = 0.7] {$d$} (4.9, 0.2);
\draw[->] (6.3, 1.8) -- node [above, scale = 0.7] {$-S$} (7.3, 1.8);
\draw[->] (6.3, 0.2) -- node [above, scale = 0.7] {$-S$} (7.3, 0.2);
\draw[->] (1.4, 0.4) -- node [above, scale = 0.7, near start, sloped] {$S$} (2.6, 1.6);
\draw[->] (3.8, 0.4) -- node [above, scale = 0.7, near start, sloped] {$D$} (5, 1.6);
\draw[->] (6.2, 0.4) -- node [above, scale = 0.7, near start, sloped] {$S$} (7.4, 1.6);
\draw[-, line width=6pt, color = white] (1.4, 1.6) -- (2.6, 0.4);
\draw[-, line width=6pt, color = white] (3.8, 1.6) -- (5, 0.4);
\draw[-, line width=6pt, color = white] (6.2, 1.6) -- (7.4, 0.4);
\draw[->] (1.4, 1.6) -- node [above, scale = 0.7, near start, sloped] {$S$} (2.6, 0.4);
\draw[->] (3.8, 1.6) -- node [above, scale = 0.7, near start, sloped] {$D$} (5, 0.4);
\draw[->] (6.2, 1.6) -- node [above, scale = 0.7, near end, sloped] {$S$} (7.4, 0.4);
\draw[->] (1.4, 3.3) -- node [above, scale = 0.7] {$d$} (2.6, 3.3);
\draw[->] (3.8,  3.3) -- node [above, scale = 0.7] {$c$} (5, 3.3);
\draw[->] (-1, 3.3) -- node [above, scale = 0.7] {$c$} (0.2, 3.3);
\drawblack{6.3, 3.2}{7.4, 2}{sloped}{\id} 
\drawblackw{3.8, 3.1}{5, 0.6}{sloped, very near end}{-S}
\drawblack{3.8, 3.2}{5, 2}{sloped}{-S}
\drawblackw{1.4, 3.1}{2.6, 0.6}{sloped, near start}{S}
\drawblack{-1,3.2}{0.2,2}{sloped}{-D}
\drawblack{-1,3.1}{0.2,0.6}{sloped}{-\bar{D}}
\drawblack{8.6,0.2}{9.9,0.9}{sloped}{d}
\drawblack{8.6,1.8}{9.9,1.1}{sloped}{D}
\drawgray{-1,3.4}{0,4}
\drawgray{1.4,3.4}{2.4,4}
\drawgray{3.8,3.4}{4.8,4}
\drawgray{6.2,3.4}{7.2,4}
\drawgray{10.9,1.1}{11.5,1.7}
\drawgray{9.4,2.4}{9.9,1.4}
\end{tikzpicture}
\]
which contains $u^{4k}q^{12k}\tcal{C}_w$ as a subcomplex. Here $s= 12k+m(w)-n(w)$.\hfill \qedsymbol
\end{proposition}

The only object from $u^{4k}q^{12k}\tcal{C}_w$ visible is $q^{s+1}\beta$, and gray arrows indicate morphisms to this subcomplex. Note that the objects coming from $\mathcal{B}_{3k}$ do not have any morphisms into $q^{12k}\tcal{C}_w$, except for the two morphisms into $q^{s+1}\beta$.

The picture implicitly assumes $n(w)\geq 4$, but if $n(w)<4$ the behaviour is as in Proposition \ref{prp:formomega4}. Also, if $n(w) > 4$ there are morphisms from objects $q^{s-2t-1}\alpha$ for $t>3$ as in Proposition \ref{prp:formomega4} (as well as gray morphisms starting in those objects).

\begin{theorem}\label{thm:decoomega6}
Let $w$ be a proper alternating word, $k\geq 1$ and $L_{(ab)^{3k}w}$ the link closure associated to the braid word $(ab)^{3k}w$ with basepoint on the middle strand. Then
\[
\CBN(L_{(ab)^{3k}w};\Z[h]) \simeq \check{B}_{3k}\oplus D_w
\]
as $A$-cochain complexes, where
\[
\check{B}_{3k} \cong q^{s-6k-2} A \oplus \bigoplus_{m=0}^{k-2}u^{4m+2}q^{s-6(k-m)+2}A(1) \oplus \bigoplus_{m=1}^{k-2}u^{4m}q^{s-6(k-m)-2}A(2),
\]
if $k\geq 2$, and $\check{B}_3 = 0$. The complex $D_w$ fits into a short exact sequence of $A$-cochain complexes
\begin{equation}\label{eq:shortexactT}
0 \longrightarrow T \longrightarrow D_w \longrightarrow u^{4k}q^{12k}C_w\longrightarrow 0,
\end{equation}
where $C_w$ is the cochain complex from Proposition \ref{prp:altBarNatan}, and
\[
T \cong u^{4k-4}q^{s-8}A(2) \oplus u^{4k-2}q^{s-4}A(1) \oplus u^{4k}q^s A(1),
\]
if $k\geq 2$, and
\[
T \cong q^{s-8}A \oplus u^2q^{s-4}A(1) \oplus u^4q^s A(1),
\]
if $k = 1$. Here $s= 12k +m(w)-n(w)$.
\end{theorem}

\begin{proof}
Let $\mathcal{C}_{k,w}$ be the cochain complex from Proposition \ref{prp:formomega6} and apply the functor $C_L$. The resulting complex looks as in
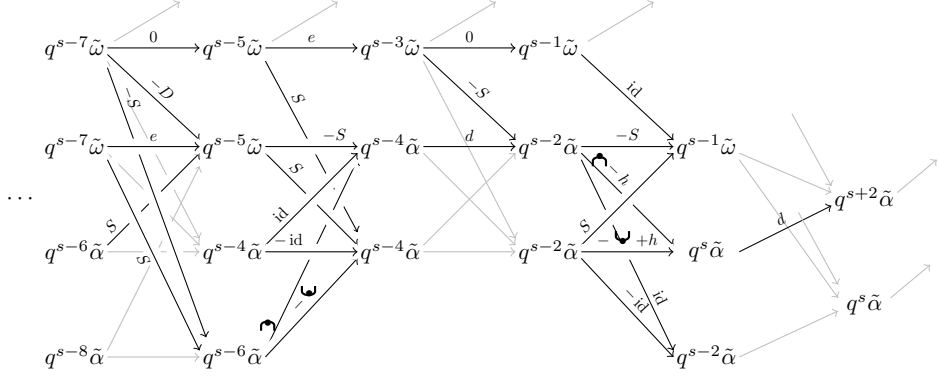
\[
\begin{tikzpicture}[scale = 0.875, transform shape]
\node at (-2.4, 1) {$\cdots$};
\node at (0.8, 1.8) {$q^{s-5}\tomega$};
\node at (-1.6, 1.8) {$q^{s-7}\tomega$};
\node at (-1.6, 0.2) {$q^{s-6}\talpha$};
\node at (-1.6,-1.4) {$q^{s-8}\talpha$};
\node at (3.2, 1.8) {$q^{s-4}\talpha$};
\node at (5.6, 1.8) {$q^{s-2}\talpha$};
\node at (8, 1.8) {$q^{s-1}\tomega$};
\node at (0.8, 0.2) {$q^{s-4}\talpha$};
\node at (3.2, 0.2) {$q^{s-4}\talpha$};
\node at (0.8, -1.4) {$q^{s-6}\talpha$};
\node at (5.6, 0.2) {$q^{s-2}\talpha$};
\node at (8, 0.2) {$q^{s}\talpha$};
\node at (8, -1.4) {$q^{s-2}\talpha$};
\node at (10.4,1) {$q^{s+2}\talpha$};
\node at (10.4,-0.6) {$q^{s}\talpha$};
\node at (5.6, 3.3) {$q^{s-1}\tomega$};
\node at (3.2, 3.3) {$q^{s-3}\tomega$};
\node at (0.8, 3.3) {$q^{s-5}\tomega$};
\node at (-1.6, 3.3) {$q^{s-7}\tomega$};
\drawgray{-1.1,-1.4}{0.3,-1.4}
\drawgray{-1.1,-1.3}{0.3,1.5}
\drawgray{-1.1, 0.2}{0.3, 0.2}
\drawgray{-1.1, 1.7}{0.3, 0.3}
\drawgray{-1.1, 3.1}{0.3, 0.4}
\drawblackw{-1.1, 0.3}{0.3, 1.7}{sloped, very near start}{S}
\drawblackw{-1.1, 1.6}{0.3, -1.3}{sloped, below}{S}
\drawblackw{-1.1, 3}{0.4, -1.1}{sloped, very near start}{-S}
\drawblack{-1.1, 3.2}{0.3,1.9}{sloped}{-D}
\drawblackw{-1.1, 1.8}{0.3,1.8}{}{e}
\drawblack{-1.1, 3.3}{0.3,3.3}{}{0}
\drawblack{1.3, 3.3}{2.7,3.3}{}{e}
\drawblack{1.3, 3.1}{2.7, 0.5}{sloped, near start}{S}
\drawblackw{1.3, -1.3}{2.7, 1.6}{sloped, very near start}{\bullet}
\drawblackw{1.3, -1.4}{2.7, 0.1}{sloped}{-\,\,\bullet}
\smcap{1.235}{-1}{0.4}{thick}
\smcup{1.855}{-0.6}{0.4}{thick}
\drawblackw{1.3, 1.7}{2.7, 0.3}{sloped, near start}{S}
\drawblackw{1.3, 1.8}{2.7, 1.8}{near end}{-S}
\drawblackw{1.3,0.3}{2.7,1.7}{sloped, near start}{\id}
\drawblackw{1.3, 0.2}{2.7, 0.2}{near start}{-\id}
\drawgray{3.7,0.2}{5.1,0.2}
\drawgray{3.7,0.3}{5.1,1.7}
\drawgray{3.7,1.7}{5.1,0.3}
\drawblack{3.7, 3.3}{5.1,3.3}{}{0}
\drawblack{3.7,3.2}{5.1,1.9}{sloped}{-S}
\drawgray{3.7,3.1}{5.1, 0.4}
\drawblack{3.7,1.8}{5.1,1.8}{}{d}
\drawblack{6.1,3.2}{7.5,1.9}{sloped}{\id}
\drawblack{6.1,1.8}{7.5,1.8}{}{-S}
\drawblack{6.1, 1.7}{7.5,0.3}{sloped, near start}{\bullet \,\, - h}
\drawblack{6.1, 1.6}{7.5, -1.3}{sloped, near end}{\id}
\drawblackw{6.1,0.3}{7.5,1.7}{sloped, very near start}{S}
\drawblackw{6.1,0.2}{7.5,0.2}{}{-\,\,\bullet \,\,+h}
\drawblack{6.1, 0.1}{7.5, -1.4}{sloped}{-\id}
\smcap{6.275}{1.5}{0.4}{thick}
\smcup{6.625}{0.235}{0.4}{thick}
\drawgray{8.5, 1.7}{9.9, 1.1}
\drawgray{8.5, 1.6}{10, -0.5}
\drawgray{8.5, -1.4}{10, -0.7}
\drawblack{8.5, 0.2}{9.9, 0.9}{sloped}{d}
\drawgray{10.8, -0.5}{11.4, 0}
\drawgray{10.9,1.1}{11.5,1.6}
\drawgray{-1,3.4}{0,4}
\drawgray{1.4,3.4}{2.4,4}
\drawgray{3.8,3.4}{4.8,4}
\drawgray{6.2,3.4}{7.2,4}
\drawgray{9.3,2.3}{9.9, 1.2}
\drawgray{9.4, 0.8}{10, -0.3}
\end{tikzpicture}
\]
We now perform Gaussian eliminations as in Section \ref{sec:toruslinks} (and, in particular, do not cancel the $\id$-morphism between the $q^{s-1}\tomega$ objects yet), resulting in the following complex.
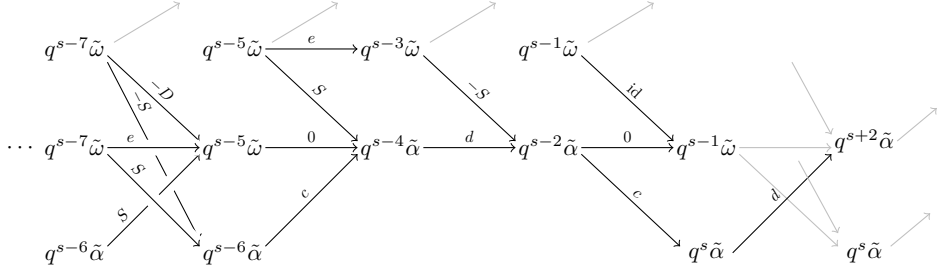
\[
\begin{tikzpicture}[scale = 0.875, transform shape]
\node at (-2.4, 1.8) {$\cdots$};
\node at (0.8, 1.8) {$q^{s-5}\tomega$};
\node at (-1.6, 1.8) {$q^{s-7}\tomega$};
\node at (-1.6, 0.2) {$q^{s-6}\talpha$};
\node at (3.2, 1.8) {$q^{s-4}\talpha$};
\node at (5.6, 1.8) {$q^{s-2}\talpha$};
\node at (8, 1.8) {$q^{s-1}\tomega$};
\node at (0.8, 0.2) {$q^{s-6}\talpha$};
\node at (8, 0.2) {$q^{s}\talpha$};
\node at (10.4,1.9) {$q^{s+2}\talpha$};
\node at (10.4,0.2) {$q^{s}\talpha$};
\node at (5.6, 3.3) {$q^{s-1}\tomega$};
\node at (3.2, 3.3) {$q^{s-3}\tomega$};
\node at (0.8, 3.3) {$q^{s-5}\tomega$};
\node at (-1.6, 3.3) {$q^{s-7}\tomega$};
\drawgray{8.5, 1.8}{9.9, 1.8}
\drawgray{8.5, 1.7}{10, 0.3}
\drawblack{8.4, 0.2}{9.9, 1.7}{sloped}{d}
\drawgray{10.8, 0.3}{11.4, 0.8}
\drawgray{10.9,1.9}{11.5,2.4}
\drawgray{-1,3.4}{0,4}
\drawgray{1.4,3.4}{2.4,4}
\drawgray{3.8,3.4}{4.8,4}
\drawgray{6.2,3.4}{7.2,4}
\drawgray{9.3,3.1}{9.9, 2}
\drawgray{9.4, 1.6}{10, 0.5}
\drawblack{6.1, 3.2}{7.5, 1.9}{sloped}{\id}
\drawblack{6.1,1.8}{7.5, 1.8}{}{0}
\drawblack{6.1, 1.7}{7.6,0.3}{sloped}{c}
\drawblack{3.7, 1.8}{5.1, 1.8}{}{d}
\drawblack{3.7, 3.2}{5.1, 1.9}{sloped}{-S}
\drawblack{1.3, 3.3}{2.7, 3.3}{}{e}
\drawblack{1.3, 3.2}{2.7, 1.9}{sloped}{S}
\drawblack{1.3, 1.8}{2.7, 1.8}{}{0}
\drawblack{1.3, 0.3}{2.7, 1.7}{sloped}{c}
\drawblack{-1.1, 3.2}{0.3, 1.9}{sloped}{-D}
\drawblack{-1.1, 3.1}{0.3, 0.4}{sloped, near start}{-S}
\drawblackw{-1.1, 0.3}{0.3, 1.7}{sloped, near start}{S}
\drawblackw{-1.1, 1.8}{0.3,1.8}{near start}{e}
\drawblackw{-1.1, 1.7}{0.3, 0.3}{sloped, near start}{S}
\end{tikzpicture}
\]
We can now perform the Gaussian elimination on the $\id$ between the $q^{s-1}\tomega$, as this is not going to produce new morphisms. Applying the functor $G$ leads to the complex in Figure \ref{fig:almostdone}.
\begin{figure}[ht]
\begin{tikzpicture}[scale = 0.875, transform shape]
\node at (-1,3.5) {$\cdots$};
\node at (-1,1) {$\cdots$};
\node at (0,0) {$q^{s-8}A$};
\node at (0,1) {$q^{s-6}A$};
\node at (0,2) {$q^{s-6}A$};
\node at (0,3) {$q^{s-8}A$};
\node at (0,4) {$q^{s-6}A$};
\node at (2.4,0) {$q^{s-6}A$};
\node at (2.4,1) {$q^{s-4}A$};
\node at (2.4,2) {$q^{s-6}A$};
\node at (2.4,3) {$q^{s-6}A$};
\node at (2.4,4) {$q^{s-4}A$};
\node at (4.8,2) {$q^{s-4}A$};
\node at (4.8,3) {$q^{s-4}A$};
\node at (4.8,4) {$q^{s-2}A$};
\node at (7.2,2) {$q^{s-2}A$};
\node at (9.6,2) {$q^sA$};
\node at (12,2) {$q^{s+2}A$};
\node at (12,3) {$q^sA$};
\draw[dashed] (-1,2.5) -- (13,2.5);
\drawgray{0.6,3.9}{1.8,2.2}
\drawgray{0.6,3.8}{1.8, 1.1}
\drawgray{0.6,3.7}{1.8,0.2}
\drawgray{0.6,2.9}{1.8,2.1}
\drawgray{0.6,2.8}{1.8,0.1}
\drawgray{3,3.9}{4.2,2.2}
\drawgray{3,2.9}{4.2,2.1}
\drawgray{5.4,3.9}{6.6,2.2}
\drawgray{5.4,2.9}{6.6,2.1}
\drawblack{3,4}{4.2,4}{}{2X-h}
\drawblack{3,3}{4.2,3}{}{2X-h}
\drawgray{0.6,4.2}{1.1,4.5}
\drawgray{0.6,3.2}{1.1,3.5}
\drawgray{3,4.2}{3.5,4.5}
\drawgray{3,3.2}{3.5,3.5}
\drawgray{5.4,4.2}{5.9,4.5}
\drawgray{5.4,3.2}{5.9,3.5}
\drawgray{11.1,3.5}{11.6,3.2}
\drawgray{12.4,3.2}{12.9,3.5}
\drawgray{10.2,3.9}{11.4,2.2}
\drawblack{10,2}{11.4,2}{}{2X-h}
\drawblack{5.4,2}{6.6,2}{}{2X-h}
\drawblack{0.6,0}{1.8,0}{}{2X-h}
\drawblack{0.6,1.9}{1.8,1.1}{sloped, very near start}{X-h}
\drawblackw{0.6,0.2}{1.8,1.8}{sloped, very near start}{X}
\drawblackw{0.6,1.8}{1.8,0.2}{sloped, very near end}{\id}
\drawblackw{0.6,1.1}{1.8,1.9}{sloped, near end}{\id}
\drawblackw{0.6,1}{1.8,1}{near start, below}{2X-h}
\end{tikzpicture}
\caption{\label{fig:almostdone}The complex after applying $G$.}
\end{figure}

The objects below the dashed line form a subcomplex, while the objects above the line give rise to a quotient complex which is exactly $u^{4k}q^{12k}C_w$ from Proposition \ref{prp:altBarNatan}, compare Figure \ref{fig:theC_w}. The gray arrows from $u^{4k}q^{12k}C_w$ to the subcomplex below the dashed line have been worked out before, see Figures \ref{fig:twoxminush} and \ref{fig:htwodiag}.

We can still perform the Gaussian eliminations that were done to Figure \ref{fig:htwodiag} leading to Figure \ref{fig:htworeduced}. For $i\geq 1$, that is, the part of the subcomplex that is not visible, we can do the same change of basis moves (and which are not going to change the quotient complex $u^{4k}q^{12k}C_w$) that were done to Figure \ref{fig:htworeduced}, to split off the complex $\check{B}_{3k}$. The visible three sub-complexes in Figure \ref{fig:almostdone} give rise to $T$. Some of the arrows from $u^{4k}q^{12k}C_w$ to $T$ contain identities, but we do not need to do further cancellations, as we are only interested in a short exact sequence, which we now have after splitting off $\check{B}_{3k}$.
\end{proof}

\begin{corollary}\label{cor:properaltcase}
Let $w$ be a proper alternating word, $k\geq 1$ and $L$ the link closure associated to the braid word $(ab)^{3k}w$ with basepoint on the middle strand. Then
\[
\rKh^{i,j}(L;\Z) \cong \rKh^{i,j-t}(T(3,3k);\Z)\oplus \rKh^{i-4k,j-12k}(L_w;\Z),
\]
where $t= m(w)-n(w)$, except in bidegrees $(4k,12k-2+t)$, $(4k, 12k+t)$ and $(4k+1,12k+2+t)$, where
\begin{align*}
\rKh^{4k, 12k-2+t}(L;\Z) &\cong \, 0, \\
\rKh^{4k,12k+t}(L;\Z) &\cong \, \rKh^{0,t}(L_w;\Z), \\
\rKh^{4k+1,12k+2+t}(L;\Z) &\cong \, \rKh^{1,2+t}(L_w;\Z) \oplus \Z.
\end{align*}
Furthermore, all pages of the reduced integral BLT-spectral sequence are free abelian. For $k=1$ the spectral sequence collapses at the $E_2$-page and for $k\geq 2$ it collapses at the $E_3$-page.
\end{corollary}

\begin{proof}
By Theorem \ref{thm:decoomega6} we have
\[
\rCKh(L;\Z) \cong \check{B}_{3k}\otimes_A\Z \oplus D_w\otimes_A\Z,
\]
where $X,h\in A$ act on $\Z$ as $0$. From Theorem \ref{thm:khovhomtorus} we get that $\check{B}_{3k}\otimes_A\Z$ calculates $\rKh^{i,j-t}(T(3,3k);\Z)$ for $i \leq 4k-5$. Notice that this is all the homology this complex has.

We have $H^{i,j}(T\otimes_A\Z)\cong \Z$ in bidegrees
\[
(4k-4,s-8), (4k-2,s-4), (4k-1,s-2), (4k,s), (4k+1,s+2),
\]
and also in $(4k-3,s-4)$ if $k\geq 2$. In all other bidegrees the homology is $0$. Here $s=12k+t$. In particular, we have $H^{i,j}(T\otimes_A\Z)\cong \rKh^{i,j-t}(T(3,3k);\Z)$ for $i=4k-4,\ldots,4k-1$.

Furthermore, $H^{4k+i,s+2i}(u^{4k}q^{12k}C_w\otimes_A\Z)$ is free abelian by Proposition \ref{prp:altBarNatan} and Corollary \ref{cor:redKhovalt}, and in all other bidegrees the homology vanishes. We therefore get short exact sequences
\[
0 \longrightarrow \Z \longrightarrow H^{4k+i,s+2i}(D_w\otimes_A\Z) \longrightarrow H^{i,t+2i}(C_w\otimes_A\Z) \longrightarrow 0,
\]
for $i = -4, -2, -1, 0, 1$, as well as
\[
0\longrightarrow \Z \longrightarrow H^{4k-3,s-4}(D_w\otimes_A\Z) \longrightarrow 0
\]
and $H^{i, j}(D_w\otimes_A\Z) \cong H^{i-4k,j-12k}(C_w\otimes_A\Z)$ in all other bidegrees. By Proposition \ref{prp:altBarNatan} $H^{i,j}(C_w\otimes_A\Z) \cong \rKh^{i,j}(L_w;\Z)$, except in bidegree $(0,t)$, where one copy of $\Z$ is missing. However, $H^{4k,s}(T\otimes_A\Z)$ provides us with a copy of $\Z$. Finally, $H^{4k+1,s+2}(T\otimes_A\Z)$ gives the extra copy of $\Z$ in bidegree $(4k+1,s+2)$.

Observe that in homological degrees $4k-4, \ldots, 4k-1$ the complex $D_w\otimes_A\Z$ has one extra copy of $\Z$ compared to $C_w\otimes_A\Z$. But these agree with the homology of $T(3,3k)$ in these degrees. However, in homological degree $4k$, the torus link has $3$ copies of $\Z$ which are not produced by $\check{B}_{3k}\otimes_A\Z$. In particular, in bidegree $(4k, 12k-2+t)$ is $0$.

This shows that all the reduced Khovanov homology groups of $L$ are free abelian and as stated. For the spectral sequence statement note that $\check{B}_{3k}$ gives rise to a spectral sequence with the required properties (it also collapses at the $E_2$-page for $k=1,2$). 

For the spectral sequence coming from $D_w$ note that all the non-zero groups of the $E_1$-page are on the diagonal $(i,t+4k+2i)$, except for a single copy of $\Z$ in bidegree $(4k-3,s-4)$ if $k\geq 2$. It then follows that $E_2^{4k-3,s-4}\cong \Z$, and since it is in an odd homological degree it cannot survive to the $E_\infty$-page. Similarly, no $E_2^{i,t+4k+2i}$ can contain torsion, as this would survive all the way to $E_\infty$. Note that if $E_\infty^{i,t+4k+2i}$ had torsion, we would need a non-zero $E_\infty^{i,t+4k+2i+2l}$ for some $l\geq 1$, which is not possible by the form of $E_1$. This shows the spectral sequence statement.
\end{proof}

\begin{proof}[Proof of Theorem \ref{thm:maintorsion} and \ref{thm:main_knight}]
Let $w\in \Omega^+_0\cup \Omega^+_1\cup \Omega^+_2\cup\Omega^+_3$. Then Theorems \ref{thm:maintorsion} and \ref{thm:main_knight} hold by Theorem \ref{thm:khovhomtorus} and Corollary \ref{cor:finalstep}. They also hold for the mirror links since Khovanov homology and Bar-Natan homology satisfy duality \cite{MR2232858}. The remaining cases for $w\in \Omega_0\cup \Omega_1\cup \Omega_2\cup\Omega_3$ involve $k=0$ and are either covered by Theorem \ref{thm:khovhomtorus} or trivial.

The cases $w\in \Omega_4\cup \Omega_5$ are handled similarly, using Theorem \ref{thm:decoomega5} and Theorem \ref{thm:decoomega4} together with Corollary \ref{cor:finalstep}. Finally, for $w\in \Omega_6$ we use Corollary \ref{cor:properaltcase} with Proposition \ref{prp:concludeconjectures}.
\end{proof}

\bibliography{KnotHomology}
\bibliographystyle{amsalpha}

\end{document}